\documentclass{article}

\usepackage{amsfonts,mathrsfs}
\usepackage[english]{babel}
\usepackage{graphicx,multirow}
\usepackage{indentfirst}
\usepackage[colorlinks=true,
linkcolor=blue,
urlcolor=blue,
citecolor=red]{hyperref}
\usepackage{amsmath, amsthm, amssymb}
\usepackage{hyperref}
\usepackage{subcaption}
\usepackage{pdfpages}
\usepackage{lipsum}
\usepackage{hyperref}
\usepackage{colortbl}
\usepackage{bbm}
\usepackage{cite}
\usepackage[resetlabels,labeled]{multibib}
\usepackage{epsf,epsfig}
\usepackage{multirow}
\usepackage{caption}
\usepackage{bbm, mathrsfs}

\usepackage[utf8]{inputenc}    
\usepackage[T1]{fontenc}

\usepackage{float,color}
\newfloat{figure}{H}{lof}
\newfloat{table}{H}{lot}
\floatname{figure}{\figurename}
\floatname{table}{\tablename}

\usepackage{tikz,lipsum,lmodern}
\usepackage[most]{tcolorbox}

\newtheorem{theorem}{Theorem}[section]

\newtheorem{definition}[theorem]{Definition}

\newtheorem{lemma}[theorem]{Lemma}

\newtheorem{proposition}[theorem]{Proposition}
\newtheorem{remark}[theorem]{Remark}

\numberwithin{equation}{section}

\renewcommand{\d}{{\rm d}}
\newcommand {\CI}{{\rm CI}}

\newcommand {\CR}{{\rm CR}}
\newcommand {\CU}{{\rm CU}}

\renewcommand {\Re}{{\rm Re \,}}

\newcommand{\R}{\mathbb{R}}
\newcommand{\N}{\mathbb{N}}

\definecolor{mycolor}{RGB}{0, 204, 204}

\usepackage{tcolorbox}
\newtcolorbox{mybox}[1]{colback=red!5!white,colframe=red!75!black,fonttitle=\bfseries,title=#1}
\begin{document}
	
	\title{\textbf{Data-Driven Mathematical Modeling Approaches for COVID-19: a survey}}
\author{\textsc{J. Demongeot$^{(a)}$, and   P. Magal$^{(b),}$\thanks{Corresponding author. e-mail: \href{mailto:pierre.magal@u-bordeaux.fr}{pierre.magal@u-bordeaux.fr}} }\\
	{\small \textit{$^{(a)}$Université Grenoble Alpes, AGEIS EA7407, F-38700 La Tronche, France}} \\
	{\small \textit{$^{(b)}$Univ. Bordeaux, IMB, UMR 5251, F-33400 Talence, France.}} \\
	{\small \textit{CNRS, IMB, UMR 5251, F-33400 Talence, France.}}%\\
%	{\small \textit{$^{(c)}$School of Public Health and Graduate School of Medicine, Kyoto University, Kyoto, Japan.}} 
}
\maketitle
\begin{abstract}
In this review, we successively present the methods for phenomenological modeling of the evolution of reported and unreported cases of COVID-19, both in the exponential phase of growth and then in a complete epidemic wave. After the case of an isolated wave, we present the modeling of several successive waves separated by endemic stationary periods. Then, we treat the case of multi-compartmental models without or with age structure. Eventually, we review the literature, based on 230 articles selected in 11 sections, ranging from the medical survey of hospital cases to forecasting the dynamics of new cases in the general population.
\end{abstract}


	\vspace{0.5cm}
	\noindent \textit{I simply wish that, in a matter which so closely concerns the well-being of mankind, no decision shall be made without all the knowledge which a little analysis and calculation can provide}, Daniel Bernoulli 1765.
	
	\vspace{0.5cm}

	\tableofcontents

	\section{Introduction} 
	\label{Section1}
	The COVID-19 outbreak has been the catalyst for increased scientific activity, particularly in data collection and modeling the dynamics of new cases and deaths due to the outbreak.
	
	\medskip 
	Such scientific excitement contemporary with a pandemic is not new. Several historical epidemic episodes have led to significant advances in public health, biostatistics, databases, and discrete or continuous mathematical modeling of disease evolution, considering the mechanisms of contagion, host resistance, and mutation of the infectious agent. Historically, we can thus distinguish several epidemic outbreaks followed by important scientific breakdowns:
	\begin{itemize}
		\item [A)]	 The plague epidemic of 1348 saw the development of the beginnings of epidemiology with the recording of cases at the abbey of St Antoine (Isère in France) and in the network of hospitals managed by the Antonin monks;
		
		\medskip 
		\item[B)]	During the London cholera epidemic of 1654, John Snow discovered the waterborne transmission of cholera, which led to significant changes to improve public health, notably by constructing improved sanitation facilities. This epidemic and its resolution by Snow even before the discovery of the responsible germ was a founding event in intervention epidemiology, with the validation of methods that can be applied to all diseases, not just contagious (infectious or social), in particular the principle of coupling the mapping of patients with that of sources of water for domestic consumption, which would later lead to the development of Geographic Information Systems (GIS) in epidemiology and to work such as the collection of water used as a COVID-19 tracer in the French Obépine project (\url{https://www.reseau-obepine.fr/});
		
		\medskip 
		\item[C)] The smallpox epidemic of 1760 led to the importation into Europe of the inoculation practiced in Turkey (subsequently leading to vaccination by inert vaccine by Jenner) and to the creation of the first models for predicting epidemic waves by Bernoulli and d'Alembert.
	\end{itemize}
	
	\medskip 
	In the tradition of these past discoveries, we will therefore present some recent progress in modeling the dynamics of infectious diseases and their transmission mechanisms in this article.
	
	\medskip 
	The plan of the paper is the following. Section \ref{Section2} presents some background about the reported data. We explain some phenomena related to data collection, such as contact tracing, daily numbers of tests, and more.   In section \ref{Section3}, we explain the main idea behind the notion of a phenomenological model. In section \ref{Section4}, we introduce an epidemic model with unreported cases and explain how to compare such a model with the data. In section \ref{Section5}, we consider the exponential phase of an epidemic, where the phenomenological model will be an exponential function.   In section \ref{Section6}, we consider a single epidemic wave, where the phenomenological model will be the Boulli-Verhulst model. We consider several successive epidemic waves in section \ref{Section7}. In section \ref{Section8}, we present some new results to understand how to compare the data and the epidemic models with several compartments during the exponential phase.  In section \ref{Section9}, we consider a model with age groups and explain how to deal with data in large systems. Section \ref{Section10} is a survey section where we try to give some references for a selected number of important topics to model epidemic outbreaks.

	\section{Reported and unreported data}
	\label{Section2}
	\subsection{What are the unreported cases?}
	\label{Section2.1}
	The unreported cases correspond to mild symptoms because people will only get tested in case of severe symptoms. Unreported cases can result from a lack of tests or asymptotic patients \cite{Oran20}. That is infected patients that do not show symptoms. Unreported cases are partly due to a low daily number of tests.
	\subsection{Example of unreported cases}
	\label{Section2.2}

	A published study	traced COVID-19 infections resulting from a business meeting in Germany attended by a person who was infected but had no symptoms at the time \cite{rothe2020transmission}. Four people were eventually infected from this single contact.
	
	\begin{figure}[H]
		\centering
		\includegraphics[scale=.25]{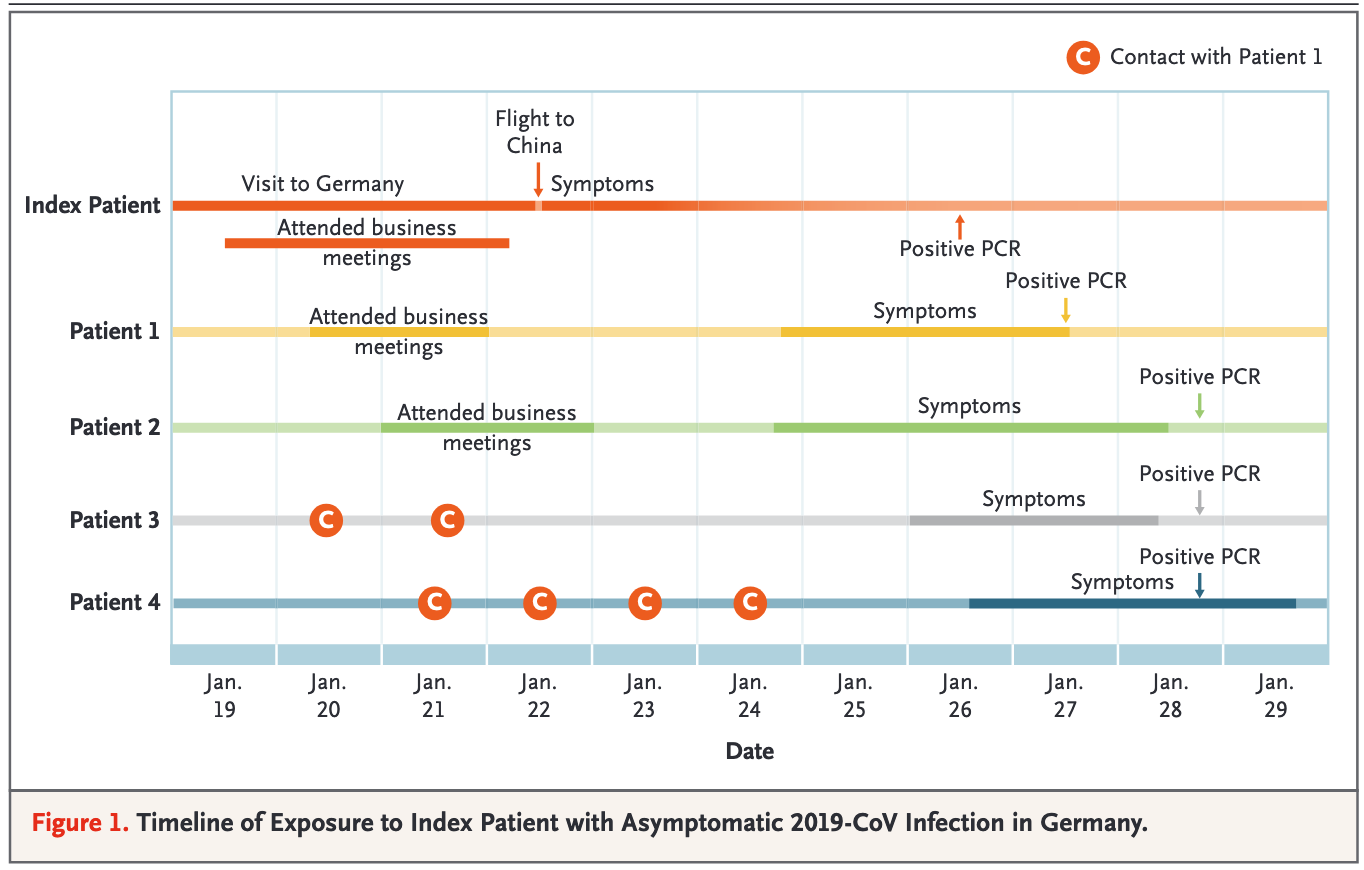}
		\caption{\textit{ Timeline of Exposure to Index Patient with Asymptomatic 2019-CoV Infection in Germany.}}
		\label{Fig1}
	\end{figure}

	\medskip A team in Japan \cite{nishiura2020serial} reports that 13 people evacuated from \textit{Diamond Princess} were infected, 4 of whom, or 31 $ \% $, never developed symptoms.
	
	\medskip 
	On the French \textit{aircraft carrier Charles de Gaulle}, clinical and biological data for all 1739 crew members were collected on arrival at the Toulon harbor and during quarantine: 1121 crew members (64\%) were tested positive for COVID-19 using RT-PCR, and among these, 24\% were asymptomatic \cite{Bylicki21}.

	\subsection{Testing data for New York state}
	\label{Section2.3}
	
	The goal of the figure below is to show that due to the changes in the method of detecting the cases, a jump occurred on February 12 in Wuhan in, China. The testing technology was not well developed at the early beginning of the epidemic, and such a problem also occurs in other countries.   
	\begin{figure}
		\centering
		\includegraphics[scale=0.15]{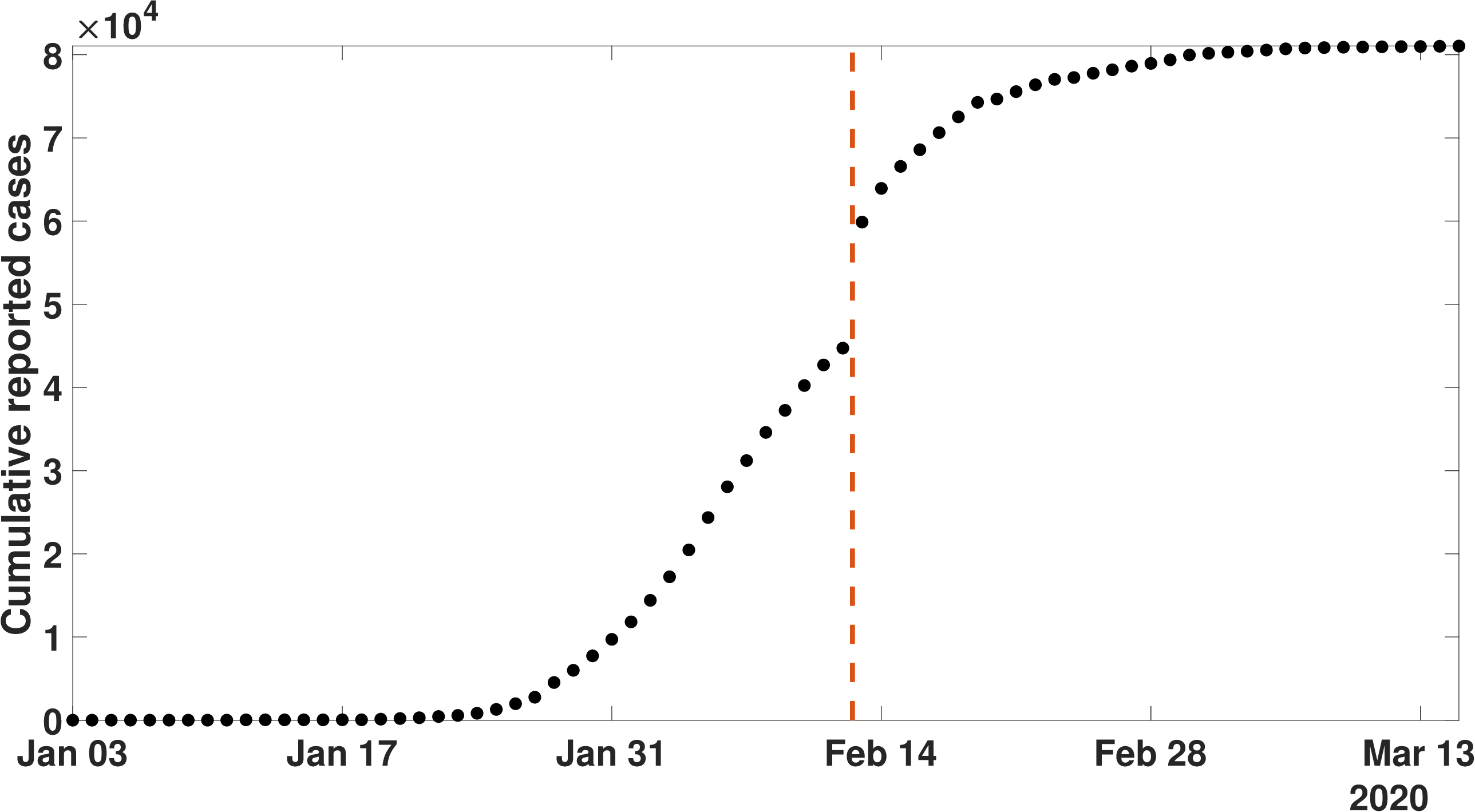} 
		\caption{\textit{Cumulative number of cases in Wuhan China. }}
		\label{Fig2}
	\end{figure}

	The dynamic of the daily number of tests is connected to the dynamic of the daily number of reported cases in a complex way \cite{griette2021clarifying}.  
	
	\medskip 	
	The large peak in the number of tests at the end of April 2020, shows that the number of cases was strongly underestimated during the period. Because increasing the number of tests increases the number of positive test. Later on, the epidemic wave passed and the changes in the number of test had almost no influence on the number of positive test.  
	
	\begin{figure}
		\centering
		\includegraphics[scale=0.2]{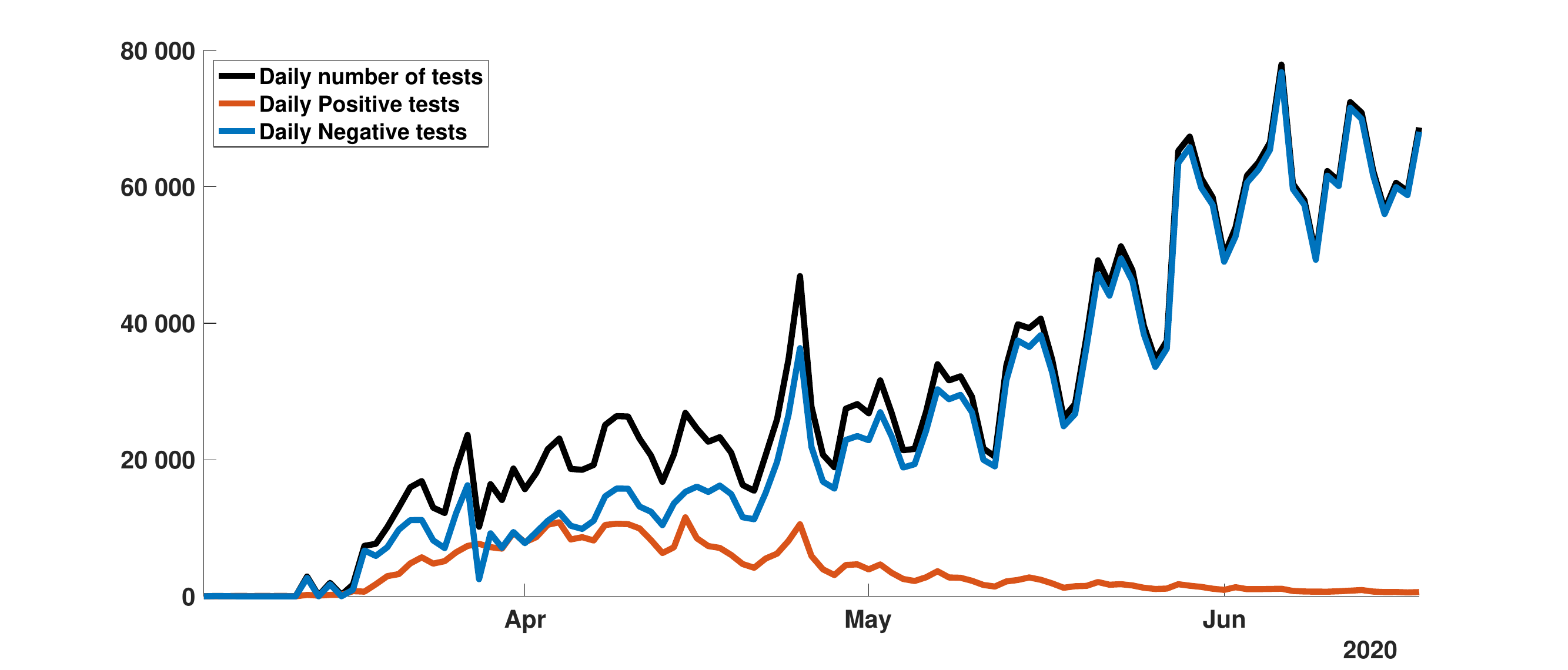}
		\caption{{\textit{ In this figure, we plot the daily number of tests for the New York State. The black curve, orange curve,
					and blue curve correspond respectively to the number of tests, the number of positive tests, and the number of
					negative tests. }}}\label{Fig3}
	\end{figure}

	The number of reported cases is the consequence of the combination of the dynamic of the number of tests (a complex dynamic which depends on human perceptions of the epidemic outbreak), and the dynamic of the epidemic outbreak (which is also very complex due the contact rate which depends on human perceptions)   and the dynamic of transmission (which can also be complex due to the changes of susceptibility in the population). 
	
	\medskip 
	Figure \ref{Fig4} presents the flowchart of the model used in \cite{griette2021clarifying}. In Figure \ref{Fig5} (which was obtained in \cite{griette2021clarifying}), we use the daily number of tests as an input of the model, and we fit the output of the model to the cumulative number of cases. 
	\begin{figure}
		\begin{center}
			\includegraphics[scale=0.6]{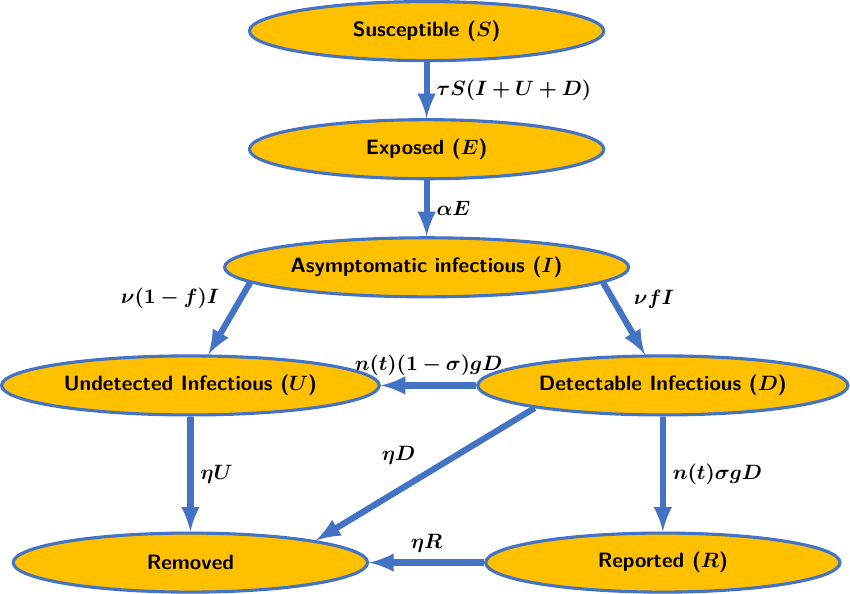}
		\end{center}
		\caption{\textit{Flow chart of the epidemic model. In this diagram $n(t)$  is the daily number of tests at time $ t$ is an input of the model. We consider a fraction $(1-\sigma)$  of false negative tests  and a fraction $\sigma$ of true positive tests.  The parameter $g$ reflects the fact that the tests are devoted not only to the symptomatic patients but also to a large fraction of the population of New York state. }}
		\label{Fig4}
	\end{figure}

	\begin{figure}
		\centering
		\hspace{1cm}	\textbf{Daily}\hspace{4cm}\textbf{Cumulative} \vspace{0cm}\\
		\includegraphics[scale=0.35]{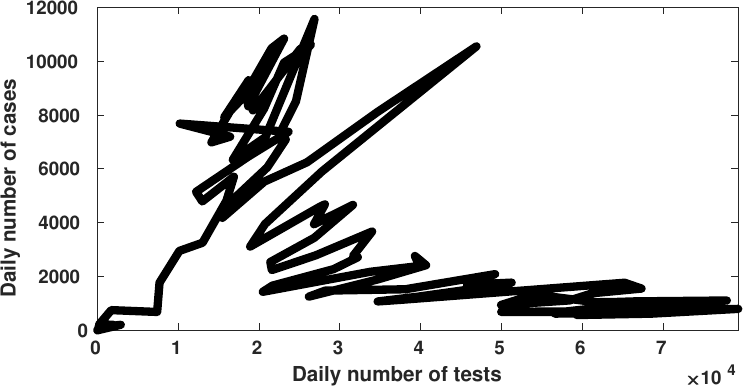} \hspace{0.4cm}    \includegraphics[scale=0.35]{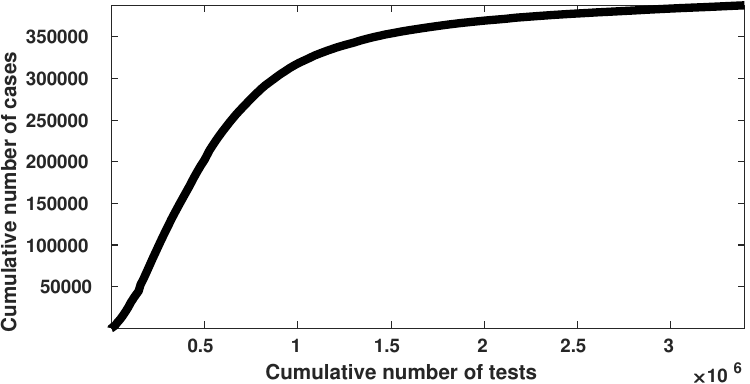} \\
		\includegraphics[scale=0.35]{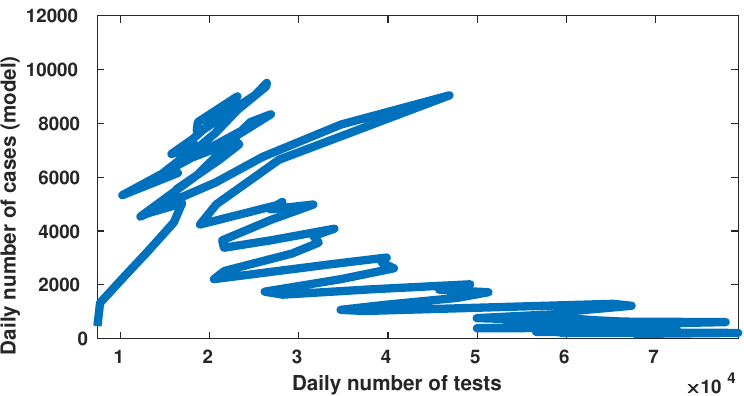} \hspace{0.4cm}    \includegraphics[scale=0.35]{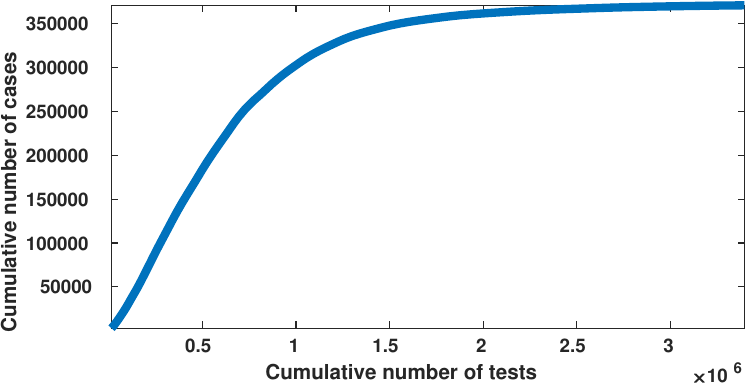} 
		\caption{{\textit{ The black curves are produced by using the New York state data only. The blue curves are constructed by using the model with the testing data as input of the model.}}}\label{Fig5}
	\end{figure}

	In Figure \ref{Fig5}, on the left-hand side, we consider the daily fluctuations of the number of reported cases (epidemic dynamic) and the daily number of tests (testing dynamics). Combining test dynamics and infection dynamics results in a complex time-parameterized curve. Nevertheless, we obtain a good correspondence between the top and the bottom left figures. The correspondence becomes excellent on the figures on the right, where we consider the cumulative number of declared cases and the cumulative number of tests.

	\section{Phenomenological models}
	\label{Section3}
	Along this note, we use phenomenological models to fit the data. 	
	
	\begin{definition}
		A phenomenological model is a mathematical model used to describe the data without mechanistic description of the processes involved in the phenomenon. 
	\end{definition} 
	
	In the next section, we will use exponential functions to get a continuous time representation of the data. This will be our first example of a phenomenological model. Our goal here is to replace the data by a function that captures the robust tendency of the phenomenon.  In some sense, we are trying to get rid of the noise around the tendency. 
	
	\medskip 
	By using, for example, spline functions, we can always fit the data perfectly. Then the fit is too precise to capture the significant information, and if we compute the derivatives of such a perfect fit,  we will obtain a very noisy signal that is not meaningful.

	\medskip 
	Therefore the underlying idea of the phenomenological model is to derive a robust tendency with a limited number of parameters that will represent the data. 
	Such a model is supposed to reduce the signal's noisy part and capture the robust part of the signal. 
	
	\medskip 
	The phenomenological model can then replace the data, permitting analysis of some consequences when injected into the models. For example, we will obtain a meaningful range of parameters. 
	
	\begin{figure}[H]
		\begin{center}
			\includegraphics[scale=0.5]{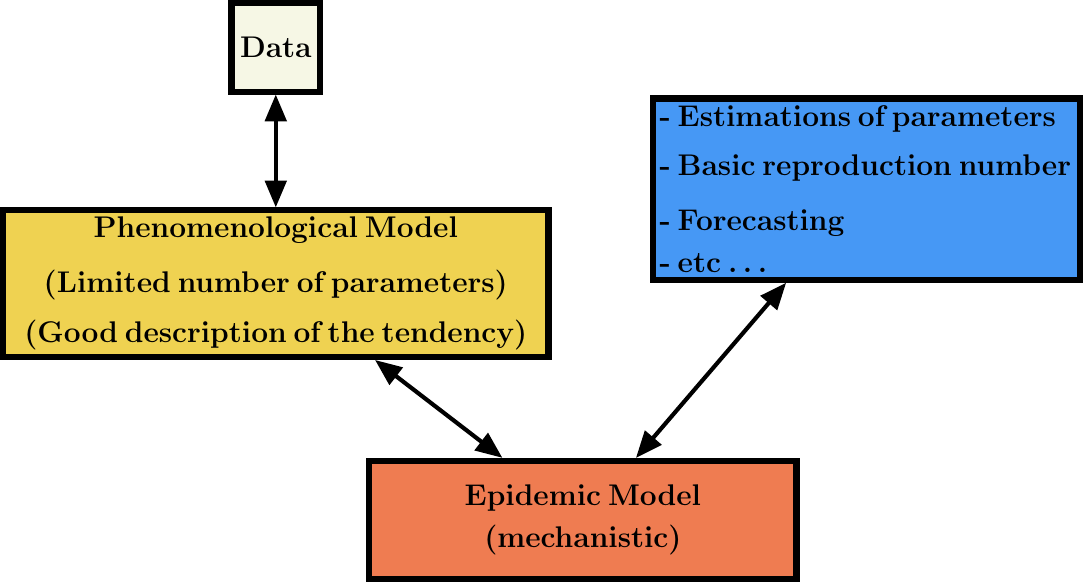} 
			\caption{{\textit{We can apply statistical methods to estimate the parameters of the proposed
						phenomenological model and derive their average values with some confidence intervals.
						The phenomenological model is used at the first step of the modelling process, providing
						regularized data to the epidemic model and allowing the identification of its parameters.}}}\label{Fig6}
		\end{center}
		
	\end{figure}

	\section{Epidemic model with reported and unreported individuals}
	\label{Section4}

	\subsection{Mathematical model}
	\label{Section4.1}
	Transmissions between infectious and susceptible individuals are described by
	\begin{equation}\label{4.1}
		\tcbhighmath[boxrule=2pt,drop fuzzy shadow=blue]{\begin{cases}
				S'(t)=-\tau(t)\, S(t)\, I(t),\\
				I'(t)=\tau(t)\, S(t)\,I(t)- \nu \, I(t),	
		\end{cases}}
	\end{equation}
	where $S(t)$ is the number of susceptible and $I(t)$ the number of infectious at time $t$. 
	
	\medskip 
	The system \eqref{4.1} is complemented with the initial data
	\begin{equation}\label{4.2}
		\tcbhighmath[boxrule=2pt,drop fuzzy shadow=blue]{	S(t_0)=S_0 \geq 0,\text{ and } I(t_0)=I_0 \geq 0,}
	\end{equation}
	where $t_0$ is a time from which the epidemic model \eqref{4.1} becomes applicable.

	\medskip 
	In this model, the rate of transmission $\tau(t)$ combines the number of contacts per unit of time and the probability of transmission (see Section \ref{Section6.1} for more information).

	\medskip 
	The number $1/\nu$ is the average duration of the asymptomatic infectious period,  $\tau(t)\, S(t)\,I(t)$ is the flow of $S$-individuals becoming $I$-infected at time $t$.  That is, 
	$$
	\int_{t_1}^{t_2} \tau(\sigma)\, S(\sigma)\,I(\sigma) \d \sigma
	$$
	is the number of individual that became $I$ during the time interval $[t_1,t_2]$.  
	
	\medskip 
	Similarly, $\nu\, I(t)$ is the flow of $I$-individuals leaving  the $I$-compartment. That is 
	$$
	\int_{t_1}^{t_2} \nu\, I(\sigma)\d \sigma
	$$
	is the number of individual that became $I$ during the time interval $[t_1,t_2]$.  
	
	\medskip 
	The epidemic model associated with the flowchart in Figure \ref{Fig7} applies to the Hong Kong flu outbreak in New York City \cite{Magal18b, Ducrot20}. 
	\begin{figure}[H]
		\begin{center}
			
			\includegraphics[scale=0.75]{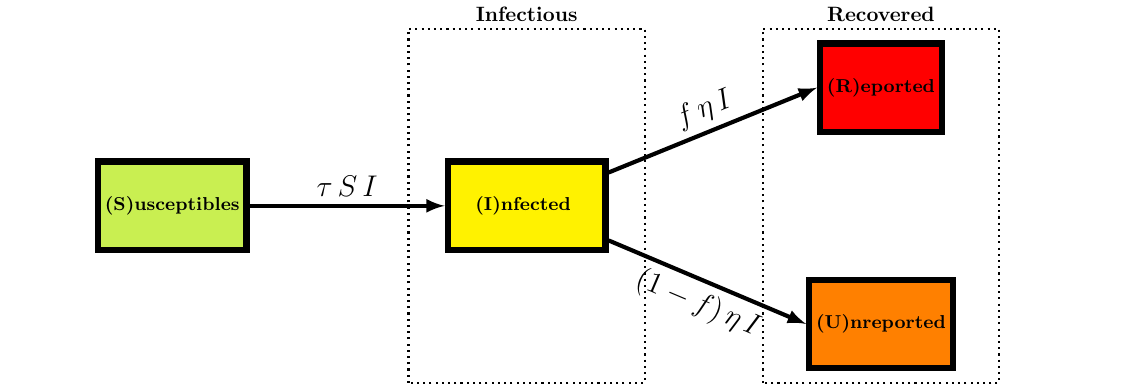}
		\end{center}
		\caption{ \textit{Flowchart.}}\label{Fig7}
	\end{figure}
	We assume that the flow of reported individuals is a fraction $0 \leq f \leq 1$ of the flow of recovered individuals $\nu\, I$. That is, 
	\begin{equation}\label{4.3}
		\tcbhighmath[boxrule=2pt,drop fuzzy shadow=blue]{	\CR'(t)= f \,    \nu \, I(t),   \text{ for }  t \geq t_0,}
	\end{equation}
	where $\CR(t)$ is the cumulative number of reported individuals, and $f$ is the  fraction of reported individuals.  The fraction $f$ is the fraction of patients with severe symptoms, and   $1-f$ the fraction of patients with mild symptoms.

	\subsection{Given Parameters}
	\label{Section4.2}
	In this study,	the following parameters will be given:

	\begin{itemize} 
		
		\item  Number of susceptible individuals when the epidemic starts
		\begin{equation*}
			\tcbhighmath[boxrule=2pt,drop fuzzy shadow=blue]{	S_0=67 \text{ millions for France.}}
		\end{equation*}

		\item  	Time from which the epidemic model starts to be valid, also called initial time of the model 	
		$\tcbhighmath[boxrule=2pt,drop fuzzy shadow=blue]{t_0.}$
		
		\begin{remark}The time $t_0$ is a time where the epidemic phase started already. 
		\end{remark}
		
		\item  The average duration of the infectiousness 			
		$\tcbhighmath[boxrule=2pt,drop fuzzy shadow=blue]{\dfrac{4.1}{\nu}=3 \text{ days}.}$
		\item The fraction of reported individuals
		$\tcbhighmath[boxrule=2pt,drop fuzzy shadow=blue]{f=0.9.}$
		
	\end{itemize}

	\subsection{Computed parameters}
	\label{Section4.3}
	The following parameters will be obtained by comparing the output of the model and the data: 
	\begin{itemize} 
		
		\item  $\tcbhighmath[boxrule=2pt,drop fuzzy shadow=blue]{I_0}$ the number of asymptomatic infectious patients at the start of the epidemic.

		\item  $\tcbhighmath[boxrule=2pt,drop fuzzy shadow=blue]{\tau(t)}$ the rate of transmission.

	\end{itemize}

	\section{Modeling the exponential phase}
	\label{Section5}
	At the early stage of the epidemic, we can assume that $S(t)$ is constant, and equal to $S_0$. We can also assume that $\tau(t)$ remains constant equal to $\tau_0=\tau(t_0)$. Therefore, by replacing these parameters into the I-equation of system \eqref{4.1} we obtain 
	$$
	I'(t)=(\tau_0 S_0 -\nu)I(t).
	$$ 
	Therefore 
	\begin{equation}\label{5.1}
		\tcbhighmath[remember as=fx, boxrule=2pt,drop fuzzy shadow=blue]{	I(t)= I_0 e^{\chi_2 (t-t_0)},}
	\end{equation}
	where 
	\begin{equation} \label{5.2}
		\tcbhighmath[remember as=fx, boxrule=2pt,drop fuzzy shadow=blue]{	\chi_2=\tau_0 S_0 -\nu.}
	\end{equation}

	\subsection{Initial number of infected and transmission rate}
	\label{Section5.1}
	By using \eqref{4.3} and  \eqref{5.1}, we obtain 
	\begin{equation} 
		\label{5.3}
		\tcbhighmath[remember as=fx, boxrule=2pt,drop fuzzy shadow=blue]{		\CR(t)= \chi_1 \left(e^{\chi_2  \left(t-t_0\right)}-1 \right)+\chi_3. }
	\end{equation} 
	We observe that 
	$$
	\CR(t_0)= \chi_3, 
	$$
	then $\chi_3$ is a parameter which must be estimated from the data. 
	
	\medskip 
	By using \eqref{4.3} at $t_0$, we obtain 
	\begin{equation}  
		\label{5.4}
		\tcbhighmath[remember as=fx, boxrule=2pt,drop fuzzy shadow=blue]{		I_0= \dfrac{\CR'(t_0)}{\nu \, f }=\dfrac{\chi_1 \, \chi_2}{\nu \, f },}
	\end{equation}
	and by using \eqref{5.2} 
	\begin{equation*}  
		\tcbhighmath[remember as=fx, boxrule=2pt,drop fuzzy shadow=blue]{		\tau_0=\dfrac{\chi_2+\nu}{S_0}.}
	\end{equation*}
	Note that the above estimations of $I_0$ and $\tau_0$ are robust since we used the data over a period (i.e., not only at $t_0$) to evaluate $\chi_1, \chi_2$.

	\subsection{Application to COVID-19 in mainland China} 
	\label{Section5.2}
	The figures below are taken from \cite{Demongeot20b} (see \cite{Liu20a} for similar results).

	\begin{figure}[H]
		\centering
		\includegraphics[scale=0.2]{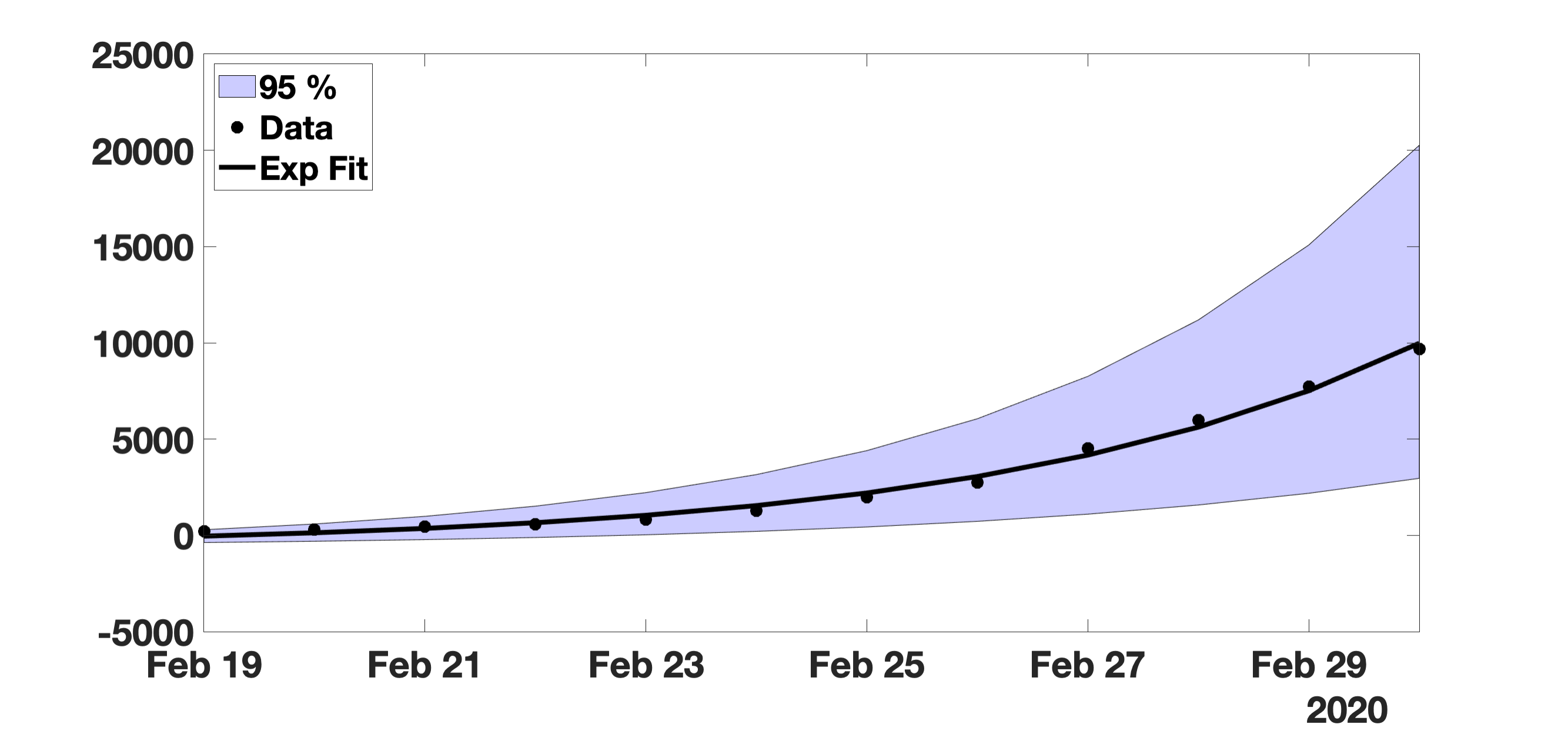}
		\caption{{\textit{In this figure, we plot the best fit of the exponential model to the cumulative number of reported cases of COVID-19 in mainland China between February 19 and March 1. We obtain $\chi_1=3.7366$, $\chi_2=0.2650$  and $\chi_3=615.41$ with $t_0=19$ Feb.  The parameter  $\chi_3$ is obtained by minimizing the error  between the best exponential fit and the data.  }}}\label{Fig8}
	\end{figure}
	
	\begin{remark}	  \label{REM5.1} Fixing $f=0.5$ and $\nu=0.2$, we obtain 
		$$
		I_0=3.7366 \times  0.2650  \times  \exp(0.2650\times 19)/(0.2 \times 0.5)= 1521, 
		$$
		and 
		$$
		\tau_0=\dfrac{0.2650+0.2}{1.4 \times 10^9}  =3.3214 \times 10^{-10}.
		$$
	\end{remark} 
	One may compare Figure \ref{Fig2} with Figure \ref{Fig9} and realize that there is no more jump in  Figure \ref{Fig9}. Here, we canceled out the jump in Figure \ref{Fig2} due to a change of method in counting the number of cases. More precisely, on February $16$, 2020, the cumulative data in Figure \ref{Fig2} jumps by $17409$ cases (the original data are available in \cite[Table 2]{Liu20c}). From that day, public health authorities in China decided to include the patients showing symptoms.
	\begin{figure}[H]
		\centering
		\includegraphics[scale=0.5]{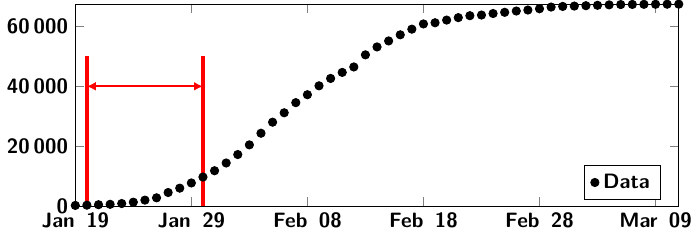}
		\caption{{\textit{In this figure, the black dots represent the cumulative number of cases for China (with correction for the jump presented in Figure \ref{Fig2}). The period marked in red corresponds to the period considered in Figure \ref{Fig8}. }}}\label{Fig9}
	\end{figure}
	\begin{remark}  \label{REM5.2}It is important to understand that, throughout this article,  we fit the cumulative reported data by using a phenomenological. The reason is simple: the cumulative data are much smoother, while the daily number of reported cases are much more fluctuating. Therefore, it is "in theory" much easier to fit the cumulative data with a phenomenological model. Unfortunately, the problem is not that simple. So for example, in the exponential phase, we obtain the parameters  
		$$
		\CR(t)= \chi_1 \left(e^{\chi_2  \left(t-t_0\right)}-1 \right)+\chi_3
		$$
		by using a best fit to the  cumulative number of cases. 
		
		Next, when we compute the first derivative of the above model to the daily number of cases, this gives a pretty reasonable approximation of the daily number of reported cases. 
		
		Another way to avoid the first derivative $t \to \CR(t)$,  is to use the following model
		$$
		D'(t)=f \nu I(t) -D(t).
		$$
		In this model, we use the same input flow of infected as for the model used to compute the cumulative number of cases. But here, we assume that daily cases individuals only stay one day in the $D$ compartment. This model is also equivalent to 
		$$
		D(t)= e^{-(t-t_0)}D_0+ \int_{t_0}^t e^{-(t-\sigma)} f \nu I(\sigma) d \sigma,
		$$
		and by replacing $ f \nu I(\sigma) $ by the cumulative data $\CR(\sigma)$, we obtain a formula for the daily number of cases. 
		
		So, during the exponential phase, once we obtain the best fit of the model to the cumulative data, the daily number of cases is given by 
		$$
		D(t)= e^{-(t-t_0)}D_0+ \int_{t_0}^t e^{-(t-\sigma)}  \chi_1 \left(e^{\chi_2  \left(\sigma-t_0\right)}-1 \right)+\chi_3 d \sigma. 
		$$	
		The model's advantage is that it avoids computing a derive of the cumulative number of cases, which can be an issue. 
	\end{remark}

	\subsection{Spectral method in epidemic time series}
	\label{Section5.3}
	During the COVID-19 pandemic, most people viewed the oscillations around the exponential growth at the beginning of an epidemic wave as the
	default in reporting the data. The residual is probably partly due to the reporting data process (random noise). Nevertheless, a significant remaining part of such oscillations could
	be connected to the infection dynamic at the level of a single average patient. Eventually, the central question we try to address here is: Is there some hidden information in the
	signal around the exponential tendency for COVID-19 data? So we consider the early stage of an epidemic phase, and we try to exploit the oscillations around the tendency in order
	to reconstruct the infection dynamic at the level of a single average patient. We investigate this question in \cite{Demongeot22}. 
	
	The figures below are taken from \cite[see Figures 13 and 14]{Demongeot22}. 
	\begin{figure}[H]
		\centering
		\includegraphics[scale=0.2]{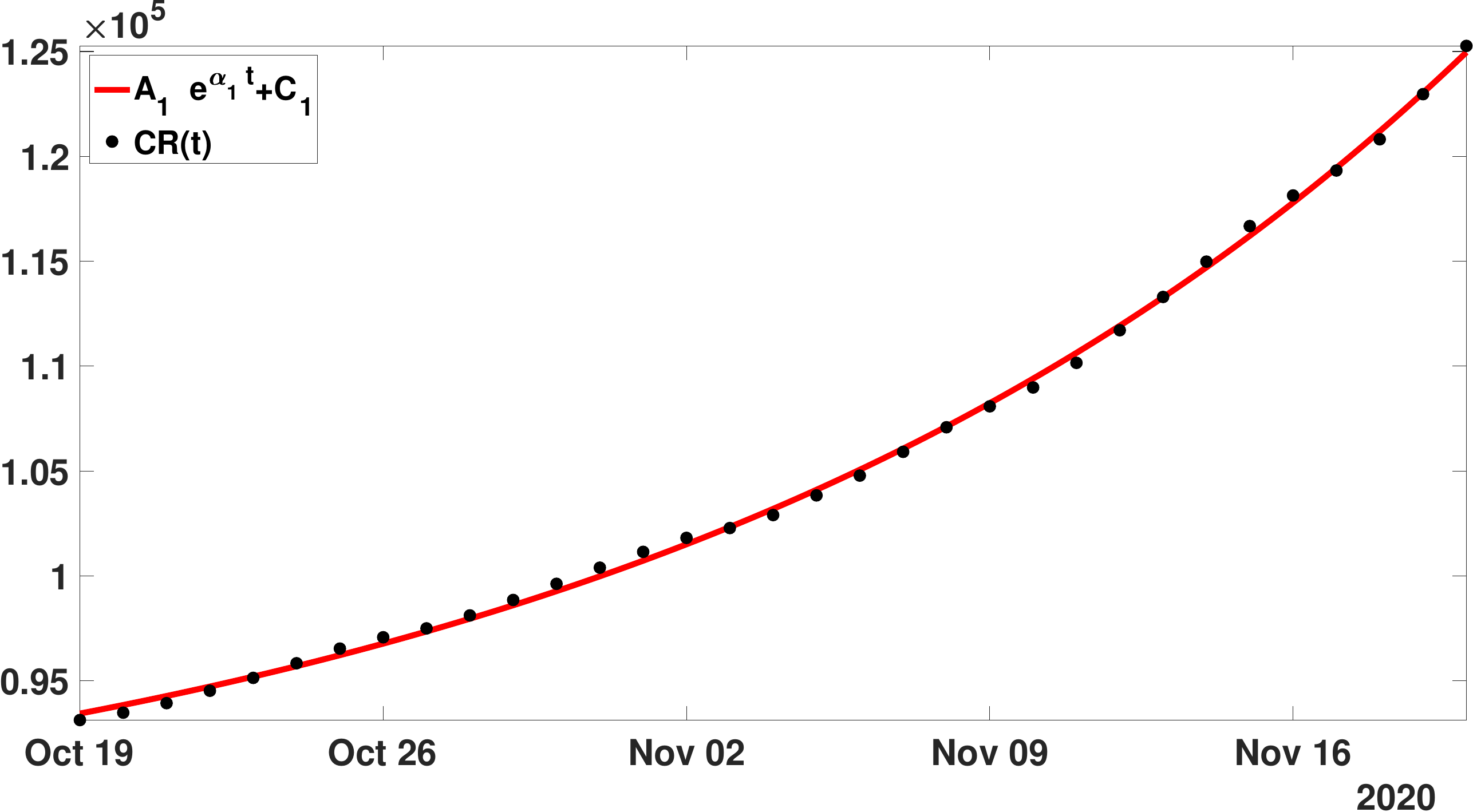}
		\caption{\textit{In this figure, we plot the cumulative number of reported cases  data for Japan  between  $19$ October and  $19$ November $2020$ (black dots). We plot the best fit of the model \eqref{7.1} to the cumulative data (red~curve).  }}\label{Fig10}
	\end{figure}
	Then in the figure below we plot the first residual. That is, 
	$$
	{\rm Residual}_1(t)=\CR(t)- \left[ A_1e^{\alpha_1 t }+C_1\right].
	$$
	\begin{figure}[H]
		\centering
		\includegraphics[scale=0.2]{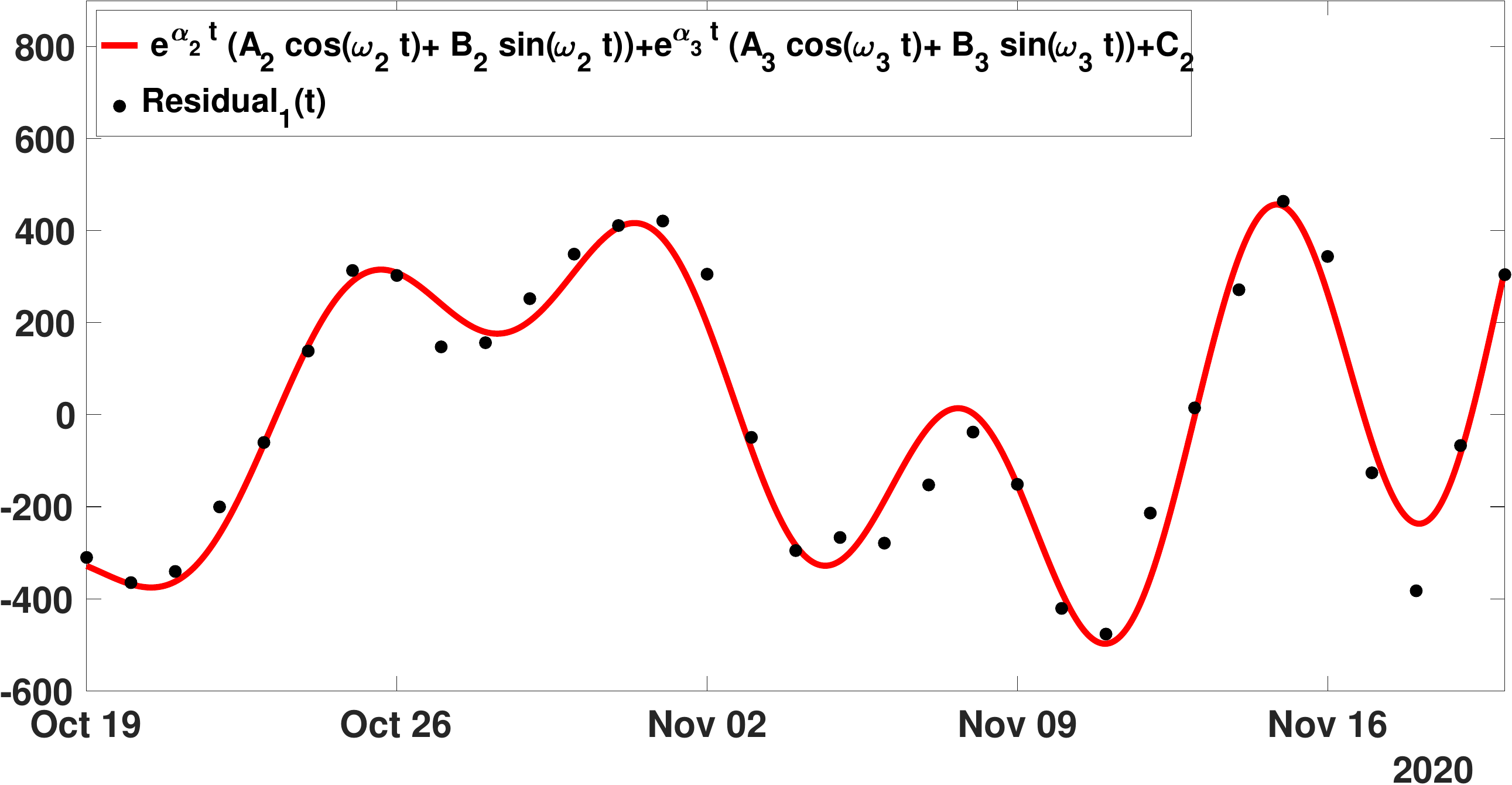}
		\caption{ \textit{ In this figure, we plot the first residual when subtracting  the exponential tendency obtained in Figure \ref{Fig9} to the cumulative reported cases  data between  $19$ October and  $19$ November $2020$ (black dots). We plot the best fit of the model  to the first residual (red curve).  }}\label{Fig11}
	\end{figure}

	\subsection{Monotone property of the cumulative distribution}
	\label{Section5.4}
	
	The influence of the errors made in the estimations (at the early stage of the epidemic) has been considered in the recent article \cite{roda2020difficult}. To understand this problem, let us first consider the case of the rate of transmission $\tau(t)=\tau_0$ in the model \eqref{4.1}.
	
	\medskip 
	\noindent \textbf{From the epidemic model to the data}  Assume that the transmission rate $\tau(t)$ is constant equal to $\tau>0$ in the model \eqref{4.1}.  Then by integrating the $S$-equation in model \eqref{4.1} between $t_0$ and $t$, we obtain
	\begin{equation} \label{5.5}
		S(t)=S_0 e^{\tau  \CI(t)}	
	\end{equation}
	where 
	\begin{equation*} 
		\CI(t)=\int_{t_0}^{t} I(\sigma) d \sigma.
	\end{equation*}
	Moreover 
	$$
	I'(t)=\tau S(t) I(t)- \nu I(t).
	$$
	replacing $S(t)$ by \ref{5.5}, and by integrating between $t_0$ and $t$ we obtain 
	$$
	I(t)=I_0+S_0 \left(1-e^{-\tau \CI(t)}\right)-\nu \CI (t).
	$$
	Remembering that  $\CI(t)'=I(t)$, we conclude that the cumulative number of cases should follow a single ordinary differential equaton 
	\begin{equation} \label{5.6}
		\tcbhighmath[boxrule=2pt,drop fuzzy shadow=blue]{ \CI(t)'= I_0+S_0 \left(1-e^{-\tau \CI(t)}\right)-\nu \CI (t).	}
	\end{equation}
	The system \eqref{5.6} is complemented with the initial distribution  of the model 
	\begin{equation*}
		\tcbhighmath[boxrule=2pt,drop fuzzy shadow=blue]{	\CI(t_0)=\CI_0 \geq 0.}
	\end{equation*}
	This equation should be a good phenomenological model whenever $t \mapsto \tau(t)$ is a constant function.  We refer to \cite{Smith}, and \cite[Chapter 8]{DGLM22}  for a comprehensive presentation on the monotone ordinary differential equations.

	%The equation \eqref{2.7} is indeed a monotone.  We refer to Hal Smith \cite{Smith} for a comprehensive presentation on the monotone systems. By applying a comparison principle to \eqref{2.7},  we are in a position to confirm the intuition about epidemic SI models.  Notice that the monotone properties are only true for the cumulative number of infectious (this is false for the number of infectious). 
	\begin{theorem}  \label{TH5.3} Let $t > t_0$ be fixed. The cumulative number of infectious $\CI(t)$ is strictly increasing with respect to the following quantities
		\begin{itemize}
			\item[{\rm(i)}] $I_0>0$ the initial number of infectious individuals;
			\item[{\rm(ii)}]  $S_0>0$ the initial number of susceptible individuals;
			\item[{\rm(iii)}] $\tau>0$ the transmission rate; 
			\item[{\rm(iv)}] $1/\nu>0$ the average duration of the infectiousness period. 
		\end{itemize}
	\end{theorem}
	
	\begin{mybox}{Error in the estimated initial number of infected and transmission rate}
		Assume that the parameters $\chi_1$ 	and $\chi_2$ are estimated with a $95\%$ confidence interval 
		%Assume that the parameters $\chi_1$ 	and $\chi_2$ are estimated into an interval of confidence of $95\%$  as
		$$
		\chi_{1 ,95\%}^{-} \leq \chi_1 \leq \chi_{1 ,95\%}^{+},
		$$	
		and 
		$$
		\chi_{2 ,95\%}^{-} \leq \chi_2 \leq \chi_{2,95\%}^{+}.
		$$
		We obtain 
		\begin{equation*} \label{2.8}
			I_{0,95\%}^-:=\dfrac{\chi_{1 ,95\%}^{-}  \, \chi_{2 ,95\%}^{-} e^{ \chi_{2 ,95\%}^{-} \, t_0}}{\nu \, f } \leq 	I_0 \leq I_{0,95\%}^+:=\dfrac{\chi_{1 ,95\%}^{+}  \, \chi_{2 ,95\%}^{+} e^{ \chi_{2 ,95\%}^{+} \, t_0}}{\nu \, f },
		\end{equation*}
		and 
		\begin{equation*} \label{2.9}
			\tau_{0, 95\%}^- :=\dfrac{\chi_{2 ,95\%}^{-} +\nu}{S_0}  \leq	\tau_0 \leq \tau_{0, 95\%}^+:=\dfrac{\chi_{2 ,95\%}^{+} +\nu}{S_0} .
		\end{equation*}
	\end{mybox}
	\begin{remark} \label{REM5.4}
		By using the data for mainland China we obtain 
		\begin{equation*}
			\chi_{1 ,95\%}^{-} =1.57, \, \chi_{1 ,95\%}^{+} =5.89, \, \chi_{2 ,95\%}^{-} =0.24, \, \chi_{2 ,95\%}^{+} =0.28.
		\end{equation*}
	\end{remark} 
	In Figure \ref{Fig12}, we plot the upper and lower solutions $\CR^+(t)$ (obtained by using $I_0=I_{0,95\%}^+$ and $\tau_0=\tau_{0, 95\%}^+$) and $\CR^-(t)$  (obtained by using $I_0=I_{0,95\%}^-$ and $\tau_0=\tau_{0, 95\%}^-$) corresponding to the blue region and the black curve corresponds to the best estimated values $I_0=1521$ and $\tau_0=3.3214 \times 10^{-10}$.  
	
	Recall that the final size of the epidemic corresponds to the positive equilibrium of \eqref{5.6} 
	\begin{equation*}
		0= I_0+S_0 \left[1 -\exp \left(- \tau_0 \CI_\infty  \right) \right] - \nu \CI_\infty.
	\end{equation*}
	In Figure \ref{Fig12}  the changes in the parameters $I_0$ and $\tau_0$ (in \eqref{2.8}-\eqref{2.9}) do not affect significantly the final size. 
	\begin{figure}[H]
		\centering
		\includegraphics[scale=0.2]{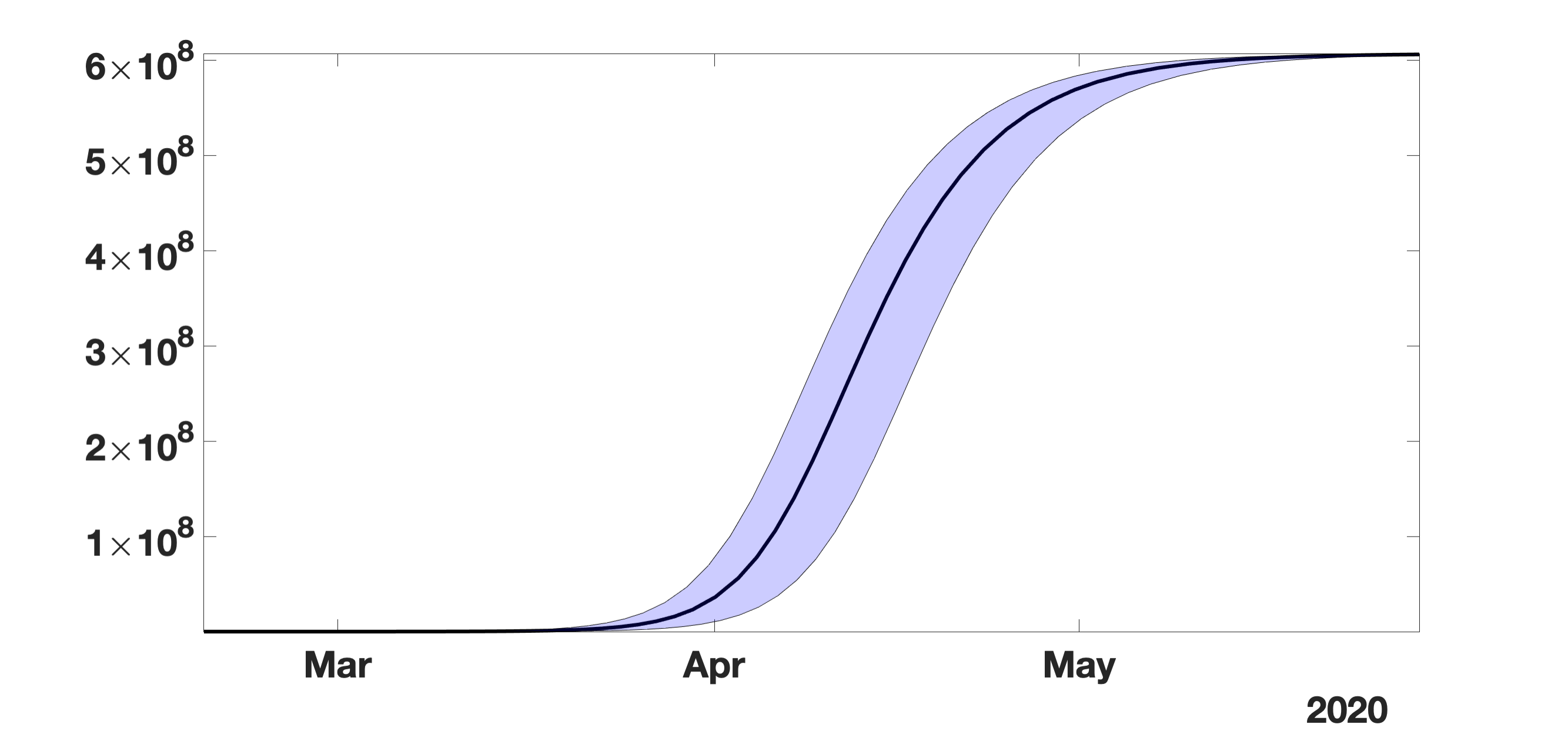}
		
		\caption{\textit{In this figure, the black curve corresponds to the cumulative number of reported cases $\CR(t)$ obtained from the model \eqref{5.4} with $\CR'(t)=\nu f I(t)$ by using the values $I_0=1521$ and $\tau_0=3.32 \times 10^{-10}$ obtained from our method and the early data from February 19 to March 1. The blue region corresponds the $95\%$ confidence  interval when the rate of transmission $\tau(t)$ is constant and equal to the estimated value $\tau_0=3.32 \times 10^{-10}$.   }}\label{Fig12}
	\end{figure}

	\begin{remark} \label{REM5.5}
		Theorem \ref{TH5.3} can be used day by day to fit the cumulative number of infected $\CI(t)$. Indeed, if we assume that $\tau(t)$ is a day-by-day piece-wise constant, we can use the monotone properties to find a unique daily value for $\tau$ to fit the cumulative data to obtain a perfect match. Such an algorithm was developed in \cite{Demongeot20b}.  
	\end{remark}
	
	\section{Modeling a single epidemic wave}
	\label{Section6}
	\subsection{What factors govern the transmission of pathogens}
	\label{Section6.1}
	
	Estimating the average transmission rate is one of the most crucial challenges in the epidemiology
	of communicable diseases. This rate conditions the entry into the epidemic phase of the disease and its
	return to the extinction phase, if it has diminished sufficiently. It is the combination of three factors, one,
	the coefficient of virulence, linked to the infectious agent (in the case of infectious transmissible diseases),
	the other, the coefficient of susceptibility, linked to the host (all summarized into the probability of
	transmission), and also, the number of contact per unit of time between individuals (see \cite{Magal14}). The coefficient of virulence may change over time due to mutation over the course of the
	disease history. The second and third also, if mitigation measures have been taken. This was the case
	in China from the start of the pandemic (see \cite{Qiu20}. Monitoring the decrease in the
	average transmission rate is an excellent way to monitor the effectiveness of these mitigation measures.
	Estimating the rate is therefore a central problem in the fight against epidemics.

	The transmission rate may vary over time, and it may significantly impact epidemic outbreaks.  As explained in \cite{Magal14}, the transmission rate can be decomposed as follow
	$$
	\tcbhighmath[boxrule=2pt,drop fuzzy shadow=blue]{ \tau(t)=\dfrac{\text{the probability of transmission}}{\text{the average duration of a contact}}.} 
	$$
	In this formula, the transmission probability may depend on climatic changes (temperature, humidity, ultraviolet, and other external factors), and the average duration of contact depends on human social behavior. It can be noted that the transmission rate is proportional to the inverse of the average contact duration because the shorter the average contact duration, the greater the number of contacts per unit of time.
	
	\begin{remark} \label{REM6.1}
		A model was proposed by \cite{Chowell} to describe the evolution of the transmission rate during a single epidemic wave. Namely, the model  is the following 
		$$
		\tau(t)=\left\{
		\begin{array}{ll}
			\tau_0,& \text{ if }  t_0 \leq t \leq N, \\
			\tau_0 \left(p \, e^{-\mu(t-N)}+(1-p)\right), & \text{ if }  t \geq N,
		\end{array}
		\right. 
		$$
		where $N$ corresponds to the day when the public measures take effect,  and $\mu$ is the rate at which they take effect (this parameter describes the speed at which the public measures are taking place). The fraction $0 \leq p \leq 1$ is the fraction by which the transmission rate is reduced when applying public measures. We can rewrite this model shortly by using $t^+=\max(t,0)$, the positive part of $t$. That is, 
		$$
		\tau(t)= \tau_0 \left(p \, e^{-\mu (t-N)^+}+(1-p)\right),
		$$
		Such a model was successfully used by \cite{Augeraud-Veron, Liu20b, Liu20c} and others. 
		
		\medskip 
		Nevertheless, the model for joining the end of an epidemic wave to the next epidemic wave is still unknown. A tentative model was proposed in \cite{Augeraud-Veron}. 
	\end{remark}
	\medskip 
	Contact patterns are impacted by social distancing measures. The average number of contacts per unit of time depends on the density of population \cite{Rocklov20, seligmann2020summer}. The probability of transmission depends of the virulence of the pathogen which can depend on the  temperature, the humidity, and the Ultraviolet \cite{Demongeot20a, Wang20}.  In COVID-19 the level of susceptibility may depend on blood group and genetic lineage. It is indeed suspected that the 
	\begin{itemize}
		\item  Blood group\cite{Guillon08} : Blood group O is associated with a lower susceptibility to SARS-CoV2;   
		\item Genetic lineage \cite{Zeberg20} \: A gene cluster inherited
		from Neanderthal has been identified as a risk factor for severe symptoms.
	\end{itemize}
	
	\subsection{More results and references about the time dependent transmission rate modeling}
	\label{Section6.2}
	Throughout this section, the parameter $S_0=1.4 \times 10^9$ will be the entire population of mainland China (since COVID-19 is a newly emerging disease).  The actual number of susceptibles $S_0$ can be smaller since some individuals can be partially (or totally) immunized by previous infections or other factors. This is also true for Sars-CoV2, even if COVID-19 is a newly emerging disease.

	\medskip 
	At the early beginning of the epidemic, the average duration of the infectious period  $1/\nu$ is unknown, since the virus has never been investigated in the past. Therefore, at the early beginning of the COVID-19 epidemic, medical doctors and public health scientists used previously estimated average duration of the infectious period to make some public health recommendations. Here we show that the average infectious period is impossible to estimate by using only the time series of reported cases, and must therefore be identified by other means. Actually, with the data of Sars-CoV2 in mainland China, we will fit the cumulative number of the reported case almost perfectly for any non-negative value  $1/\nu<3.3$ days. In the literature, several  estimations were obtained: $11$ days in \cite{Zhou},  $9.5$ days in \cite{Hu},  $8$ days in \cite{Ma}, and  $3.5$ days in \cite{Li}. The recent survey by Byrne et al. \cite{byrne2020inferred} focuses on this subject.  
	
	\medskip  
	\begin{mybox}{Result}
		In Section \ref{Section6.4}, our analysis shows that  
		\begin{itemize}
			\item   It is hopeless to estimate the exact value of the duration of infectiousness by using SI models. Several values of the average duration of the infectious period give the exact same fit to the data. 
			\item  We can estimate an upper bound for the duration of infectiousness by using SI models. In the case of Sars-CoV2 in mainland China, this upper bound is $3.3$ days. 
		\end{itemize}
	\end{mybox}
	
	In  \cite{rothe2020transmission}, it is reported that transmission of COVID-19 infection may occur from an infectious individual who is not yet symptomatic.  In \cite{WHO}, it is reported that COVID-19 infected individuals generally develop symptoms, including mild respiratory symptoms and fever,  on average $5-6$ days after the infection date (with a confience of $95\%$, range $1-14$ days). In \cite{Yang}, it is reported that the median time prior to symptom onset is 3 days, the shortest 1 day, and the longest 24 days. It is evident that these time periods play an important role in understanding COVID-19 transmission dynamics. Here the fraction of reported individuals $f$  is unknown as well.

	\medskip  
	\begin{mybox}{Result}
		In Section \ref{Section6.4}, our analysis shows that: 
		\begin{itemize}
			\item It is hopeless to estimate the fraction of reported by using the SI models. Several values for the fraction of reported give the exact same fit to the data.
			\item We can estimate a lower bound for the fraction of unreported. We obtain $3.83 \times 10^{-5} < f \leq 1$. This lower  bound is not significant. Therefore we can say anything about the fraction of unreported from this class of models.
		\end{itemize}
	\end{mybox}
	As a consequence,  the parameters $1/\nu$ and $f$ have to be estimated by another method, for instance by  a direct survey methodology that should be employed on an appropriated sample in the population in order to evaluate the two parameters.

	The goal of this section is to focus on the estimation of the two remaining parameters. Namely, knowing the above-mentioned parameters, we plan to identify 
	\begin{itemize}
		\item $I_0$ the initial number of infectious at time $t_0$; 
		\item $\tau(t)$ the rate of transmission at time $t$.
	\end{itemize}
	This problem has already been considered in several articles. In the early 70s, London and Yorke \cite{London73, Yorke73}  already discussed the time dependent rate of transmission in the context of measles, chickenpox and mumps. More recently, \cite{Wang-Ruan} the question of reconstructing the rate of transmission was considered for the 2002-2004 SARS outbreak in China. In \cite{Chowell}  a specific form was chosen for the rate of transmission and applied to the Ebola outbreak in Congo.  Another approach was also proposed in  \cite{Smirnova}. 
	
	\subsection{Why do we need a time-dependent transmission rate?}
	\label{Section6.3}
	In Figure \ref{Fig13}, we observe that the SI model with a constant transmission rate initially fits the data well. With this choice of parameters, the SI model is also supposed with the exponential function. But the model and the exponential function diverge relatively rapidly from the data. It is easy to understand that once people were informed about the COVID-19 outbreak, they tried to protect themself, and the number of contacts per unit of time then reduced gradually. That is, the transmission rate gradually decreased. 
	\begin{figure}[H]
		\centering
		\vspace{-0.3cm}
		\includegraphics[scale=0.9]{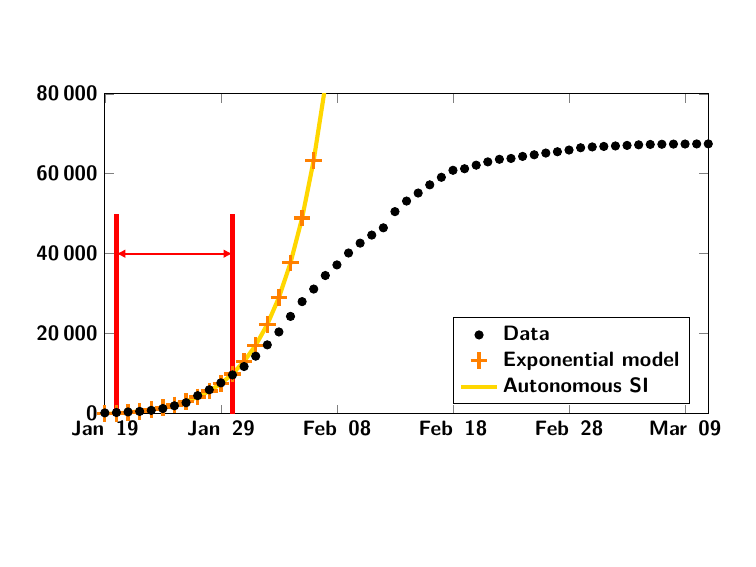}
		\vspace{-1.2cm}
		\caption{{\textit{ In this figure, the black dots represent the cumulative number of cases for China (we a correction for the jump presented in Figure \ref{Fig2}). The period marked in red corresponds to the period considered in Figure \ref{Fig8}. The yellow curve corresponds to the number of infected obtained using model \eqref{4.1} with a constant rate of transmission $\tau(t)$. 
					We observe a rapid divergence between the epidemic model and the data whenever the transmission rate is constant with time.}}}\label{Fig13}
	\end{figure}

	\subsection{Theoretical formula for $\tau(t)$}
	\label{Section6.4}
	By using the S-equation of model \eqref{4.1}  we obtain 
	\begin{equation*}
		S(t)=S_0 \exp \left(-\int_{t_0}^t  \tau(\sigma)\, I( \sigma) \d \sigma \right),
	\end{equation*}
	next by using the I-equation of model \eqref{4.1} we obtain 
	\begin{equation*}
		I'(t)=S_0 \exp \left(-\int_{t_0}^t  \tau(\sigma)\, I( \sigma) \d \sigma \right) \tau(t)\, I(t) - \nu I(t),
	\end{equation*}
	and by taking the integral between $t$ and $t_0$ we obtain a Volterra  integral equation for the cumulative number of infectious 
	\begin{equation}  \label{6.1}
		\CI'(t)= I_0+S_0 \left[1 -\exp \left( -\int_{t_0}^t  \tau(\sigma)\, I( \sigma)  \d \sigma\right) \right] - \nu \CI(t),
	\end{equation}
	which is equivalent to (by using \eqref{4.3})
	\begin{equation} \label{6.2}
		\CR'(t)= \nu \, f  \left(  I_0+S_0 \left[1 -\exp \left( - \dfrac{1}{\nu \, f} \int_{t_0}^t  \tau(\sigma) \, \CR'( \sigma)  \d \sigma\right) \right]   \right)  +\nu \, \CR_0- \nu \CR(t). 
	\end{equation}
	The following result permits to obtain a perfect match between the SI model and the time-dependent rate of transmission $\tau(t)$. 
	\begin{theorem} \label{TH6.2} Let $S_0$, $\nu$, $f$, $I_0>0$ and $\CR_0\geq 0$ be given. Let $t \to I(t)$ be the second component of system \eqref{4.1}. Let $\widehat{\CR} : [t_0, \infty) \to \R$  be a two times continuously differentiable function satisfying 
		\begin{equation}  \label{6.3}
			\widehat{\CR} (t_0)=\CR_0, 
		\end{equation}
		\begin{equation}  \label{6.4}
			\widehat{\CR}'(t_0)=\nu \, f \, I_0,
		\end{equation}
		\begin{equation}  \label{6.5}
			\widehat{\CR}'(t)>0, \forall t  \geq t_0,		
		\end{equation}
		and 
		\begin{equation}  \label{6.6}
			\nu f  \left(I_0+S_0 \right)-\widehat{\CR}'(t)-\nu \left( \widehat{\CR}(t) - \CR_0\right) >0, \forall t  \geq t_0.	 
		\end{equation}
		Then 
		\begin{equation} \label{6.7}
			\widehat{\CR}(t)=\CR_0+ \nu f \int_{t_0}^t I \left(s \right) ds, \forall t \geq t_0, 	
		\end{equation}
		if and only if 
		\begin{equation} \label{6.8}
			\tau(t)=  \dfrac{ \nu f \left(\dfrac{\widehat{\CR}''(t)}{\widehat{\CR}'(t)} +\nu \right)  }{ \nu f  \left(I_0+S_0 \right)-\widehat{\CR}'(t)-\nu \left( \widehat{\CR}(t) -\CR_0  \right)  }.
		\end{equation}
	\end{theorem} 
	\begin{proof} Assume first \eqref{6.7} is satisfied.  Then by using equation \eqref{6.1} we deduce that 
		\begin{equation*}
			S_0 \exp \left(- \int_{t_0}^t \tau(\sigma) I(\sigma) d \sigma \right) = I_0+S_0- I(t)-  \nu \CI(t).
		\end{equation*}
		Therefore 
		\begin{equation*}
			\int_{t_0}^t \tau(\sigma) I(\sigma) d \sigma =\ln \left[\dfrac{S_0}{ I_0+S_0-I(t)-\nu \CI(t) } \right] =\ln \left(S_0 \right)- \ln\left[ I_0+S_0-I(t)-\nu \CI(t)\right] 
		\end{equation*}
		therefore by taking the derivative on both side 
		\begin{equation} \label{6.9}
			\tau(t) I(t)=  \dfrac{ I'(t) +\nu I(t) }{  I_0+S_0-I(t)-\nu \CI(t)  } \Leftrightarrow \tau(t)=  \dfrac{\dfrac{I'(t)}{I(t)}  +\nu  }{  I_0+S_0-I(t)-\nu \CI(t)  }
		\end{equation}
		and by using the fact that $\CR(t)-\CR_0=\nu f \CI(t)  $ we obtain \eqref{6.8}. 
		
		Conversely, assume that $\tau(t)$ is given by  \eqref{6.8}. Then if we define $\widetilde{I}(t)= \widehat{\CR}'(t)/ \nu f$ and $\widetilde{\CI}(t)= \left(\widehat{\CR}(t)-\CR_0 \right)/\nu f$, by using \eqref{6.3} we deduce that 
		$$
		\widetilde{\CI}(t)=\int_{t_0}^t \widetilde{I}(\sigma) d \sigma,
		$$
		and by using \eqref{6.4} 
		\begin{equation}\label{6.10}
			\widetilde{I}(t_0)=I_0. 
		\end{equation}
		Moreover from \eqref{6.8} we deduce that $\widetilde{I}(t)$ satisfies \eqref{6.9}. By using \eqref{6.10} we deduce that $t \to \widetilde{\CI}(t)$ is a solution of \eqref{6.1}. By uniqueness of the solution of \eqref{6.1}, we deduce that  $\widetilde{\CI}(t)=\CI(t), \forall t \geq t_0$ or equivalently $\CR(t)=\CR_0+ \nu f \int_{t_0}^t I \left(s \right) ds, \forall t \geq t_0$. The proof is completed. 
	\end{proof}
	
	The formula \eqref{6.8} was already obtained by Hadeler \cite[see Corollary 2]{Hadeler1}.
	\subsection{Explicit formula for  $\tau(t)$ and $I_0$}
	\label{Section6.5}
	
	In 1766,  Bernoulli \cite{Bernoulli1766} investigated an epidemic phase followed by an endemic phase. This appears clearly in Figures 9 and 10 in \cite{Dietz-Heesterbeek}  who revisited the original article of Bernoulli. We also refer to \cite{Blower04} for another article revisiting the original work of Bernoulli. A similar article has been re-written in French as well by  \cite{Bernoulli-Chapelle}.  In 1838, Verhulst \cite{Verhulst1838} introduced the same equation to describe population growth.  Several works comparing cumulative reported cases data and the Bernoulli--Verhulst model appear in the literature  (see \cite{Hsieh, Wang-Wu-Yang, Zhou-Yan}).  The Bernoulli--Verhulst model is sometimes called Richard's model,\index{Richardsmodel@Richard's model} although Richard's work came much later in 1959.

	Many phenomenological models have been compared to the data during the first phase of the COVID-19 outbreak. We refer to the paper of \cite{Tsoularis} for a nice survey on the generalized logistic equations. Let us consider here for example, the   Bernoulli-Verhulst equation 
	\begin{equation} \label{6.11}
		\CR'(t)=\chi_2 \, \CR(t)\left( 1- \left(\dfrac{\CR(t)}{\CR_\infty}  \right)^\theta \right), \forall t \geq t_0,
	\end{equation}
	supplemented with the initial data 
	\begin{equation*}
		\CR(t_0)=\CR_0 \geq 0. 
	\end{equation*}
	Let us recall the explicit formula for the solution  of \eqref{6.11} 
	\begin{equation}\label{6.12}
		\CR(t)=\dfrac{e^{\chi_2 (t-t_0)} \CR_0 }{\left[  1+ \dfrac{\chi_2 \theta}{\CR_\infty^\theta} \int_{t_0}^t  \left(e^{\chi_2  \left(\sigma- t_0 \right)  } \CR_0 \right)^\theta d \sigma \right]^{ 1/\theta} }=\dfrac{e^{\chi_2 (t-t_0)} \CR_0 }{\left[  1+ \dfrac{ \CR_0^\theta}{\CR_\infty^\theta}  \left(e^{\chi_2 \theta   \left(t- t_0 \right)  } -1 \right) \right]^{ 1/\theta} }.
	\end{equation}
	The model's main advantage is that it is rich enough to fit the data, together with a limited number of parameters. To fit this model to the data, we only need to estimate four parameters $\chi_2, \theta, \CR_0$, and $\CR_\infty$.  
	
	\begin{remark} \label{REM6.3}
		Plenty of possibilities exist to fit the data, including split functions (irregular functions with many parameters) and others. In \cite{Burger},  they proposed several possible alternatives, including a generalized logistic equation of the form
		\begin{equation*} 
			\CR'(t)=\chi_2 \, \CR(t)^\theta \left( 1- \left(\dfrac{\CR(t)}{\CR_\infty}  \right) \right), \forall t \geq t_0.
		\end{equation*}
		The above equation has no explicit solution. Therefore it is more difficult to use it than the Bernoulli-Verhulst model. We also refer to  \cite{Pelinovsky22, Pelinovsky20}  for more phenomenological model to fit an epidemic wave. 
	\end{remark}
	
	\begin{figure}[H]
		\centering
		\includegraphics[scale=0.2]{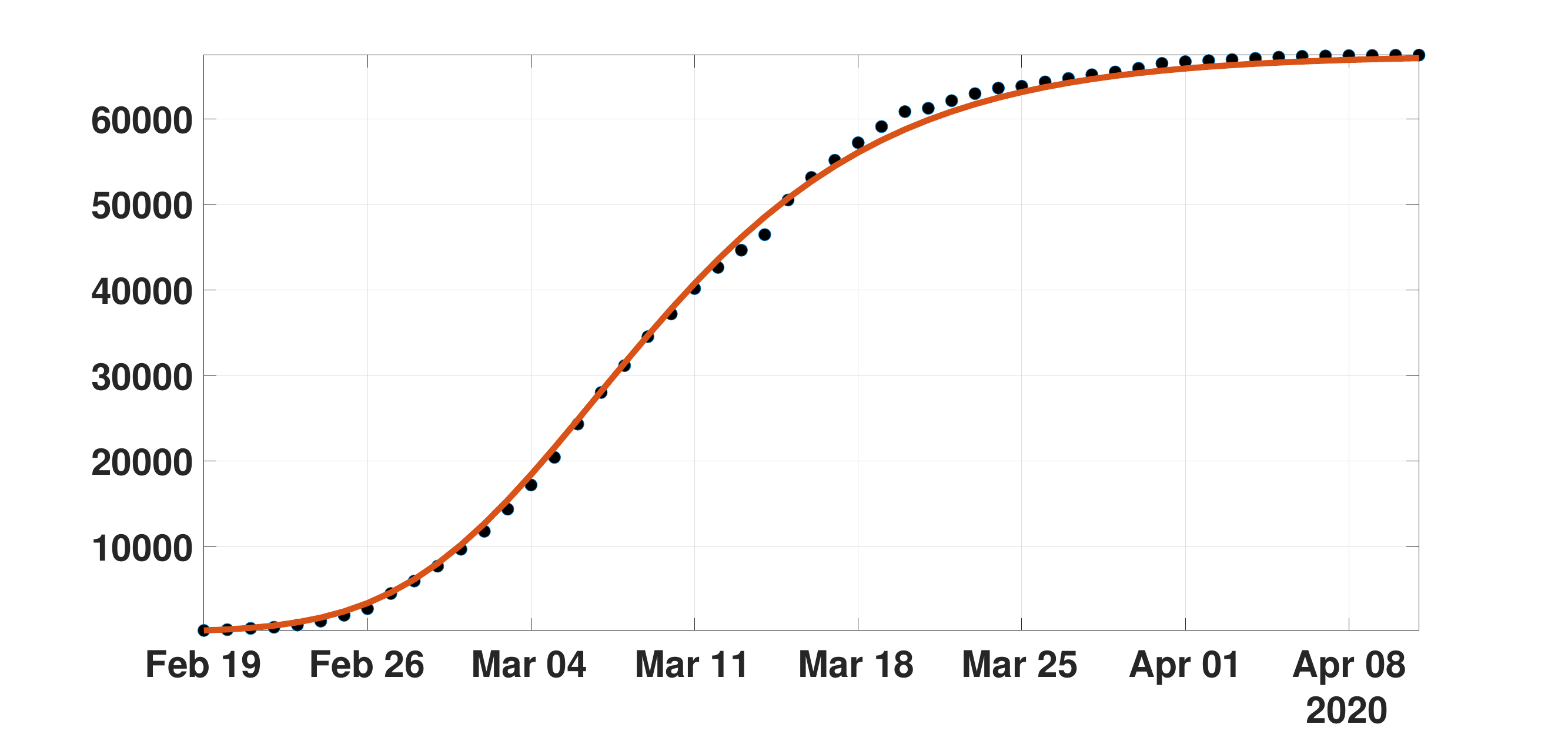}
		\caption{\textit{In this figure, we plot the best fit of the Bernoulli-Verhulst model to the cumulative number of reported cases of COVID-19 in China. We obtain $\chi_2=0.66$ and $\theta=0.22$. The black dots correspond to data for the cumulative number of reported cases and the red curve corresponds to the model. }}\label{Fig14}
	\end{figure}

	\begin{mybox}{Estimated initial number of infected}
		By combining \eqref{6.1}  and the Bernoulli-Verhulst equation  \eqref{6.11}  for $t \to \CR(t)$, we deduce the initial number of infected 
		\begin{equation}\label{6.13}
			I_0=  \dfrac{\CR'(t_0)}{\nu \, f }= \dfrac{\chi_2 \, \CR_0 \left( 1- \left(\dfrac{\CR_0 }{\CR_\infty}  \right)^\theta \right) }{\nu \, f }  .
		\end{equation}
	\end{mybox}
	\begin{remark}  
		We fix $f=0.5$, from the COVID-19 data in mainland China  and formula \eqref{6.13} (with $\CR_0=198$), we obtain 
		$$
		I_0= 1909 \text{ for } \nu=0.1, 
		$$
		and
		$$
		I_0= 954 \text{ for } \nu=0.2.
		$$
	\end{remark}
	
	By using \eqref{6.11} we deduce that 
	\begin{equation*}
		\begin{array}{ll}
			\CR''(t)&=\chi_2 \, \CR'(t)\left( 1- \left(\dfrac{\CR(t)}{\CR_\infty}  \right)^\theta \right)-\dfrac{\chi_2 \theta}{\CR_\infty^\theta} \, \CR(t)  \left(\CR(t)  \right)^{\theta-1}  \CR'(t)\\
			&=\chi_2 \, \CR'(t)\left( 1- \left(\dfrac{\CR(t)}{\CR_\infty}  \right)^\theta \right)-\dfrac{\chi_2 \theta}{\CR_\infty^\theta} \,   \left(\CR(t)  \right)^{\theta}  \CR'(t),	
		\end{array}
	\end{equation*}
	therefore 
	\begin{equation} \label{6.14}
		\CR''(t)=\chi_2 \, \CR'(t)\left( 1-(1+\theta) \left(\dfrac{\CR(t)}{\CR_\infty}  \right)^\theta \right).
	\end{equation}
	\begin{mybox}{Estimated rate of transmission}
		By using the Bernoulli-Verhulst equation \eqref{6.11} and  substituting \eqref{6.14}  in \eqref{6.8}, we obtain
		\begin{equation}\label{6.15}
			\tau(t)=  \dfrac{  \nu \, f \left( \chi_2 \, \left( 1-(1+\theta) \left(\dfrac{\CR(t)}{\CR_\infty}  \right)^\theta \right) +\nu \right) }{\nu \, f  \left(I_0+S_0 \right)+ \nu \CR_0 - \CR(t) \left(\chi_2 \left( 1- \left(\dfrac{\CR(t)}{\CR_\infty}  \right)^\theta \right)+ \nu   \right) }.
		\end{equation}
		This formula \eqref{6.15} combined with \eqref{6.12}  gives an explicit formula for the  rate of transmission. 
	\end{mybox}
	Since $\CR(t) < \CR_\infty$,   by considering the sign of the numerator and the denominator of \eqref{6.15}, we obtain the following proposition. 
	\begin{proposition}  \label{PROP6.4}
		The rate of transmission  $\tau(t)$ given by \eqref{6.15} is non negative for all $t \geq t_0$ if 
		\begin{equation} \label{6.16}
			\nu\geq \chi_2 \, \theta,
		\end{equation}
		and 
		\begin{equation}\label{6.17}
			f  \left(I_0+S_0 \right)+ \nu \CR_0 > \CR_\infty \left(\chi_2+ \nu  \right).
		\end{equation}
	\end{proposition}
	
	\begin{mybox}{Compatibility of the model SI with the COVID-19 data for mainland China}
		The model SI is compatible with the data only when $\tau(t)$ stays positive for all $t \geq t_0$. From our estimation of the Chinese's COVID-19 data we obtain $\chi_2 \, \theta=0.14$. Therefore from \eqref{6.16} we deduce that model is compatible with the data only when 
		\begin{equation} \label{6.18}
			1/\nu \leq 1/0.14=3.3 \text{ days}.
		\end{equation}
		This means that the average duration of infectious period $1/\nu$ must be shorter than $3.3$ days. 
		
		\medskip 
		Similarly the condition  \eqref{6.17} implies 
		$$
		f \geq \dfrac{\CR_\infty \chi_2+\left( \CR_\infty- \CR_0\right) \nu  }{S_0+I_0} \geq \dfrac{\CR_\infty \chi_2+\left( \CR_\infty- \CR_0\right) \chi_2 \, \theta  }{S_0+I_0} 
		$$ 
		and since   we have $CR_0=198$ and $\CR_\infty=67102$, we obtain 
		\begin{equation} \label{6.19}
			f \geq \dfrac{67102 \times 0.66+ (67102-198)\times  0.14}{1.4\times 10^{9}} \geq 3.83 \times 10^{-5}.
		\end{equation}
		So according to this estimation the fraction of unreported $0<f \leq 1$ can be almost as small as we want.  
	\end{mybox}
	Figure \ref{Fig15} illustrates the  Proposition \ref{PROP6.4}. We observe that the formula for the rate of transmission \eqref{6.15} becomes negative  whenever $\nu < \chi_2 \theta$. 
	\begin{figure}[H]
		\centering
		\textbf{(a)} \\
		\includegraphics[scale=0.2]{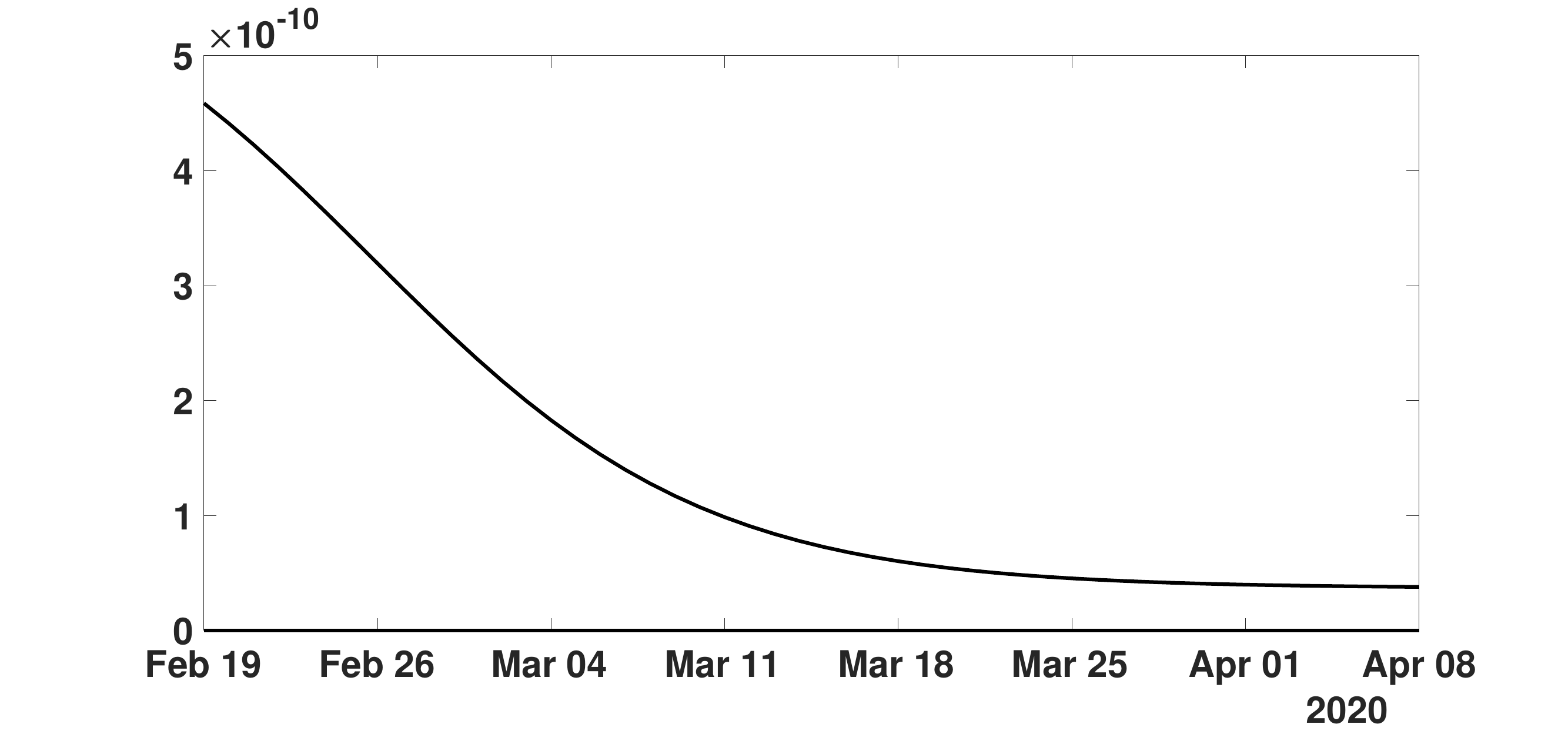}\\
		\textbf{(b)} \\
		\includegraphics[scale=0.2]{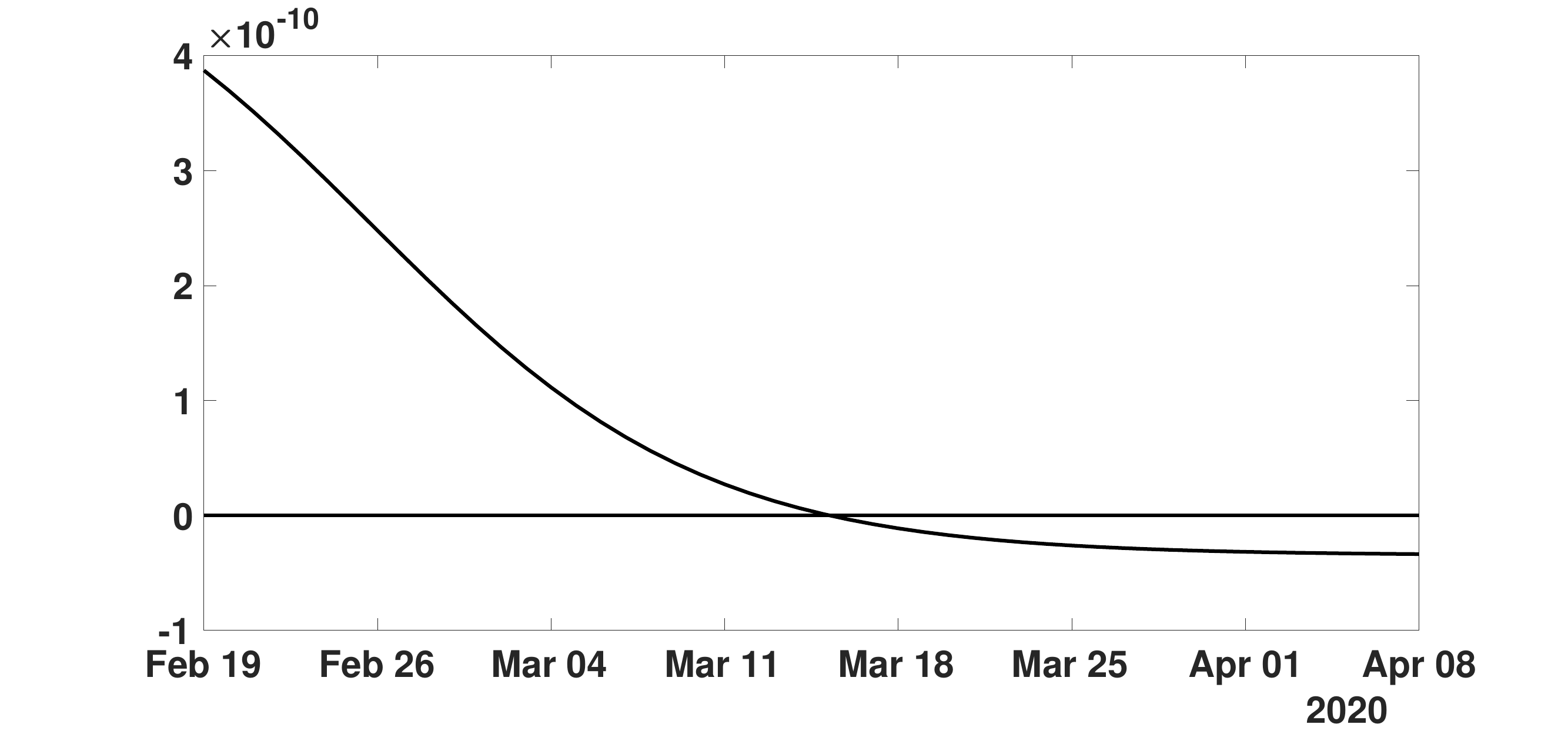}
		\caption{\textit{In this figure, we plot the rate of transmission obtained from formula \eqref{6.15} with $f=0.5$, $\chi_2 \, \theta=0.14<\nu=0.2$ (in Figure (a)) and $\nu=0.1<\chi_2 \, \theta=0.14$ (in Figure (b)),  $\chi_2=0.66$ and $\theta=0.22$ and $\CR_\infty=67102$ which is the latest value obtained from the cumulative number of reported cases for China.  }}\label{Fig15}
	\end{figure}
	In Figure \ref{Fig16} we plot the numerical simulation  obtained from \eqref{4.1}-\eqref{4.3} when $t \to \tau(t)$ is replaced by the explicit formula \eqref{6.15}. It is surprising that we can reproduce perfectly the original Bernoulli-Verhulst even when $\tau(t)$ becomes negative. This was not guaranteed at first, since the I-class of individuals is losing some individuals which are recovering. 
	\begin{figure}[H]
		\centering
		\includegraphics[scale=0.2]{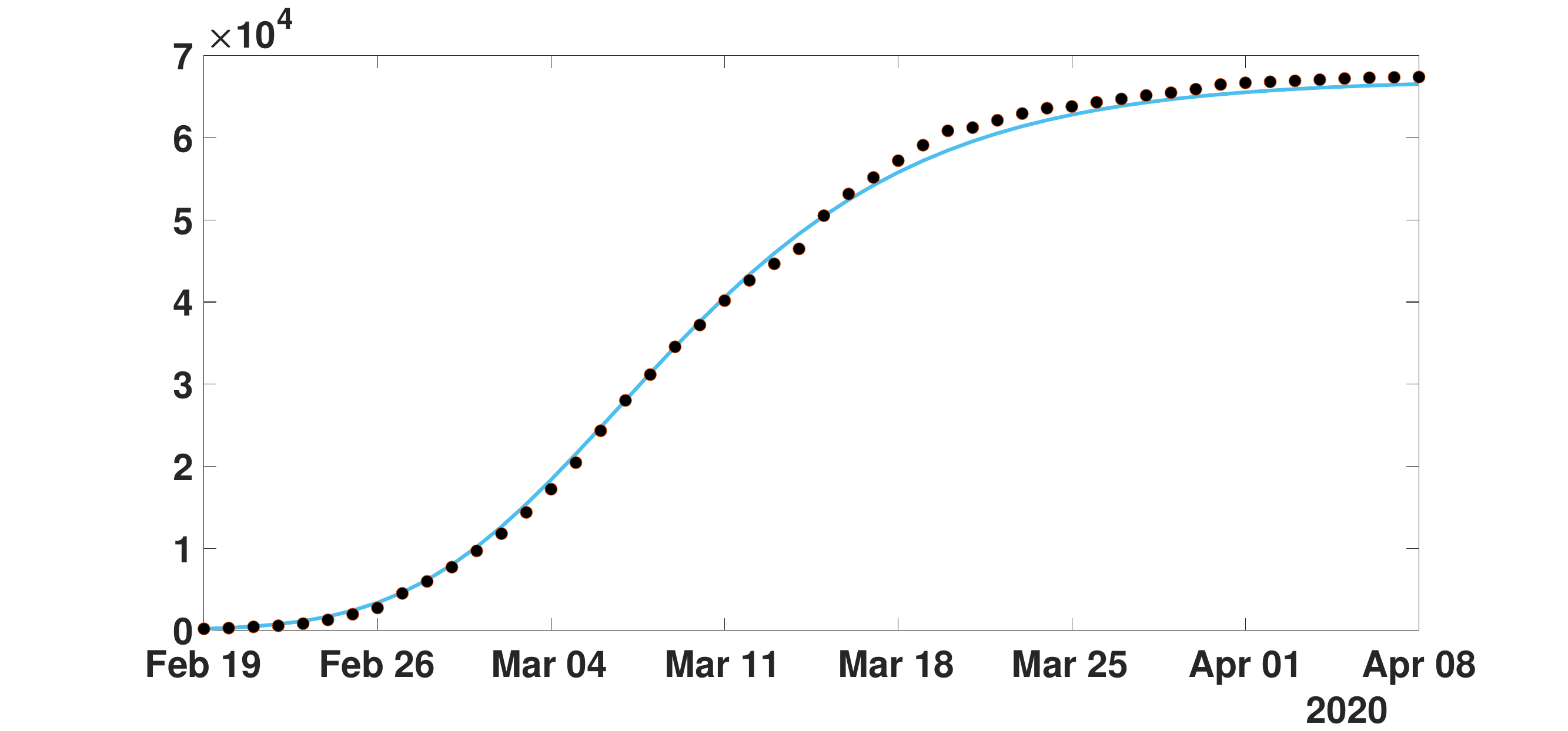}
		\caption{\textit{In this figure, we plot the number of reported cases by using model \eqref{4.1} and \eqref{6.1}, and the rate of transmission is obtained in \eqref{6.15}.  The parameters values are $f=0.5$, $\nu=0.1$ or $\nu=0.2$,  $\chi_2=0.66$ and $\theta=0.22$  and $\CR_\infty=67102$ is the latest value obtained from the cumulative number of reported cases for China. Furthermore, we use $S_0=1.4 \times 10^9$ for the total population of China and $I_0=954$ which is obtained from formula \eqref{6.13}. The black dots correspond to data for the cumulative number of reported cases observed and the blue curve corresponds to the model. }}\label{Fig16}
	\end{figure}
	
	\subsection{Results}
	\label{Section6.6}
	In \cite{Demongeot20b}, we designed an algorithm, based on the monotone property described in Theorem \ref{TH5.3} to recover the transmission rate from the data.  In this section, we reconsider the result presented in \cite{Demongeot20b} where several method was used to regularized the data. 
	
	In Figure \ref{Fig17} we plot several types of regularized cumulative data in figure (a) and several types of regularized daily data in figure (b). Among the different regularization methods, an important one is the Bernoulli-Verhulst best fit approximation. 
	\begin{figure}[H]
		\centering
		\textbf{(a)} \hspace{6cm} \textbf{(b)}\\
		\includegraphics[scale=0.13]{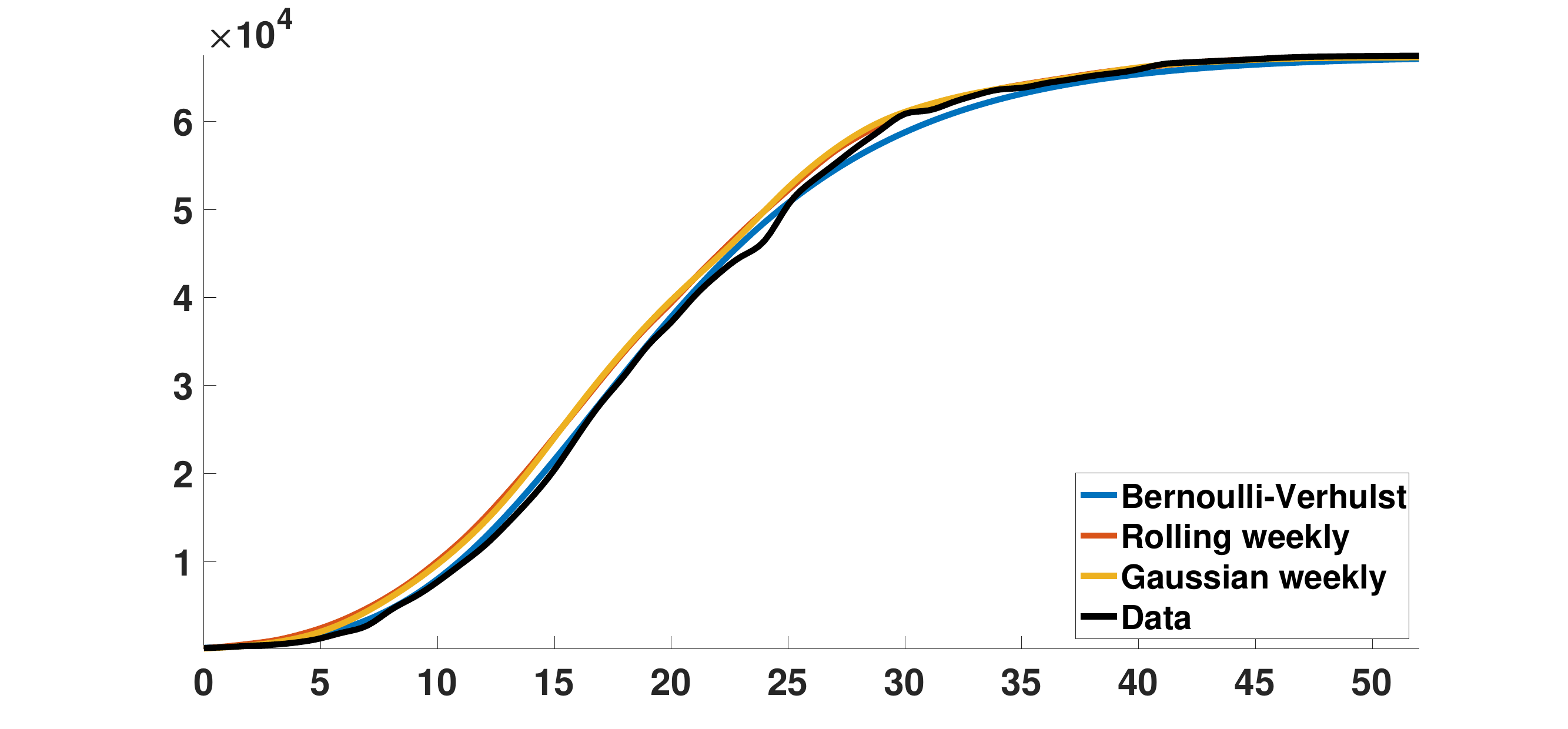}
		\includegraphics[scale=0.13]{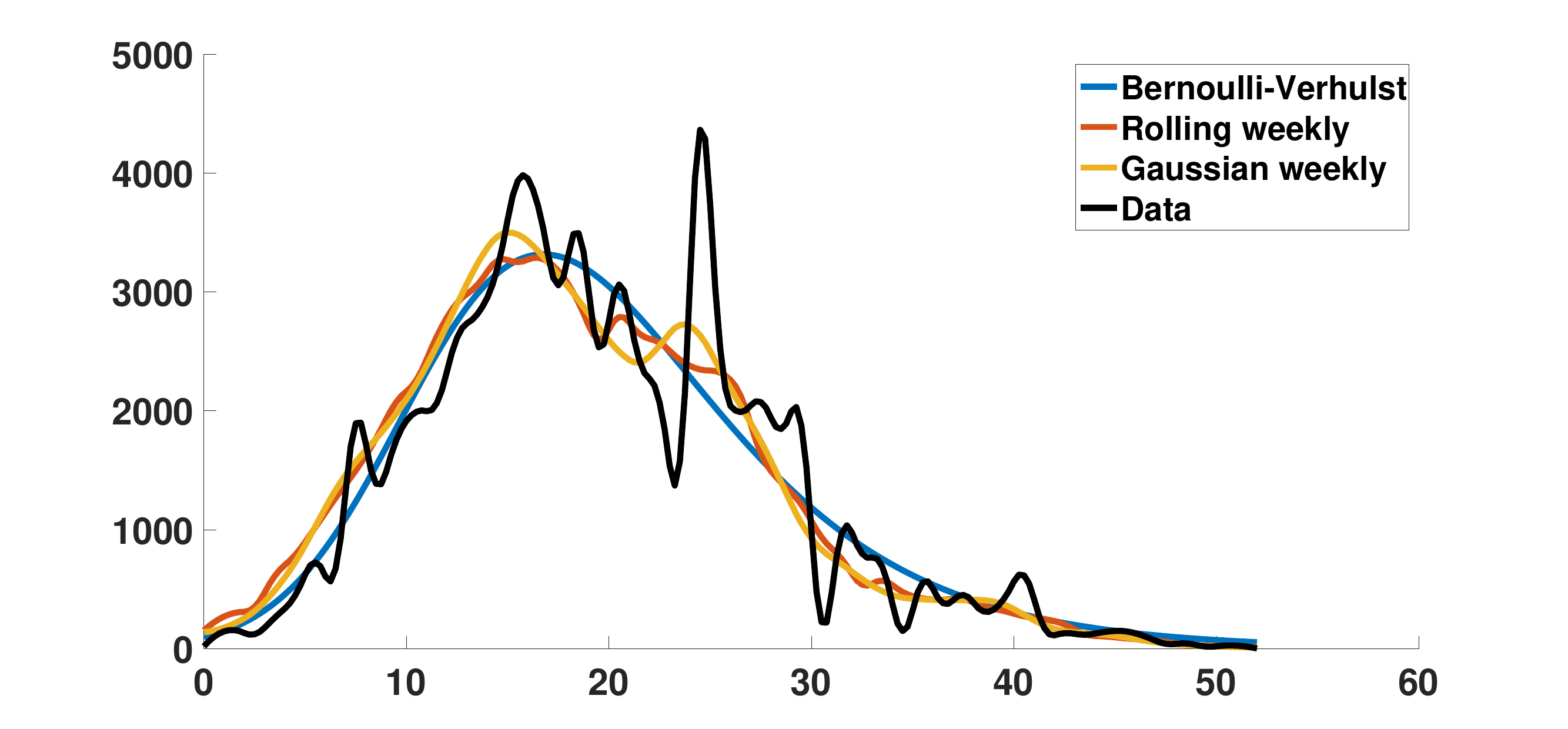}
		
		\caption{\textit{In this figure, we plot the cumulative number of reported cases (left) and the daily number of reported cases (right).  The black curves are obtained by applying the cubic spline matlab function "spline(Days,DATA)" to the cumulative data. The left-hand side is obtained by using the cubic spline function and right-hand side is obtained by using the derivative of the cubic spline interpolation. The blue curves are obtained by using cubic spline function to the day by day values of cumulative number of cases obtained from the best fit of the Bernoulli-Verhulst model. The orange curves are obtained by computing the rolling weekly  daily number of cases (we use the matlab function "smoothdata(DAILY,'movmean',7)") and then by applying the cubic spline function the corresponding cumulative number of cases. The yellow curves are obtained by Gaussian the rolling weekly to the daily number of cases (we use the matlab function "smoothdata(DAILY,'gaussian',7)") and then by applying the cubic spline function to the corresponding cumulative number of cases.  }}\label{Fig17}
	\end{figure}
	In Figure   \ref{Fig18} we plot the rate of transmission $t \to \tau(t)$ obtained by using Algorithm 2.  We can see that the original data gives a negative transmission rate while at the other extreme the Bernoulli-Verhulst seems to give the most regularized transmission rate. In Figure   \ref{Fig18}-(a) we observe that we now recover almost perfectly the theoretical transmission rate obtained in \eqref{6.15}. In Figure   \ref{Fig18}-(b) the rolling weekly average regularization and in Figure   \ref{Fig18}-(c) the Gaussian weekly average  regularization still vary a lot and in both cases  the transmission rate becomes negative after some time. In Figure   \ref{Fig18}-(c) the original data gives a transmission rate that is negative from the beginning. We conclude that it is crucial to find a "good" regularization of the daily number of case. So far the best regularization method is obtained by using the best fit of the Bernoulli-Verhulst model.

	%\medskip 
	%{\color{red} \textbf{Put the algorithm to approximate stpe by step the the function $\tau(t)$. }}

	\begin{figure}[H]
		\centering
		\textbf{(a)} \hspace{6cm} \textbf{(b)}\\
		
		\includegraphics[scale=0.13]{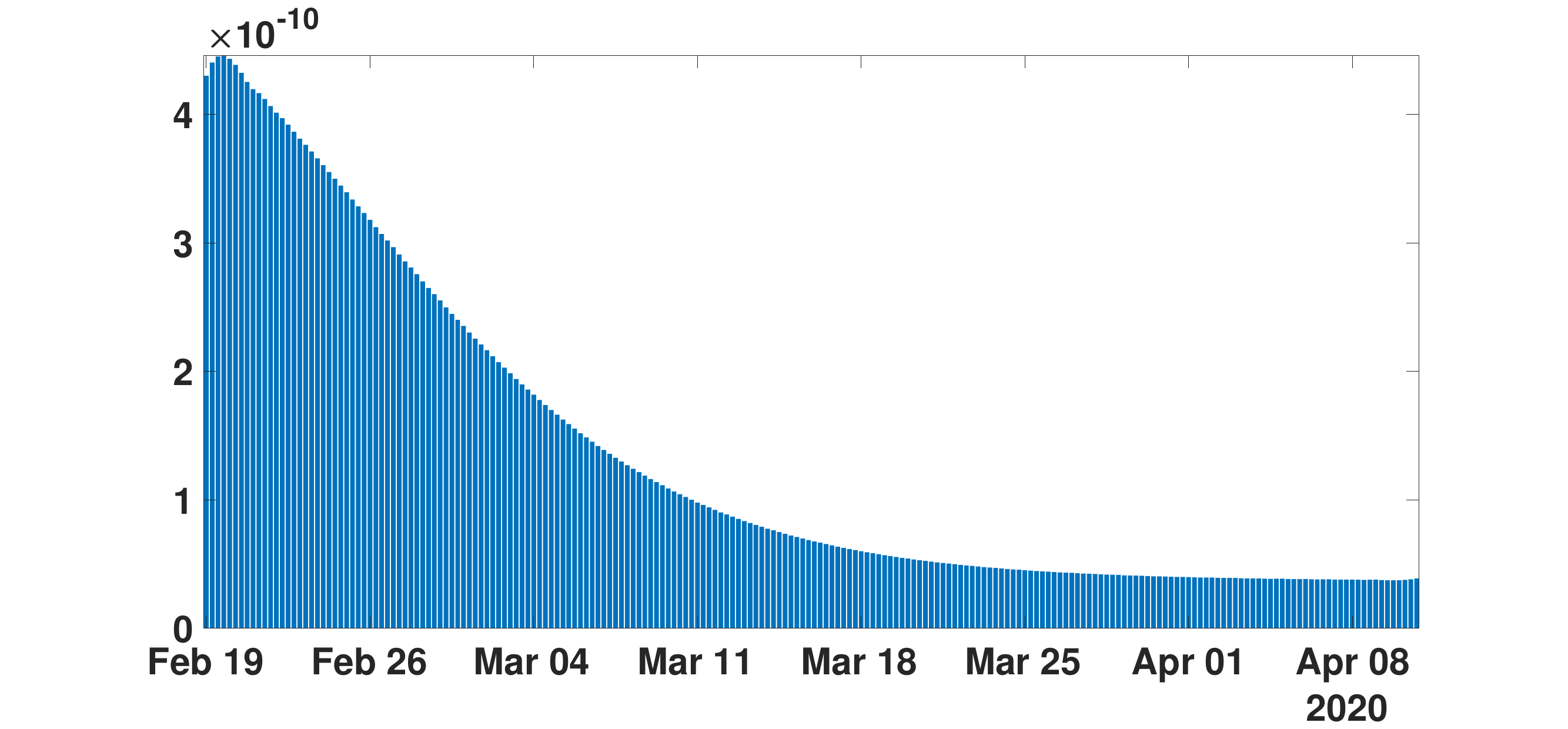}
		\includegraphics[scale=0.13]{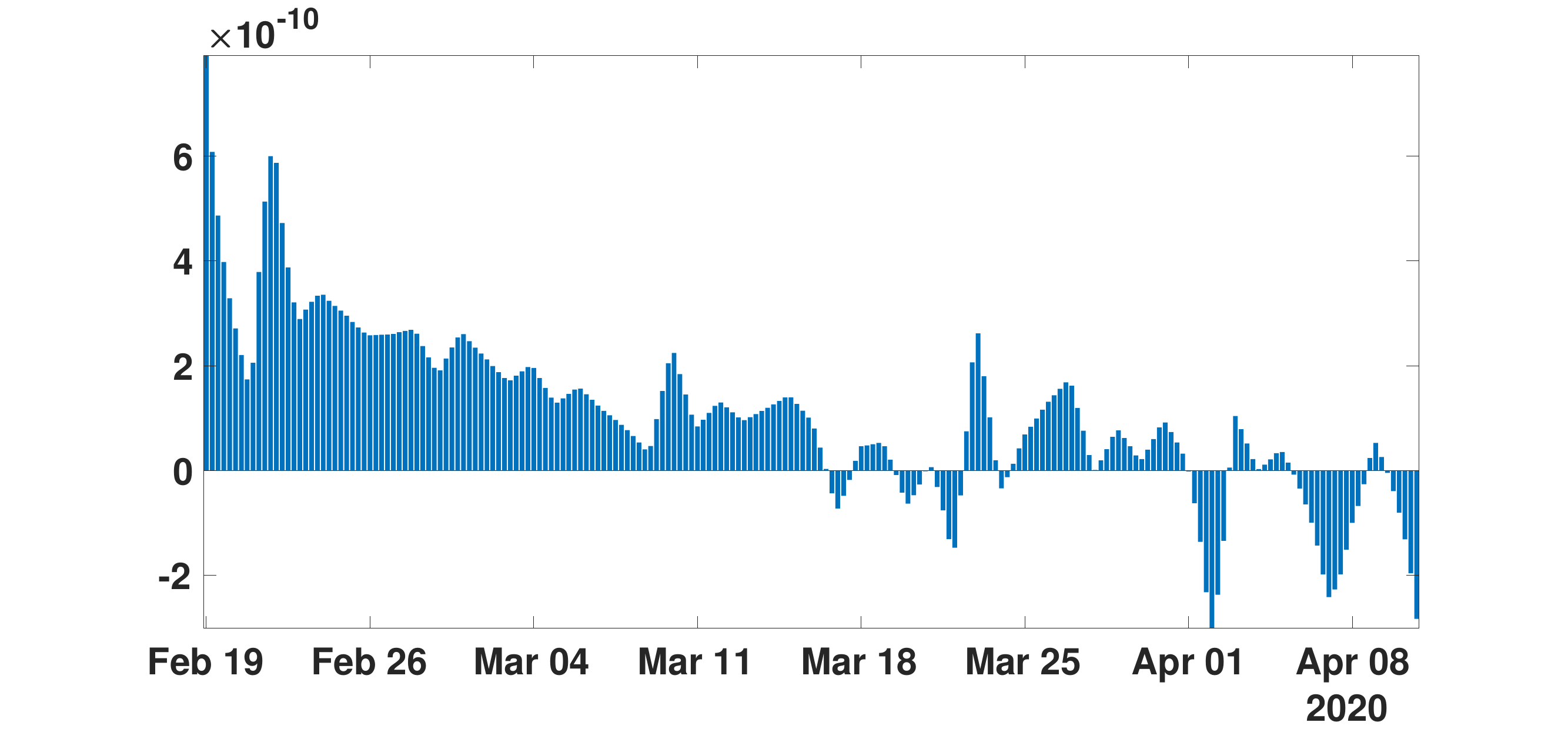}\\
		\textbf{(c)} \hspace{6cm} \textbf{(d)}\\
		\includegraphics[scale=0.13]{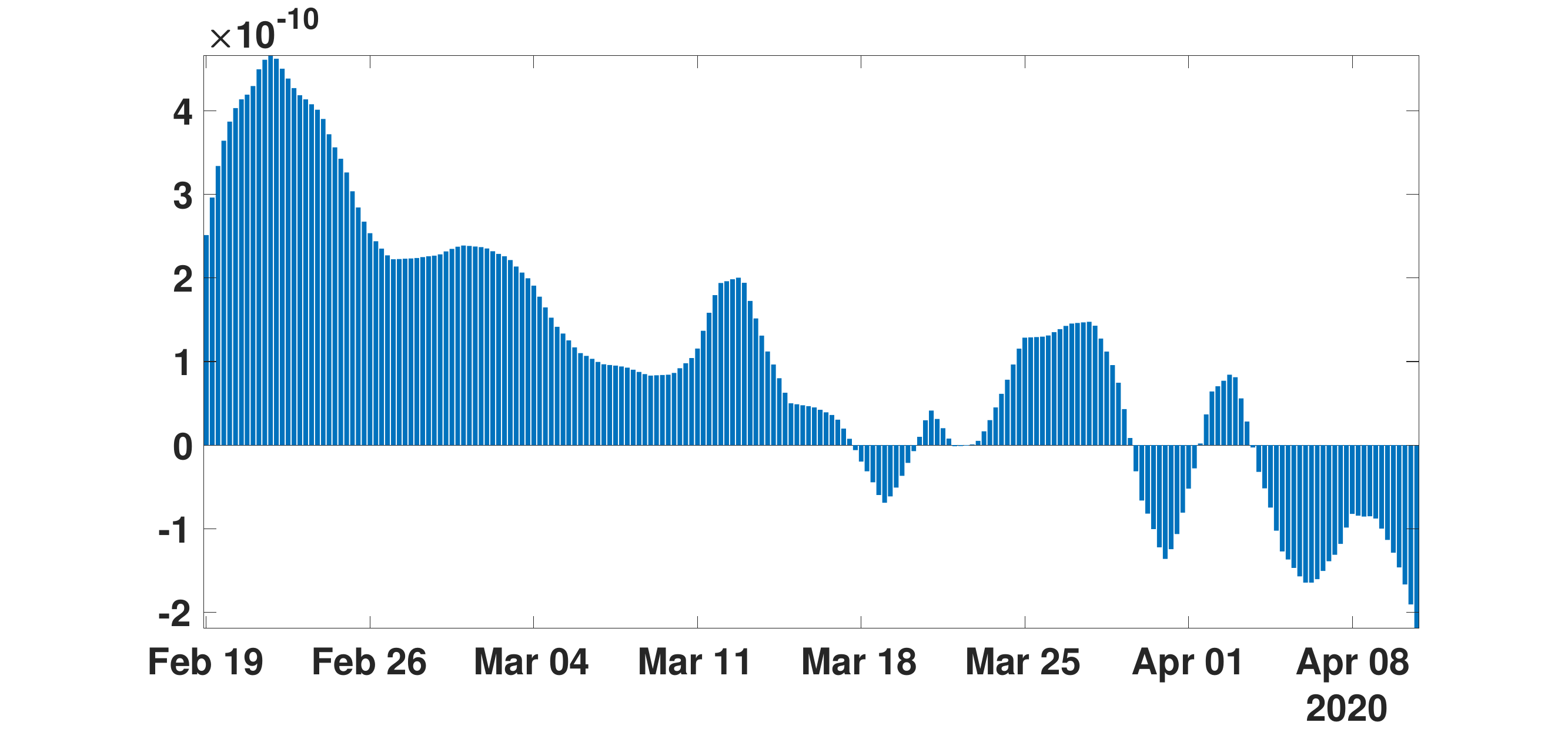}
		\includegraphics[scale=0.13]{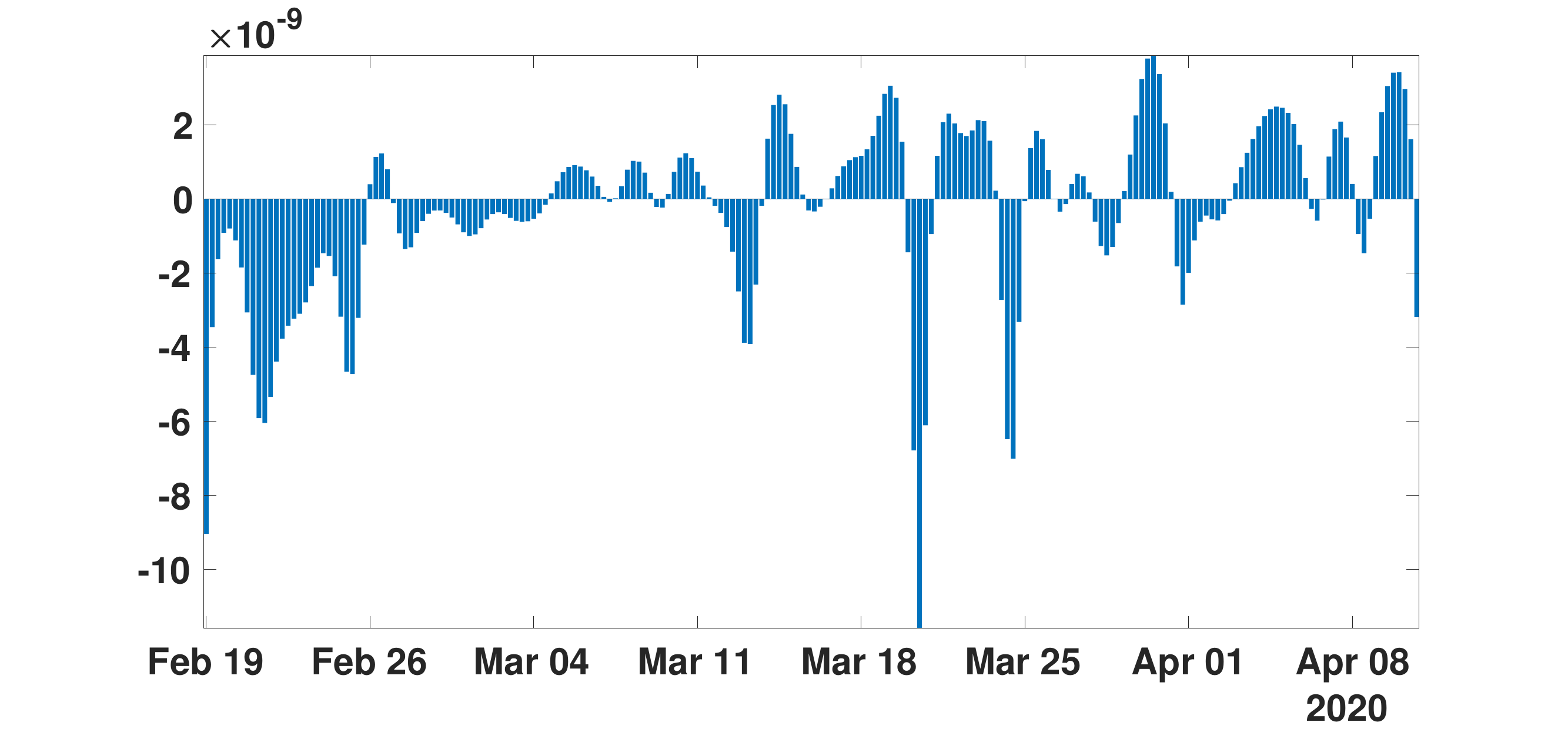}\\

		\caption{\textit{In this figure we plot the transmission rates $t \to \tau(t)$ obtained by using Algorithm 2 with the parameters $f=0.5$ and $\nu=0.2$.   In figure (a) we use the cumulative data obtained by using the Bernoulli-Verhulst regularization. In figure (b) we use the cumulative data obtained by using the rolling weekly average regularization.  In figure (c) we use the cumulative data obtained by using  the Gaussian weekly average regularization.  In figure (d) we use the original cumulative data.    }}\label{Fig18}
	\end{figure}

	\section{Modeling multiple epidemic waves}
	\label{Section7}
	\subsection{Phenomenological model used for multiple epidemic waves}
	\label{Section7.1}
	\medskip
	\noindent \textbf{Endemic phase:}\index{endemic phase} During the endemic phase, the dynamics of new cases appears to fluctuate around an average value independently of the number of cases. Therefore the average cumulative number of  cases is given by
	\begin{equation} \label{7.1}
		\tcbhighmath[boxrule=2pt,drop fuzzy shadow=blue]{ 	\CR(t)=N_0+ (t-t_{0}) \times a, \text{ for } t \in [t_0 , t_1],}
	\end{equation}
	where $t_0$ {denotes} the beginning of the endemic phase, $N_0$ is the number of new cases at time $t_0$, and $a$ is the average value of the daily number of new cases.

	\medskip 
	\noindent \textbf{Epidemic phase:}\index{endemic phase} In the epidemic phase, the new cases  are contributing to produce  secondary cases. Therefore the daily number of new cases is no longer constant, but varies with time as follows
	\begin{equation}\label{7.2}
		\tcbhighmath[boxrule=2pt,drop fuzzy shadow=blue]{ 	\CR(t)=	N_{\rm{base}}+ \dfrac{\mathrm{e}^{\chi (t-t_0)} N_0 }{\left[  1+ \dfrac{ N_0^\theta}{N_\infty^\theta}  \left(\mathrm{e}^{\chi \theta   \left(t- t_0 \right)  } -1 \right) \right]^{ 1/\theta} }, \text{ for } t \in [t_0 , t_1].}
	\end{equation}
	In other words, the daily number of {new} cases  follows the Bernoulli--Verhulst equation. That is, 
	\begin{equation}  \label{7.3}
		\tcbhighmath[boxrule=2pt,drop fuzzy shadow=blue]{ N(t)=\CR(t)-N_{\rm{base}},}
	\end{equation} we obtain
	\begin{equation} \label{7.4}
		N'(t)= \chi \,  N(t) \, \left[ 1 - \left( \dfrac{N(t) }{N_\infty}\right)^\theta \right],
	\end{equation}
	completed with the initial value
	\begin{equation*}
		N(t_0)=N_0.
	\end{equation*}
	In the model, $N_{\rm{base}}+N_0$  corresponds to the value $\CR(t_0)$ of the cumulative number of cases at time $t=t_0$. The parameter $N_\infty+N_{\rm{base}}$ is the maximal value of the cumulative reported {cases} after the time $t=t_0$.   $\chi>0$ is a Malthusian growth  parameter, and $\theta$ {regulates} the speed at which $\CR(t)$ {increases} to $N_\infty+N_{\rm{base}}$.

	\medskip
	\noindent \textbf{Regularize the junction between the epidemic phases:} \index{regularized model} Because the formula for $\tau(t)$ involves derivatives of the phenomenological model {regularizing} $\CR(t)$ (see equations \eqref{5.5}), we need to connect the phenomenological models {of} the different phases (epidemic and endemic) as smoothly as possible. Let $t_0, \ldots, t_n$ denote the $n+1$ breaking points of the model, that is, the times at which there is a transition between one phase and the next one. We let $\widetilde{\CR}(t)$ be the global model obtained by placing the phenomenological models of the different phases side by side.
	
	More precisely, $\widetilde \CR(t) $ is defined by \eqref{7.2} during an epidemic phase $[t_i, t_{i+1}]$, or during the initial phase $(-\infty, t_0]$ or the last phase $[t_n, +\infty)$. During an endemic phase, $\widetilde \CR(t)$ is defined by \eqref{7.1}. The parameters are chosen so that the resulting global model $\widetilde\CR$ is continuous.   We define the regularized model by using the convolution formula:
	\begin{equation}\label{7.5}
		\CR(t) = \int_{-\infty}^{+\infty } \mathcal G(t-s ) \times \widetilde{\CR}(s)\d s = ( \mathcal G \ast \widetilde{\CR} )(t),
	\end{equation}
	where
	$$
	\mathcal G(t):= \frac{1}{\sigma\sqrt{2\pi}}\mathrm{e}^{-\frac{t^2}{2\sigma^2}}
	$$
	is the Gaussian function with mean 0 and variance $\sigma^2$. The parameter $\sigma$ controls the trade-off between smoothness and precision: increasing $\sigma$ reduces the variations in $\CR(t)$ and reducing $\sigma $ reduces the distance between $\CR(t)$ and $\widetilde{\CR}(t)$.   In any case the resulting function $\CR(t)$ is very smooth (as well as its derivatives) and close to the original model $\widetilde{\CR}(t)$ when $\sigma$ is not too large. Here, we fix 
	$$
	\sigma = 7 \text{ days.}
	$$
	Numerically, we will need to compute some $t \to \CR(t)$ derivatives. Therefore it is convenient to take advantage of  the convolution \eqref{7.5} and deduce that
	\begin{equation}\label{7.6}
		\dfrac{\d^n \CR(t)}{\d t^n}	 = \int_{-\infty}^{+\infty } \dfrac{\d^n  \mathcal G(t-s )  }{\d t^n} \times \widetilde{\CR}(s)\d s ,
	\end{equation}
	for $n=1,2, 3$.
	
	\subsection{Phenomenological Model apply to France}
	\label{Section7.2}
	Figures \ref{Fig18b}-\ref{Fig18c} below is taken from	\cite{griette2021robust}.  In Figure \ref{Fig18b}, we present the best fit of our phenomenological model for the cumulative reported case data of COVID-19 epidemic in France. The yellow regions correspond to the endemic phases and the blue regions correspond to the epidemic phases. Here we consider the two epidemic waves for France, and the chosen period, as well as the parameters values for each period.

	\begin{figure}[H]
		\begin{center}
			\includegraphics[scale=0.2]{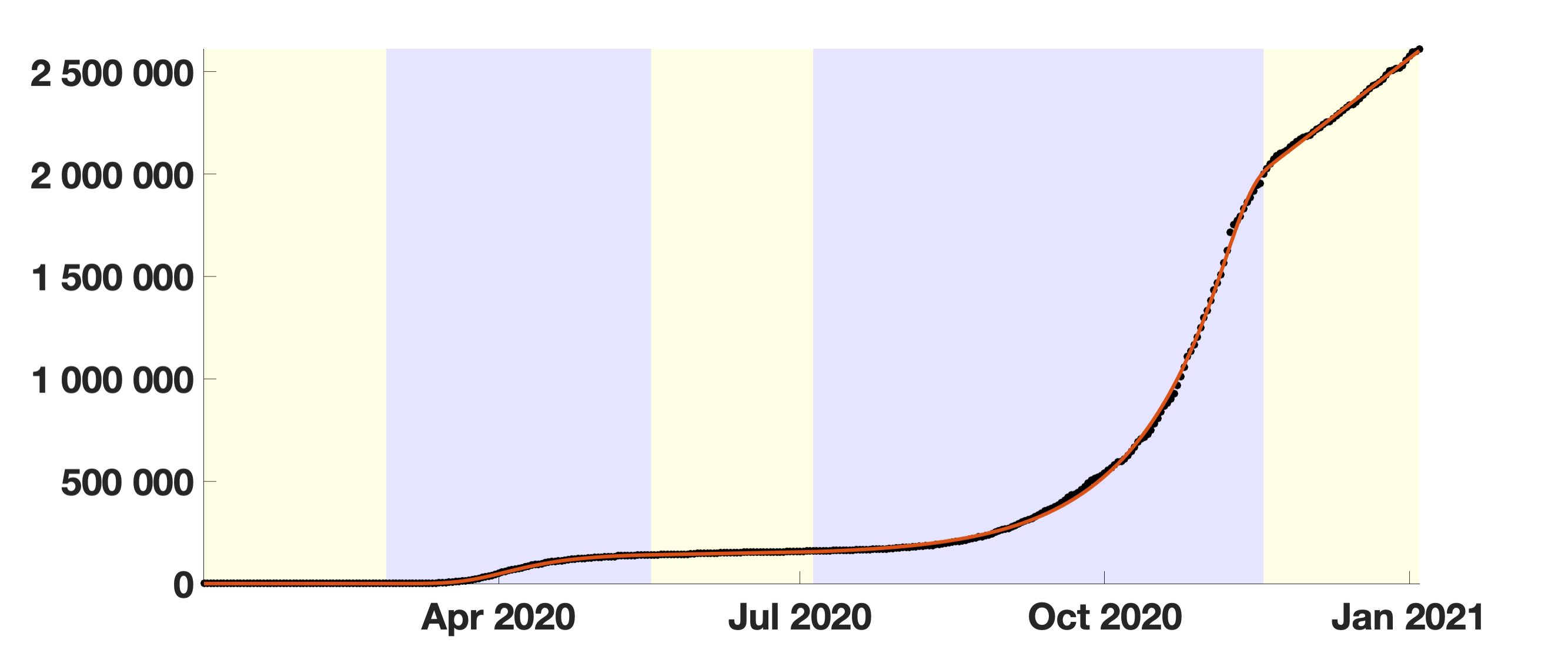} 
		\end{center}
		\caption{\textit{The red curve corresponds to the phenomenological model and the black dots
				correspond to the cumulative number of reported cases in France.}}
		\label{Fig18b}
	\end{figure}
	
	Figure \ref{Fig18c} shows the corresponding daily number of new reported cases data (black dots) and the first
	derivative of our phenomenological model (red curve). 
	\begin{figure}[H]
		\begin{center}
			\includegraphics[scale=0.2]{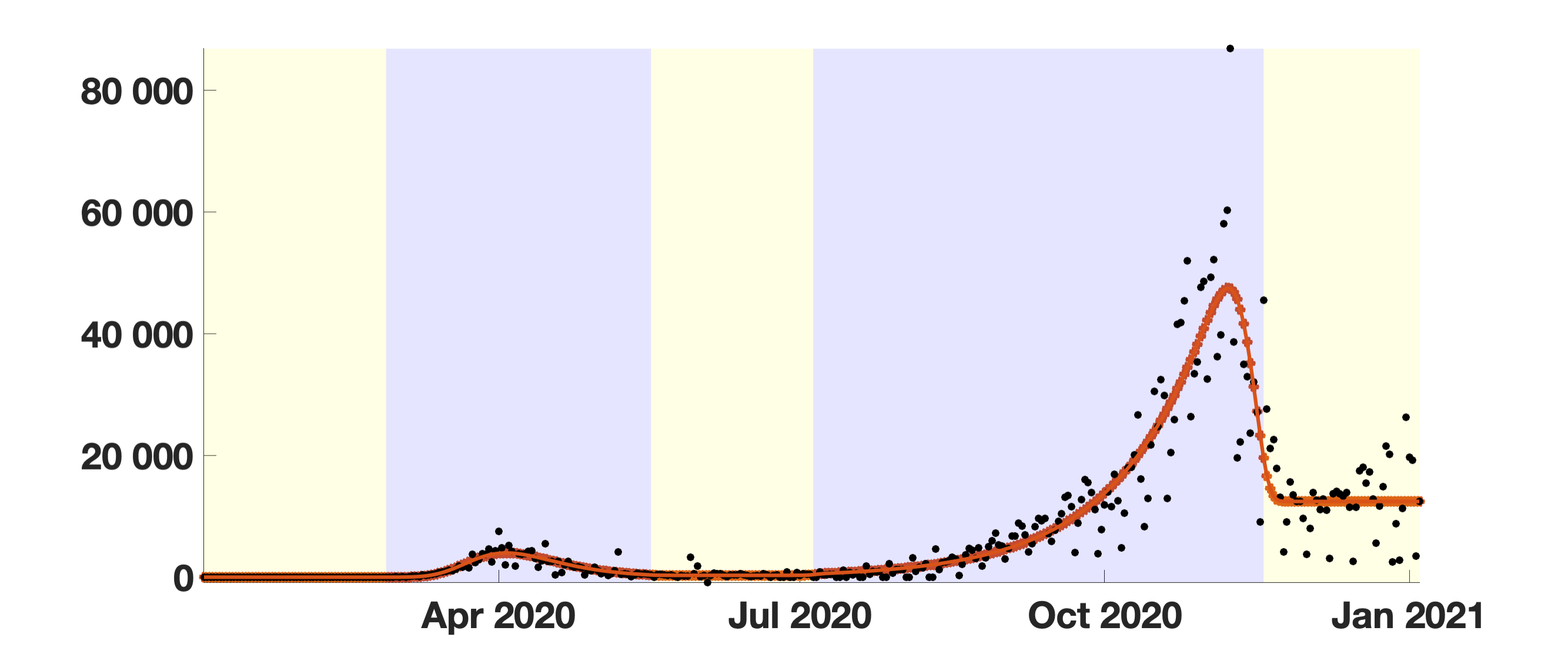} 
		\end{center}
		\caption{\textit{The red curve corresponds to the phenomenological model and the black dots
				correspond to the cumulative number of reported cases in France.}}
		\label{Fig18c}
	\end{figure}
	
	\subsection{Phenomenological Model apply to several countries}
	\label{Section7.3}

	Our method to regularize the data was applied to the eight geographic areas. The resulting curves are presented in Figure \ref{Fig19}. The blue background color regions correspond to epidemic phases and the yellow background color regions to endemic phases. We added a plot of the daily number of cases (black dots) and the derivative of the regularized model for comparison, even though the daily number of cases is not used in the fitting procedure. The figures generally show an excellent agreement between the time series of reported cases (top row, black dots) and the regularized model (top row, blue curve). The match between the daily number of cases (bottom row, black dots) and the derivative of the regularized model (bottom row, blue curve) is also excellent, even though it is not a part of the optimization process. Of course, we lose some information, like the extreme values (“peaks”) of the daily number of cases. This is because we focus on an averaged value of the number of cases. More information could be retrieved by statistically studying the variation around the phenomenological model. However, we leave such a study for future work. The relative error between the regularized curve and the data may be relatively high at the beginning of the epidemic because of the stochastic nature of the infection process and the small number of infected individuals but quickly drops below $1\%$ (see the supplementary material in 	\cite{Griette22} for more details).
	\begin{figure}[H]
		\centering
		\includegraphics[width=\textwidth]{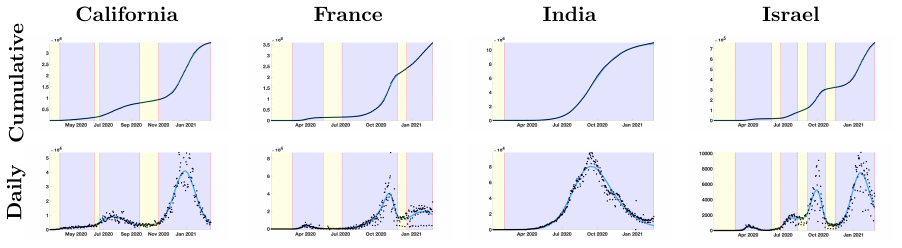}
		\includegraphics[width=\textwidth]{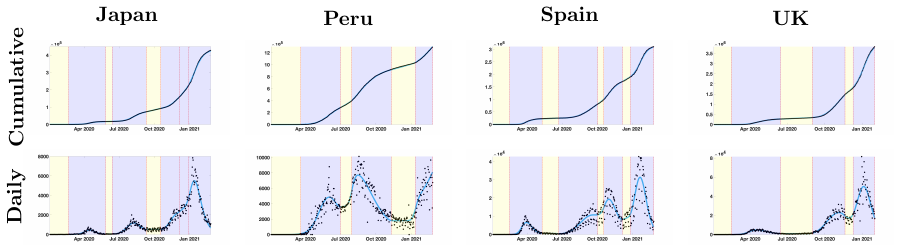}
		\caption{\textit{In the top rows, we plot the cumulative number of reported cases (black dots) and
				the best fit of the phenomenological model (blue curve). In the bottom rows, we plot the daily
				number of reported cases (black dots) and the first derivative of the phenomenological model
				(blue curve). This figure is taken from \cite{Griette22}.}}
		\label{Fig19}
	\end{figure}

	\subsection{Earlier results about transmission rate reconstructed from the data}
	\label{Section7.4}
	
	This problem has already been considered in several articles. In the early 70s, London and Yorke \cite{London73, Yorke73} discussed the time dependent rate of transmission in the context of measles, chickenpox and mumps. Motivated by applications to the data for COVID-19 in \cite{bakhta2020epidemiological} the authors also obtained some new results about reconstructing the rate of transmission. 
	\subsection{Instantaneous reproduction number}
	\label{Section7.5}
	We use the formula \eqref{5.5} to compute the transmission rate, and we  consider the {\bf instantaneous reproduction number} 
	\begin{equation*} \label{15.9}
		{\large \mathbf{	R_e(t) = \tau(t)S(t)/\nu,}}
	\end{equation*}
	and the {\bf quasi-instantaneous reproduction number} 
	\begin{equation*}  \label{15.10}
		{\large \mathbf{	R_e^0(t) = \tau(t)S_0/\nu, }}	
	\end{equation*}
	We compare the above indicators with ${\large \mathbf{R^C_e(t)}}$ the classical notion of {\bf instantaneous reproduction number} \cite{Obadia12, Cori13}.
	
	\begin{remark} The standard method to compute ${\large \mathbf{R^C_e(t)}}$  (see  \cite{Obadia12, Cori13}) proposes another form of regularization of the data, which consists of computing the instant of contamination backward in time. 
		This instant is random and follows a standard exponential law. 
	\end{remark} 
	\subsection{Results}
	\label{Section7.6}
	In Figure \ref{Fig20}, our analysis allows us to compute the transmission rate $\tau(t)$. We use this transmission rate to calculate two different indicators of the epidemiological dynamics for each geographic area,  the instantaneous reproduction number and the quasi-instantaneous reproduction number. Both coincide with the basic reproduction number $R_0$ on the first day of the epidemic. The instantaneous reproduction number at time $t$, $R_e(t)$, is the basic reproduction number corresponding to an epidemic starting at time $t$ with a constant transmission rate equal to $\tau(t)$ and with an initial population of susceptibles composed of $S(t)$ individuals (the number of susceptible individuals remaining in the population). The quasi-instantaneous reproduction number at time $t$, $R^0_e(t)$, is the basic reproduction number corresponding to an epidemic starting at time $t$ with a constant transmission rate equal to $\tau(t)$ and with an initial population of susceptibles composed of $S_0$ individuals (the number of susceptible individuals at the start of the epidemic). The two indicators are represented for each geographic area in the top row of Figure \ref{Fig20}  (black curve: instantaneous reproduction number; green curve: quasi-instantaneous reproduction number).

	\medskip 
	One interpretation for $R_e(t)$ and another for $R^0_e(t)$. The instantaneous reproduction number indicates if, given the current state of the population, the epidemic tends to persist or die out in the long term (note that our model assumes that recovered individuals are perfectly immunized). The quasi-instantaneous reproduction number indicates if the epidemic tends to persist or die out in the long term, provided the number of susceptible is the total population. In other words, we forget about the immunity already obtained by recovered individuals. Also, it is directly proportional to the transmission rate and therefore allows monitoring of its changes. Note that the value of $R_e^0(t)$ changed drastically between epidemic phases, revealing that $\tau(t)$ is far from constant.  In any case, the difference between the two values starts to be visible in the figures one year after the start of the epidemic.
	
	\medskip 
	We also computed the reproduction number using the method described in \cite{Cori13}, which we denoted $R_e^c(t)$. The precise implementation is described in the supplementary material in \cite{Griette22}. It is plotted in the bottom row of Figure \ref{Fig20}  (green curve), along with the instantaneous reproduction number $R_e(t)$ (green curve). 
	
	\begin{remark}
		In the bottom of Figure \ref{Fig20}, we compare the instantaneous reproduction numbers obtained by our method in black and the classical method in \cite{Cori13} in green. We observe that the two approaches are not the same at the beginning. This is because the method of \cite{Cori13} does not consider the initial values  $I_0$ and $E_0$ while we do. Indeed the method of \cite{Cori13} assumes that $I_0$ and $E_0$ are close to $0$ at the beginning when it is viewed as a Volterra equation reformulation of the Bernoulli--Kermack--McKendrick model with the age of infection. On the other hand, our method does not require such an assumption since it provides a way to compute the initial states $I_0$ and $E_0$.
	\end{remark}
	
	\begin{figure}[H]
		\centering
		\includegraphics[scale=0.8]{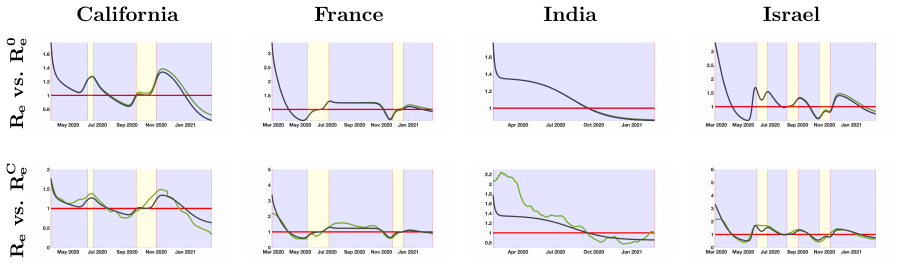}
		\includegraphics[scale=0.8]{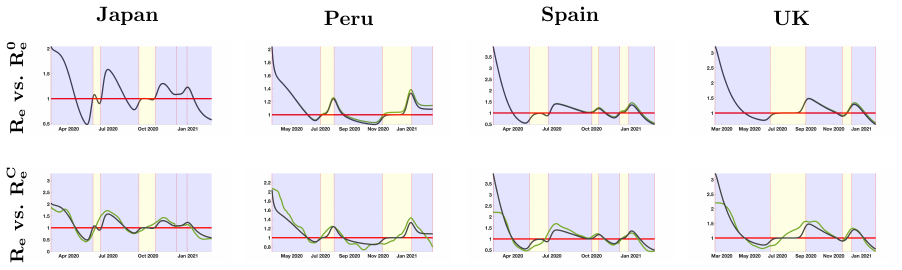}
		\caption{\textit{In the top rows, we plot the instantaneous reproduction number $R_e(t)$ (in black)
				and the quasi instantaneous reproduction number $R^0_e (t)$ (in green). In the bottom rows, we
				plot the instantaneous reproduction number $R_e(t)$ (in black) and the one obtained by the
				standard method \cite{Obadia12, Cori13} $R^C_e(t)$ (in green). This figure is taken from \cite{Griette22}.}}
		\label{Fig20}
	\end{figure}

	It is essential to ``regularize'' the data to obtain a comprehensive outcome from SIR epidemic models. In general, the rate of transmission in the SIR model (applying identification methods) is not very noisy and meaningless. For example, at the beginning of the first epidemic wave, the transmission rate should be decreasing since peoples tend to have less and less contact while to epidemic growth.  The standard regularization methods (like, for example, the rolling weekly average method) have been tested for COVID-19 data in  \cite{Demongeot20b}.  The outcome in terms of transmission rate is very noisy and even negative transmission (which is impossible). Regularizing the data is not an easy task, and the method used is very important in order to obtain a meaningful outcome for the models. Here, we tried several approaches to link an epidemic phase to the next endemic phase. So far, this regularization procedure is the best one.
	
	\medskip 
	Figure \ref{Fig21} illustrates why we need a phenomenological model to regularize the data. On the left-hand side, we observe the $\tau(t)$ becomes negative almost immediately. Therefore, without regularization, the fit may not make sense. 
	\begin{figure}[H]
		\begin{center}
			\textbf{(a)}\hspace{6cm}\textbf{(b)}\\
			\hspace{-0.5cm}	\includegraphics[scale=0.13]{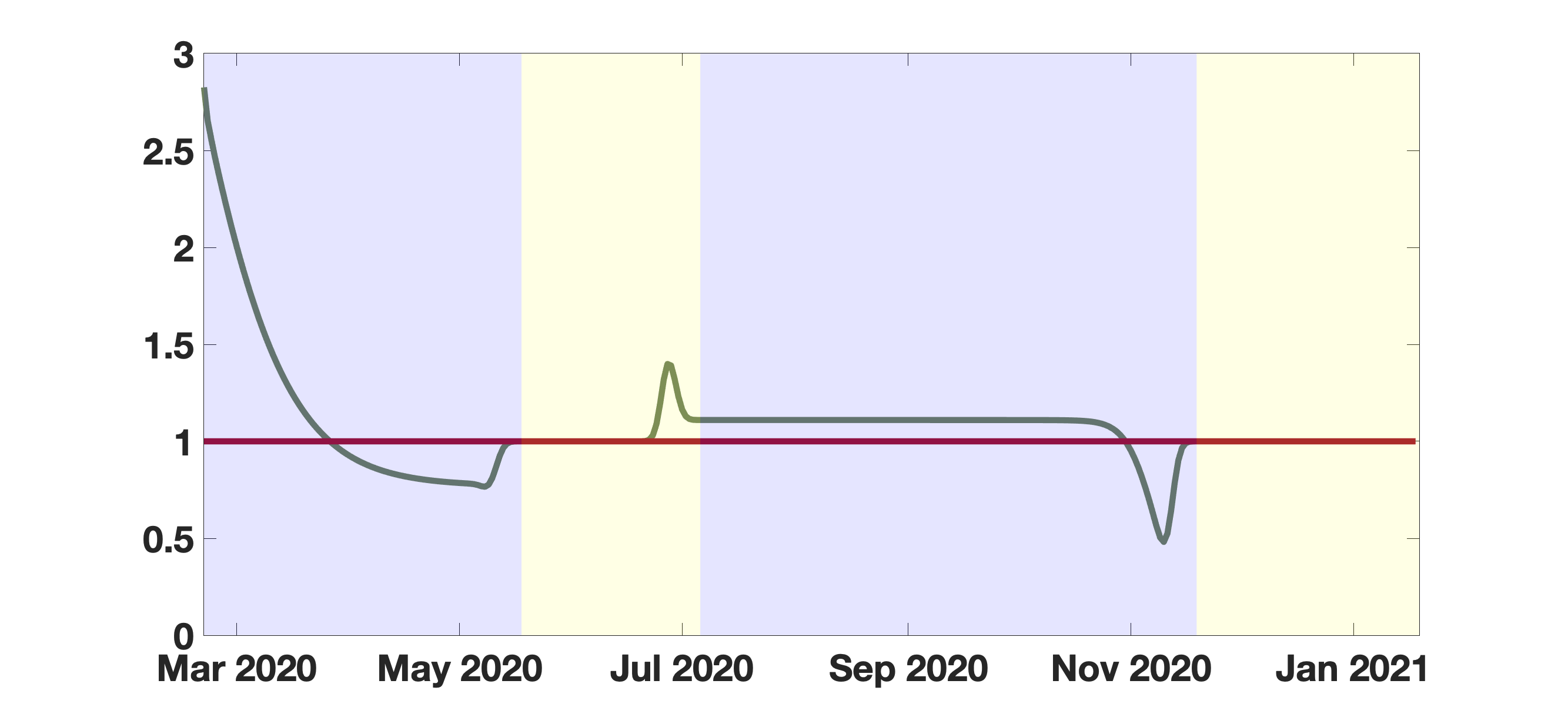} \hspace{0.3cm}
			\includegraphics[scale=0.13]{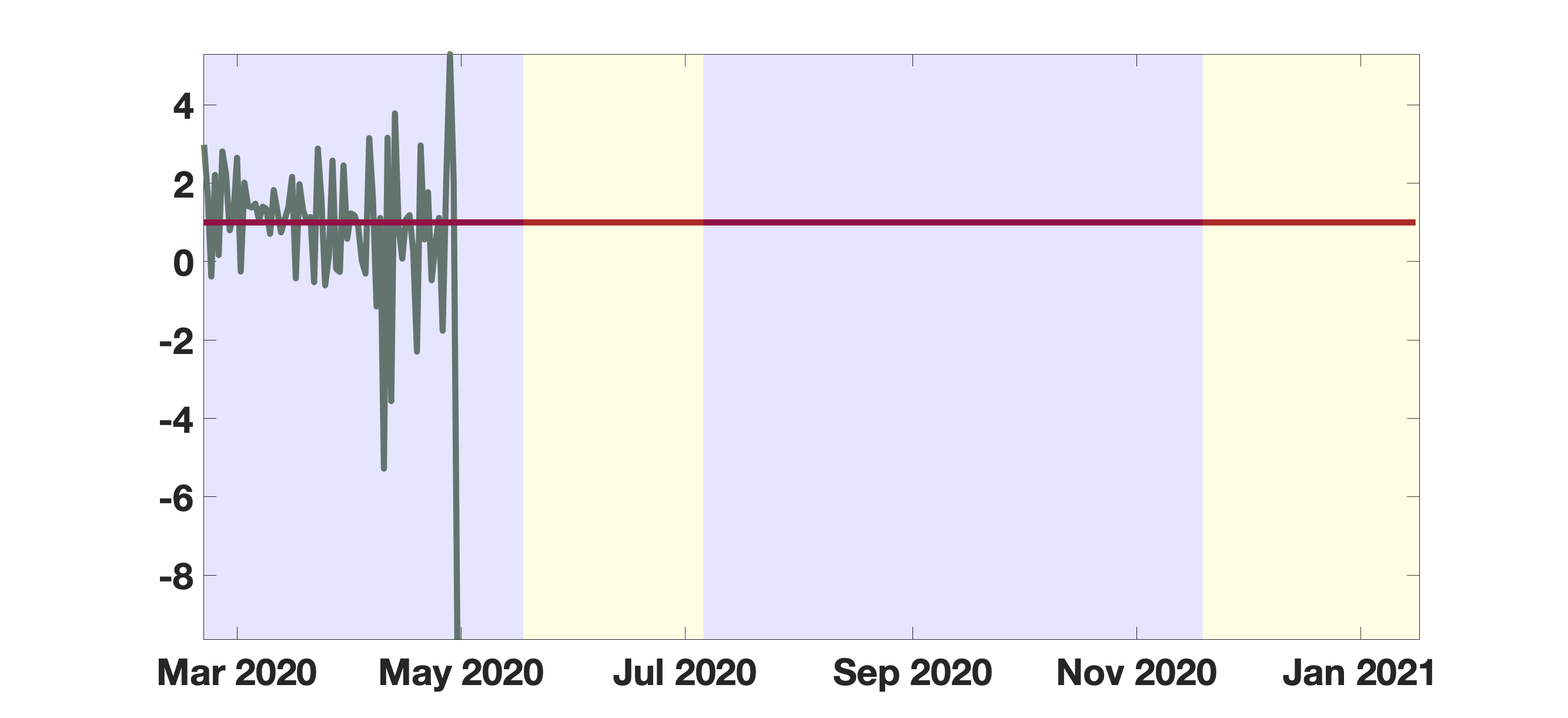} 
		\end{center}
		\caption{\textit{ In this figure, we plot the instantaneous $R_0$. On the left-hand side, we use our smoothing method (with Bernoulli-Verhulst model (endemic) line (endemic) together with a convolution with a Gaussian). On the right-hand side, we use the original cumulative data and our algorithm to fit the cumulative number of cases.}}
		\label{Fig21}
	\end{figure}

	\subsection{Consequences of the results}
	\label{Section7.7}
	In Figure \ref{Fig20}, we saw that the population of susceptible patients is almost unchanged after the epidemic passed. Therefore, the system behaves almost like the non-autonomous system
	$$
	\tcbhighmath[remember as=fx, boxrule=2pt,drop fuzzy shadow=blue]{ I'(t)=\tau(t) S_0 I(t)-\nu I(t), \forall t \geq t_0, \text{ and } I(t_0)=I_0, }
	$$	 
	This means that $I(t) $ depends  linearly on $I_0$. That is, if we multiply $I_0$ by some number, the result $I(t)$ will be multiply by the same number.  
	
	Figure \ref{Fig22} shows two things. The initial number of infected is crucial when we try to predict the number of infected. 
	The average daily number of cases during the endemic phases have strong impact on the amplitude of the next epidemic waves \cite{griette2021robust}. 
	\begin{figure}[H]
		\centering
		\includegraphics[scale=0.2]{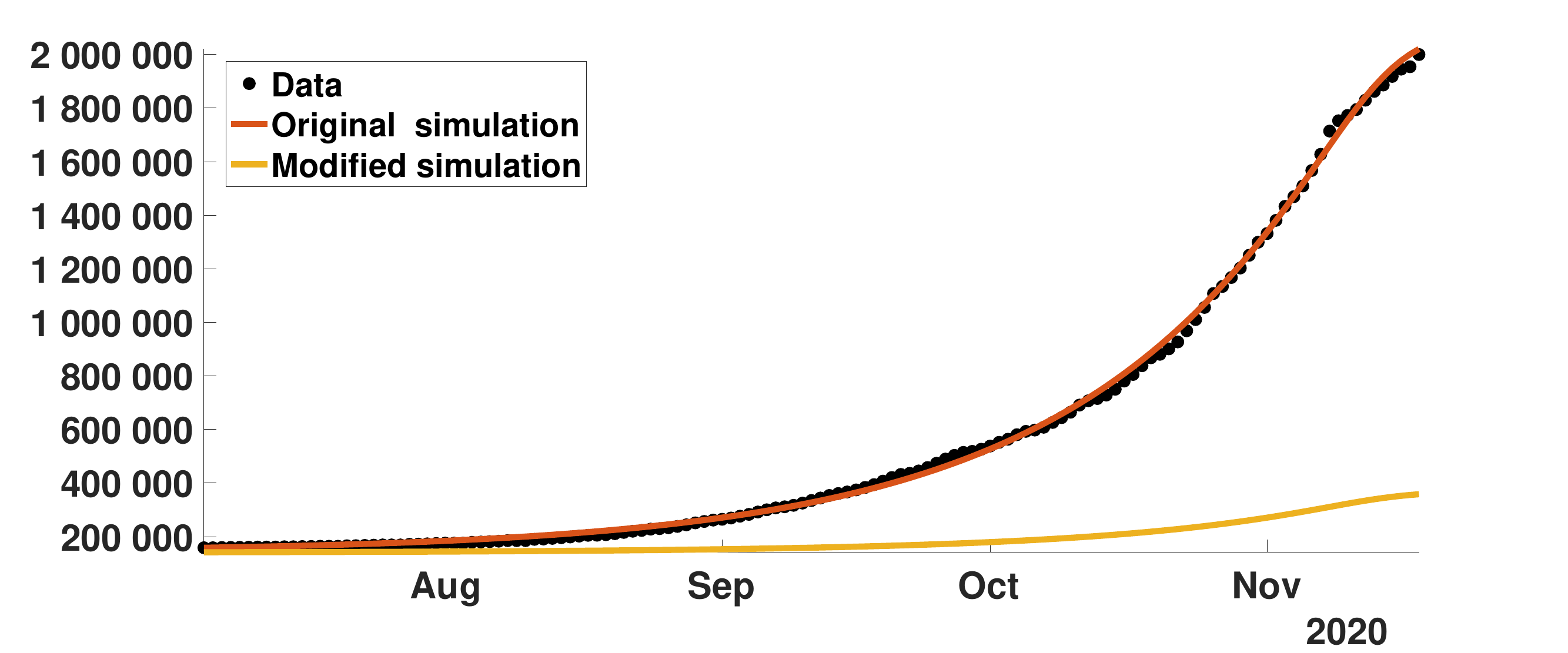} 
		\caption{\textit{We start the simulation at time $t_0=$July 5, 2020  with the initial value $I_0= \dfrac{\CR'(t_2)}{\nu f}$ for red curve and with $I_0=0.41 \dfrac{\CR'(t_2)}{\nu f}$ for yellow curve).}}
		\label{Fig22}
	\end{figure}

	In this section, we obtained a model that covert the changes of regimen (from endemic to epidemic and conversely). Moreover the detection of the changes of regimen between epidemic wave and endemic period is still difficult to detect. An attempt to study this question can be found in  \cite{Demongeot23}.

	\section{Exponential phase with more compartments}
	\label{Section8}
	\subsection{A model with transmission from the unreported infectious}
	\label{Section8.1}
	We consider a model with unreported infection individuals.  
	\begin{equation} \label{8.1}
		\left\lbrace
		\begin{array}{l}
			S'(t)=-\tau(t) S(t) \left(   I(t)+U(t)   \right) , \vspace{0.2cm}\\
			I'(t)=\tau(t) S(t) \left(   I(t)+U(t)   \right)  - \nu I(t), \vspace{0.2cm}\\
			U'(t)=\nu \, \left(1-f\right) \,  I(t)-\eta U(t),
		\end{array}
		\right.
	\end{equation}
	for $t \geq t_0$, and with initial distribution
	\begin{equation}\label{8.2}
		S(t_0)=S_0, \,	I(t_0)=I_0, \, \text{ and } U(t_0)=U_0. 
	\end{equation}
	
	The epidemic model associated with the flowchart in Figure \ref{Fig23A} applies to influenza outbreaks in \cite{Arino06}, hepatitis A outbreaks in \cite{Regan16}, and COVID-19 in \cite{Liu20a}. 
	\begin{figure}[H]
		\begin{center}		
			\includegraphics[scale=0.6]{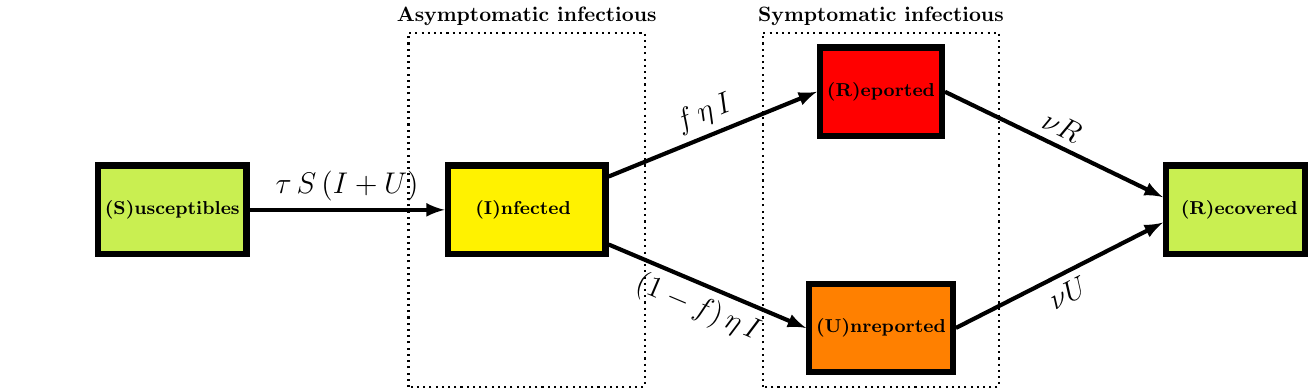}
		\end{center}
		\caption{ \textit{Flowchart.}}\label{Fig23A}
	\end{figure}
	\subsection{The exponential phase approximation}
	\label{Section8.2}
	We assume that $S(t)$ is constant, and equal to $S_0$, and $\tau(t)$ remains constant equal to $\tau_0=\tau(t_0)$. The consider for example the case of a single age group, we obtain the following model which was first considered for COVID-19
	\begin{equation} \label{8.3}
		\left\lbrace
		\begin{array}{l}
			I'(t)=\tau S_0 \left(   I(t)+U(t)   \right)  - \nu I(t), \vspace{0.2cm}\\
			U'(t)=\nu \, \left(1-f\right) \,  I(t)-\eta U(t),
		\end{array}
		\right.
	\end{equation}
	for $t \geq t_0$, and with initial distribution
	\begin{equation}\label{8.4}
		I(t_0)=I_0, \, \text{ and } U(t_0)=U_0. 
	\end{equation}
	
	We can reformulate this system using a matrix formulation
	\begin{equation*}
		\left(\begin{array}{c}
			I'(t)\\
			U'(t)
		\end{array}\right)
		=A
		\left(\begin{array}{c}
			I(t)\\
			U(t)
		\end{array}\right), \forall t \in [t_0,t_1],
	\end{equation*}
	where 
	\begin{equation*}
		A=\left(\begin{array}{cc}
			\tau \, S_0-\nu  & \tau \, S_0\\
			\nu  \left(1-f\right) &-\eta 
		\end{array}\right).
	\end{equation*}
	Then the matrix $A$ is \textbf{irreducible} if and only if 
	$$
	\nu  \left(1-f\right)>0 \text{ and }\tau \, S_0>0.
	$$ 
	Remember 	the model  \eqref{8.3}  to connect the data and the epidemic model
	\begin{equation*}
		\CR'(t)= f \,    \nu \, I(t),   \text{ for }  t \geq t_0.
	\end{equation*}
	Consider the \textbf{exponential phase of the epidemic}. That is, 
	$$
	\CR'(t)= \chi_1 \chi_2e^{\chi_2 t}, \forall t \in [t_0, \tau+t_0], 
	$$
	for some $\tau>0$.  
	\medskip 
	Combining the two previous equations, we obtain 
	$$
	f \,    \nu \, I(t)=  \chi_1e^{\chi_2 (t-t_0)}, \forall t \in [t_0, \tau+t_0], 
	$$
	Remember that $\chi_1$ and $\chi_2$ are computed by using the data.  	 More precisely, these parameters are obtained by fitting $t \to \chi_1e^{\chi_2 t}- \chi_3$ to the cumulative number of cases data during a period of time $[t_0, \tau+t_0]$.

	We can rewrite $	f \,    \nu \, I(t)= \chi_1e^{\chi_2 t} $	by using an inner product  
	\begin{equation*} 
		\left\langle y_0, 	
		\left(\begin{array}{c}
			I(t)\\
			U(t)
		\end{array}
		\right)\right\rangle 
		=	\chi_1 e^{\chi_2 t}, \text{ with } 	y_0= \left(
		\begin{array}{c}
			\nu \, f \\
			0
		\end{array}
		\right),	
	\end{equation*}
	where $	\left\langle ., .\right\rangle$ is the Euclidean inner product defined  in dimension $2$ as
	\begin{equation*}
		\left\langle x, y\right\rangle= x_1 y_1 + x_2 y_2.  
	\end{equation*}
	
	The following theorem is proved in Appendix \ref{AppendixA}. 	
	\begin{theorem} \label{TH8.1} Let $\chi_1>0$, $\chi_2>0$, and $\tau>0$. Let $A$ be a $n$ by $n$ real matrix. Assume that the off-diagonal elements of $A$ are non-negative, and $A$ is irreducible.   Assume that  there exist two vectors $y_0 >0$, and $x_0 >0$ such that 
		
		$$
		\hspace{-1cm}	\text{(Linear model)	}	\hspace{3.2cm}	\dot{x}(t)=Ax(t), \text{ and } x(0)=x_0, 
		$$
		satisfies 
		\begin{equation*} 
			\hspace{-1cm} \text{(Connection with the data)	}		\hspace{1cm}	\langle	y_0, x(t) \rangle =\chi_1 e^{\chi_2 t },\forall t \in \left[0, \tau \right],
		\end{equation*}
		where $\left\langle x, y\right\rangle$ is the Euclidean inner product. 
		
		\medskip 
		Then $\chi_2$ must be the dominant eigenvalue $A$ (i.e.,  the one with the largest real part). Moreover, we can choose a vector $x_0 \gg 0$ (i.e., with all its components strictly positive), satisfying
		$$
		Ax_0=\chi_2 x_0.
		$$ 
		Multiplying $x_0$ by a suitable positive constant, we obtain 
		$\langle	y_0, x_0 \rangle =\chi_1$, and we obtain
		
		$$
		\langle	y_0, x(t) \rangle =\chi_1 e^{\chi_2 t },\forall t \in \left[0, \tau \right].
		$$ 
	\end{theorem}

	\medskip 
	
	\noindent  Returning back to the example of epidemic model with unreported cases, we must find $I(0)>0$ and $U(0)>0$ such that 
	$$
	\left(\begin{array}{cc}
		\tau \, S_0-\nu  & \tau \, S_0\\
		\nu  \left(1-f\right) &-\eta 
	\end{array}\right)	\left(\begin{array}{c}
		I(0)\\
		U(0)
	\end{array}
	\right)=\chi_2	\left(\begin{array}{c}
		I(0)\\
		U(0)
	\end{array}
	\right). 
	$$
	After a few computations (see the supplementary in Liu et al. \cite{Liu20a}), we obtain 
	\begin{equation}
		\label{8.5}
		\tau=\frac{\chi_2+\nu}{ S_0 } \dfrac{\eta+\chi_2 }{\nu(1-f)+\eta+\chi_2 },
	\end{equation}
	and
	\begin{equation}
		\label{8.6}
		U_0= \dfrac{\nu(1-f)}{\eta+\chi_2} I_0=\dfrac{(1-f)\nu}{\eta+\chi_2} I_0.
	\end{equation}
	
	\begin{remark}
		\label{REM8.2}
		Let $\chi_1>0$, $\chi_2>0$, $\phi_1>0$, $\phi_2>0$, and $\tau>0$. 	Assume that $x_0 >0$, $y_0 >0$ and $z_0 >0$ satisfy
		\begin{equation*}
			\dot{x}(t)=Ax(t), \text{ and } x(0)=x_0, 
		\end{equation*}
		and
		\begin{equation*}
			\langle	y_0, x(t) \rangle =\chi_1 e^{\chi_2 t },\forall t \in \left[0, \tau \right],
		\end{equation*}
		\begin{equation*}
			\langle	z_0, x(t) \rangle =\phi_1 e^{\phi_2 t },\forall t \in \left[0, \tau \right]. 
		\end{equation*}
		If $\chi_2  \neq \phi_2 $ the matrix $A$ must be reducible.  That is, up to a re-indexation of the components of $x(t)$, the matrix $A$ reads as 
		$$
		A= \left(
		\begin{array}{cc}
			A_{11} & 0\\
			A_{21} & 	A_{22}  
		\end{array}
		\right)
		$$
		where $A_{ij}$ are block matrices. The matrix $A$ presents a weak coupling between the last block's components and the first block's components. 
	\end{remark}	
	
	\subsection{Uncertainty due to the period chosen to fit the data}
	\label{Section8.3}

	The principle of our method is the following. By using an exponential best fit method we obtain a best fit of \eqref{5.3} to the data over a time $[t_1,t_2]$ and we derive the parameters $\chi_1$ and $\chi_2$. The values of $I_0$ $U_0$,  and $\tau_0$ are obtained by using 	\eqref{5.4}, \eqref{8.2} and \eqref{8.3}. Next, we use 
	$$
	\tau(t)= \tau_0 e^{-\mu (t-N)^+}, 
	$$
	we fix $N$ (first day of public intervention) to some value and we obtain $\mu$ by trying to get the best fit to the data.

	In the method the uncertainty in our prediction is due to the fact that several sets of parameters $(t_1,t_2,N,f)$ may give a good fit to the data. As a consequence, at the early stage of the epidemics (in particular before the turning point) the outcome of our method can be very different from one set of parameters to another. We try to solve this uncertainty problem by using several choices of the period to fit an exponential growth of the data to determine $\chi_1$ and $\chi_2$ and several choices for the first day of intervention $N$. So in this section, we vary the time interval $[t_1,t_2]$,  during which we use the data to obtain $\chi_1$ and $\chi_2$ by using an exponential fit. In the simulations below, the first day $t_1$ and the last day $t_2$ vary such that
	$$
	\text{Earliest day } \leq t_1 \leq t_2 \leq \text{Latest day}.
	$$
	We also vary the first day of public intervention:
	$$
	\text{Earliest first day of intervention} \leq N \leq \text{Latest first day of intervention}.
	$$
	
	We vary $f$ between $0.1$ to $0.9$. For each $(t_1,t_2,\nu,f,\eta,\mu,N)$ we evaluate $\mu$ to obtain the best fit of the model to the data. We use the mean absolute deviation as the distance to data to evaluate the best fit to the data.
	We obtain a large number of best fit depending on $(t_1,t_2,\nu,f,\eta,\mu,N)$ and we plot smallest mean absolute deviation $\rm{MAD}_{\min}$. Then we plot all the best fit graphs  with mean absolute deviation between $\rm{MAD}_{\min}$ and $\rm{MAD}_{\min}+40$.
	
	\medskip 
	The figure below is taken from Liu et al. \cite{Liu21}. 
	\begin{figure}
		\begin{center}
			\textbf{(a)} \hspace{7cm} 	\textbf{(b)}  \\
			\includegraphics[scale=.15]{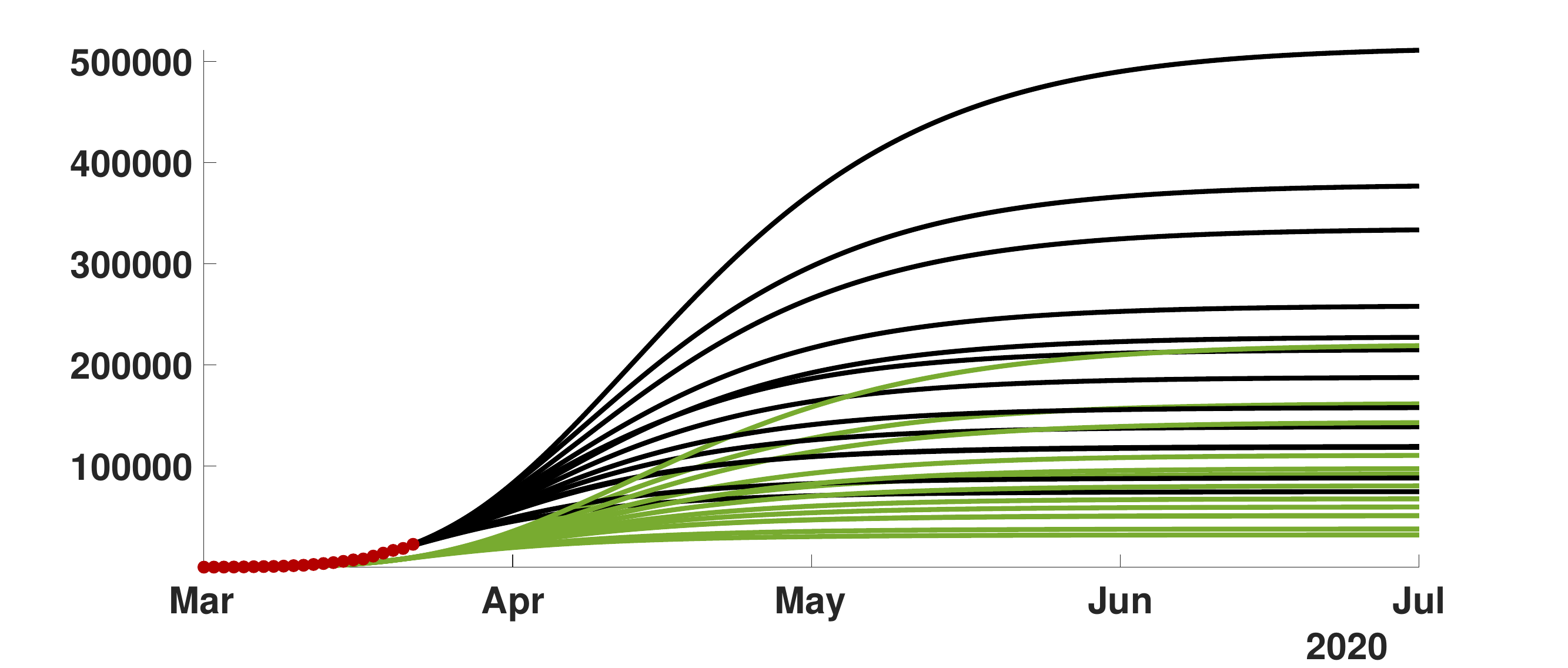}\includegraphics[scale=.15]{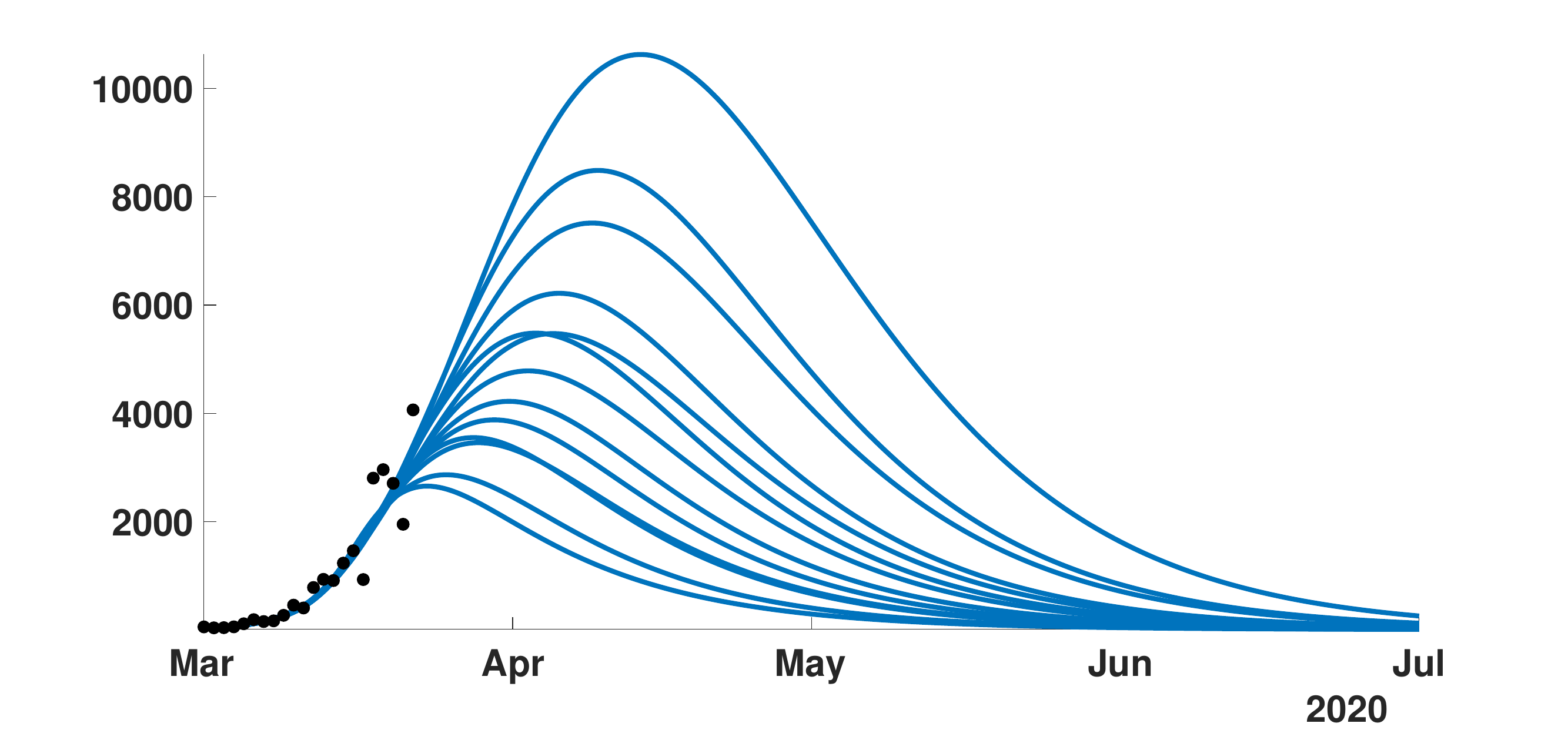}\\
			\textbf{(c)} \hspace{7cm} 	\textbf{(d)}  \\
			\includegraphics[scale=.15]{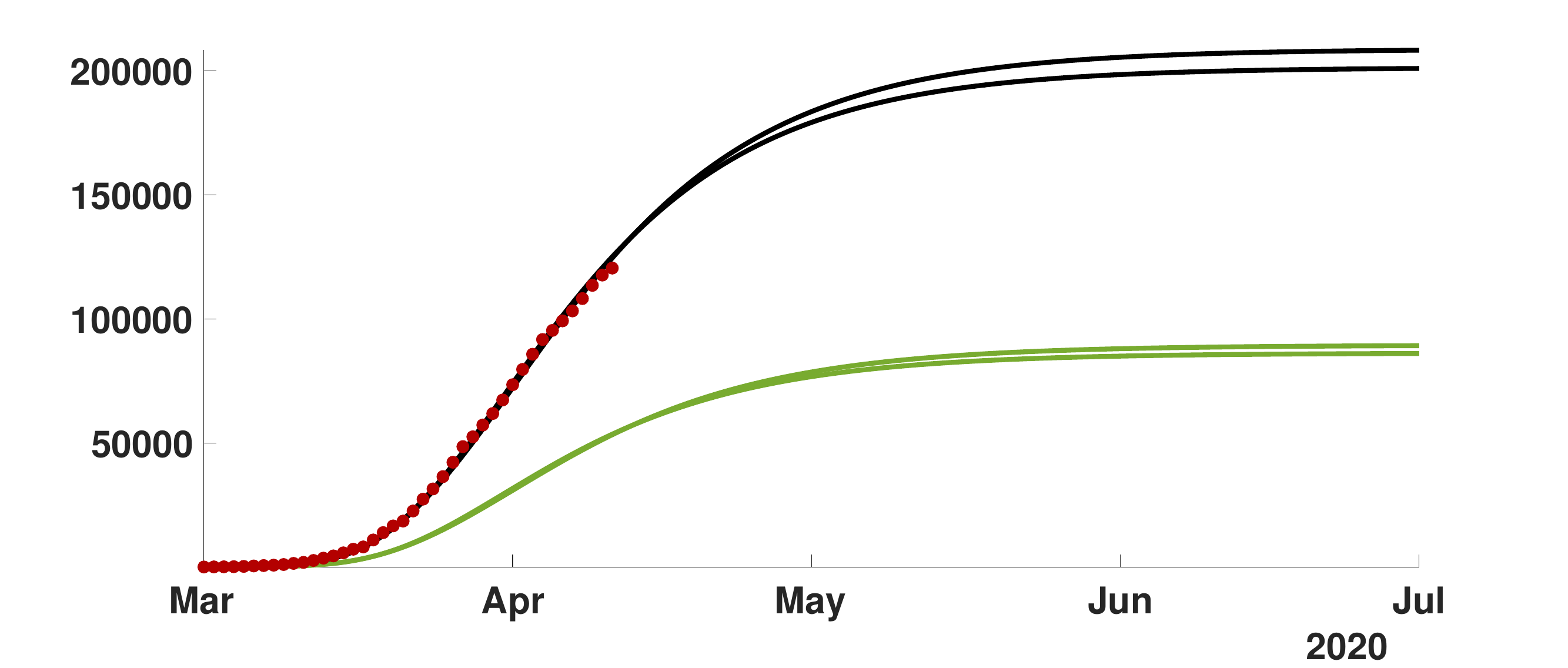}\includegraphics[scale=.15]{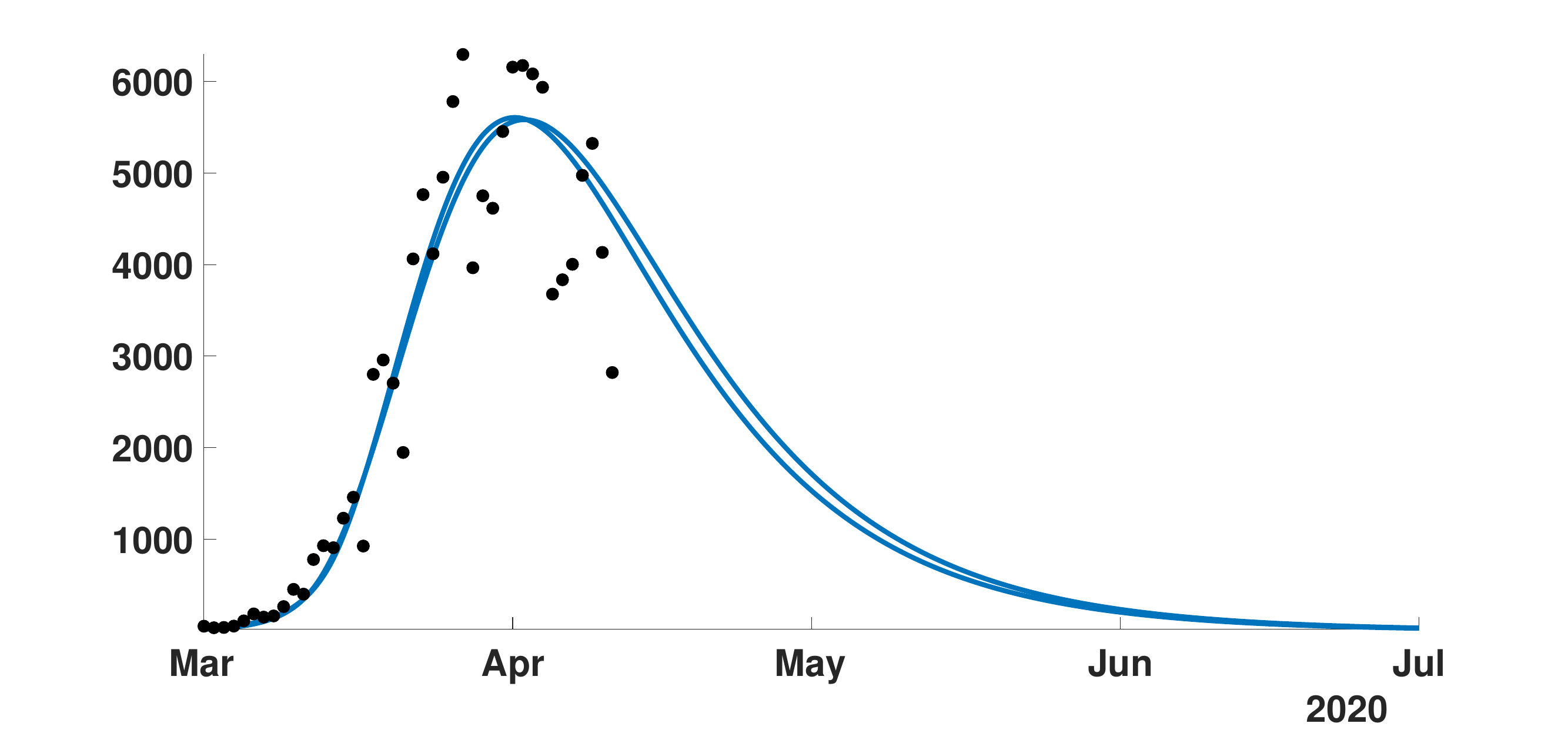}\\
			\textbf{(e)} \hspace{7cm} 	\textbf{(f)}  \\
			\includegraphics[scale=.15]{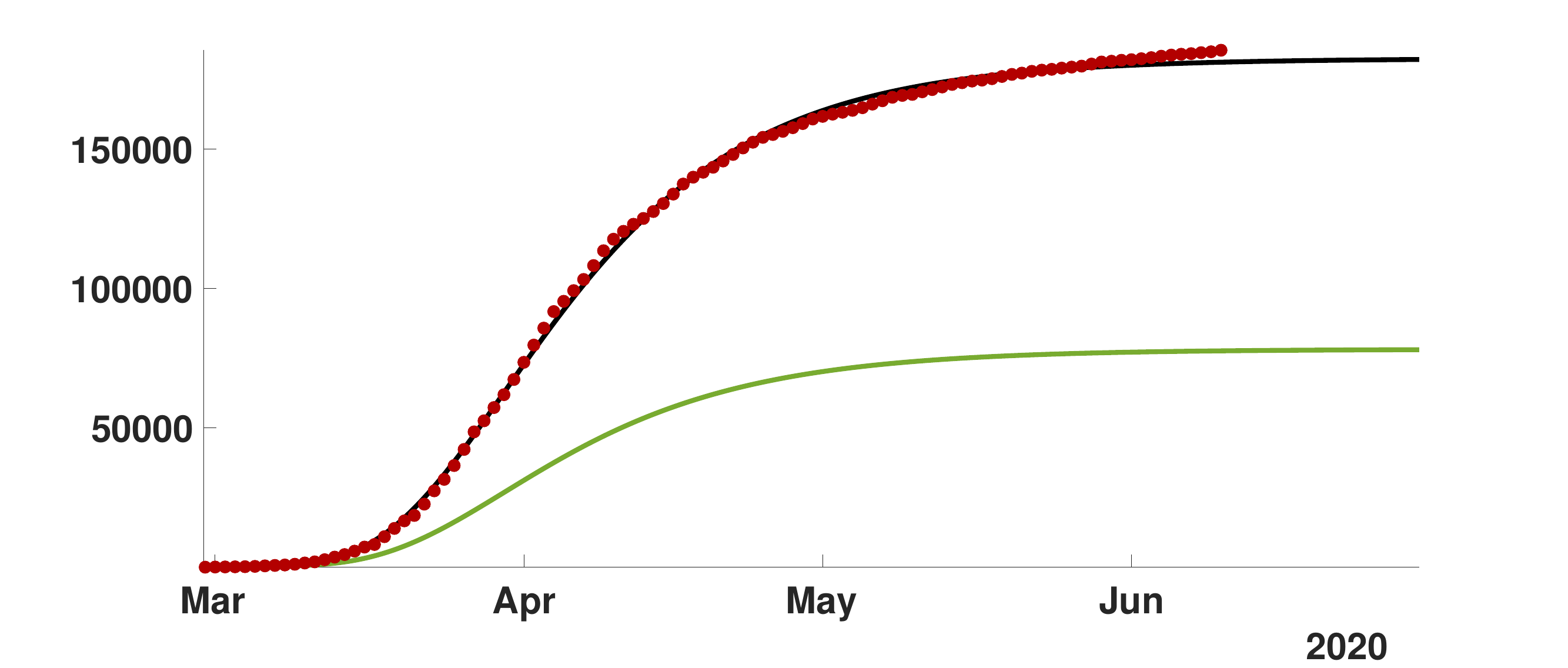}\includegraphics[scale=.15]{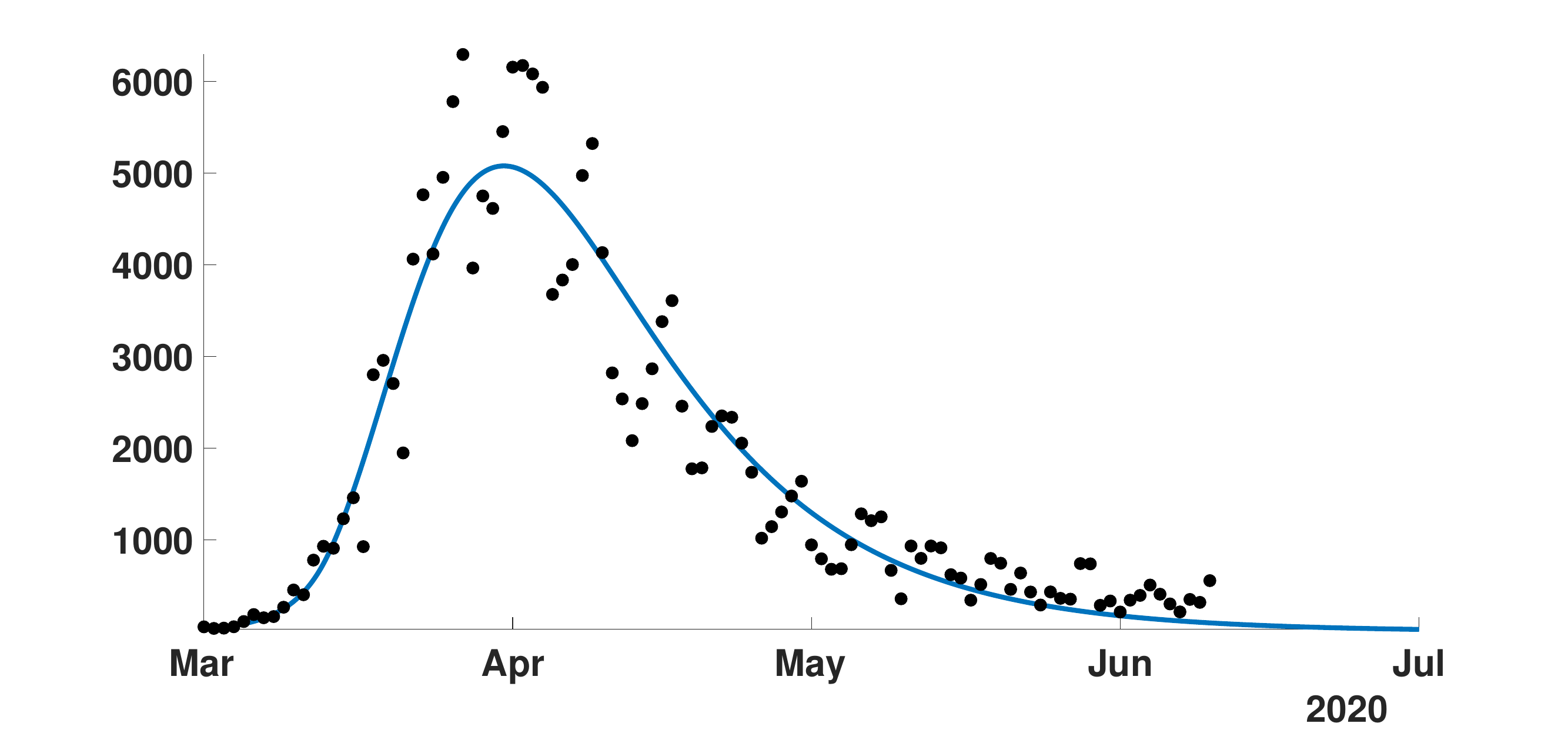}
		\end{center}
		\caption{\textit{In this figure, we consider the data for Germany. We plot the cumulative number of cases of the left hand side and the daily number of cases on the right hand side. In (a) and (b) we use the data until March $22$. In (c) and (d)  we use the data until April $11$. (e) and (f)  we use the data until June $10$.
		}}
		\label{Fig23}
	\end{figure}
	
	\section{Modeling COVID-19 epidemic with age groups}
	\label{Section9}
	
	This section considers an epidemic whenever the population is divided into age groups. Here, age means the chronological age, which is nothing but the time since birth.

	\subsection{Epidemic model with  age groups}
	\label{Section9.1}
	
	The epidemic model with age structure and unreported cases reads as follows, for each $t \geq  t_0,$ 
	\begin{equation*} 
		\left\{
		\begin{array}{c}
			S_1'(t)=- \tau_{1} S_{1}(t) \bigg[\phi_{11} \dfrac{  I_1(t)+U_1(t) }{N_1}+\ldots +\phi_{1n}\dfrac{  I_n(t)+U_n(t) }{N_n} \bigg]  ,\\
			\vdots\\
			S_n'(t)=- \tau_{n} S_{n}(t)  \bigg[ \phi_{n1} \dfrac{  I_1(t)+U_1(t) }{N_1} + \ldots +\phi_{nn}\dfrac{ I_n(t)+U_n(t) }{N_n} \bigg], 
		\end{array}
		\right.
	\end{equation*}
	\begin{equation*} 
		\left\{
		\begin{array}{c}
			I_1'(t)=\tau_{1} S_{1}(t) \bigg[\phi_{11} \dfrac{  I_1(t)+U_1(t) }{N_1}+\ldots +\phi_{1n}\dfrac{  I_n(t)+U_n(t) }{N_n} \bigg]  - \nu I_1(t),\\
			\vdots\\
			I_n'(t)=\tau_{n} S_{n}(t)  \bigg[ \phi_{n1} \dfrac{  I_1(t)+U_1(t) }{N_1} + \ldots +\phi_{nn}\dfrac{ I_n(t)+U_n(t) }{N_n} \bigg] - \nu I_n(t),
		\end{array}
		\right.
	\end{equation*}
	
	and 
	\begin{equation*} 
		\left\lbrace
		\begin{array}{c}
			U_1'(t)=\nu^1_2 \,  I_1(t)-\eta U_1(t),\\
			\vdots\\
			U_n'(t)=\nu^n_{2} \,  I_n(t)-\eta U_n(t),
		\end{array}
		\right.
	\end{equation*}
	with the initial values 
	$$
	S_i (t_0)=S^0_i , I_i (t_0)=I^0_i , \text{ and } U_i (t_0)=U^0_i , \forall i=1, \ldots, n. 
	$$
	
	\subsection{Cumulative reported cases with age  structure in Japan}
	\label{Section9.2}
	We first choose two days $d_1$ and $d_2$ between which each cumulative age group grows like an exponential. By fitting the cumulative age classes $[0,10[$,$[10,20[$, \ldots and $[90,100[$ between $d_1$ and $d_2$,  for~each age class $j=1,\ldots 10$ we can find $\chi^j_1$,  $\chi^j_2$  and $\chi^j_3$  
	$$
	\CR^{data}_j(t)  \simeq  \chi^j_1 \, e^{\chi^j_2t } -\chi^j_3.
	$$
	
	We obtain 
	\begin{equation} \label{9.1}
		\left\lbrace 
		\begin{array}{c}
			\CR_1(t)= \chi^1_1 \, e^{\chi^1_2 t } -\chi^1_3, \\
			\vdots\\
			\CR_n(t)= \chi^n_1 \, e^{\chi^n_2 t } -\chi^n_3,
		\end{array}
		\right.
	\end{equation}
	where 
	$$
	\chi^i_j \geq 0,\forall i=1,\ldots,n, \,\forall j=1,2,3.  
	$$
	In Figures \ref{Fig24}-\ref{Fig25}, the growth rate of the exponential fit depends on the age group \cite{Griette20}. 
	In Figures \ref{Fig24}-\ref{Fig25},  we see the similarity of dynamical behavior at the two extreme age groups $[0,20]$ and $[70,100]$.

	\begin{figure}[H]
		\centering
		\includegraphics[scale=0.15]{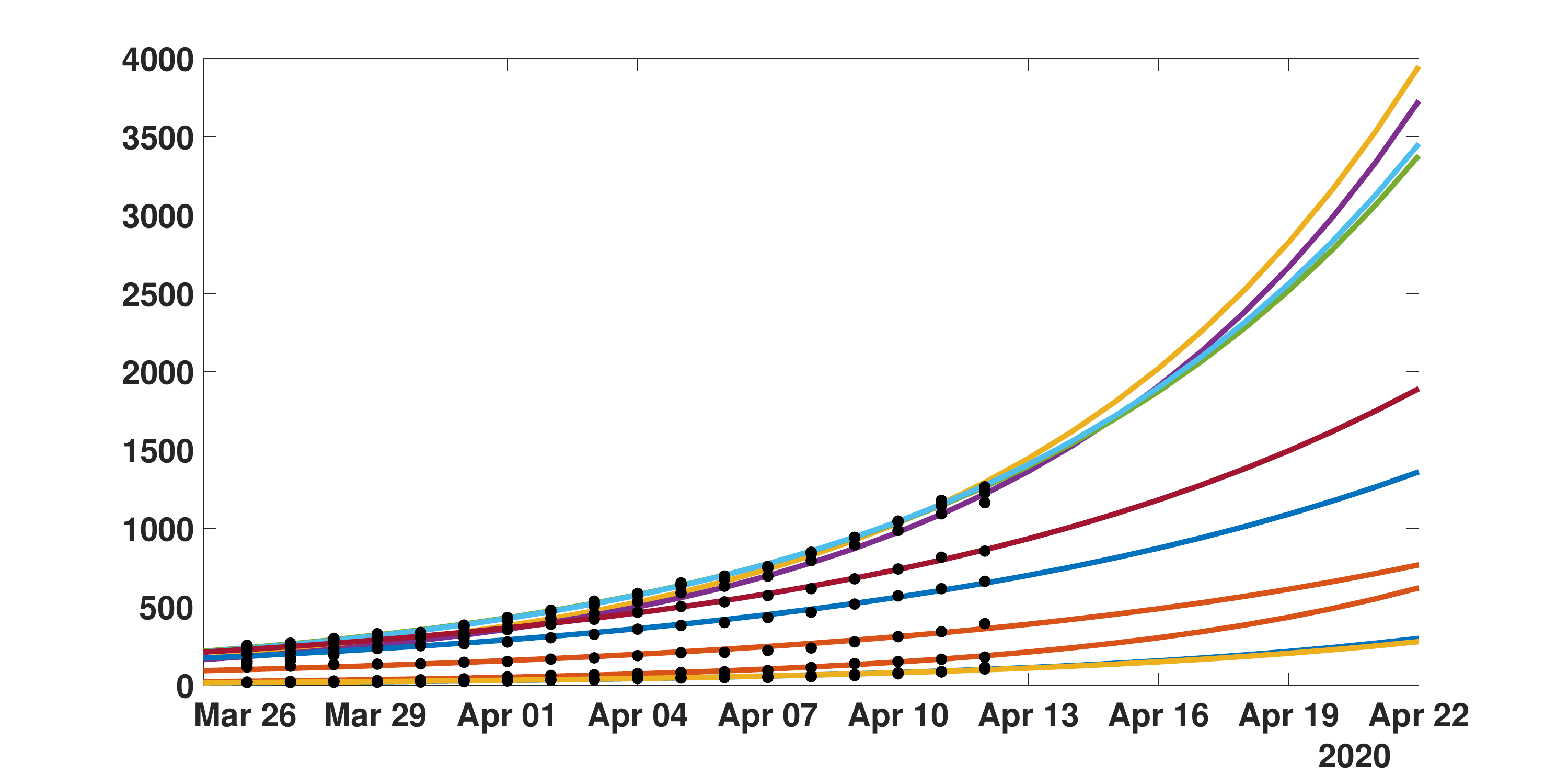}
		\caption{\textit{In this figure, we plot an exponential fit to the cumulative data for each age groups $[0,10[$,$[10,20[$, \ldots and $[90,100[$   in Japan.}}
		\label{Fig24}
	\end{figure}
	
	\begin{figure}[H]
		\centering
		\includegraphics[scale=0.15]{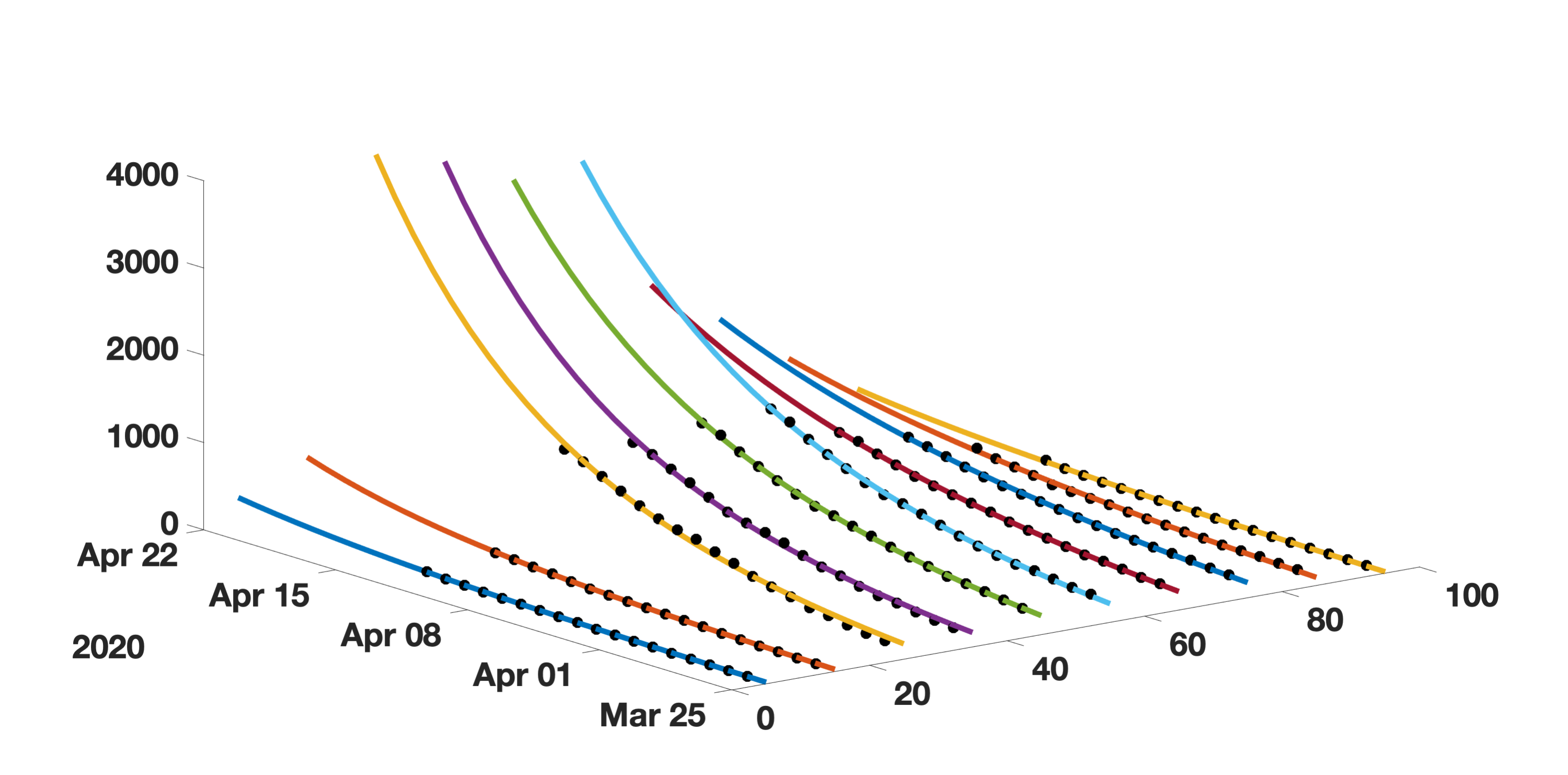}
		\caption{\textit{In this figure, we plot an exponential fit to the cumulative data for each age groups $[0,10[$,$[10,20[$, \ldots and $[90,100[$   in Japan.}}
		\label{Fig25}
	\end{figure}

	\subsection{Method to Fit of the Age Structured Model to the Data}
	\label{Section9.3}
	
	By assuming that the number of  susceptible individuals remains constant we have for each $t \geq  t_0,$ 
	\begin{equation*} 
		\left\{ 
		\begin{array}{c}
			I_1'(t)=\tau_{1} S_{1} \bigg[\phi_{11} \dfrac{  I_1(t)+U_1(t) }{N_1}+\ldots +\phi_{1n}\dfrac{  I_n(t)+U_n(t) }{N_n} \bigg]  - \nu I_1(t),\\
			\vdots\\
			I_n'(t)=\tau_{n} S_{n} \bigg[ \phi_{n1} \dfrac{  I_1(t)+U_1(t) }{N_1} + \ldots +\phi_{nn}\dfrac{ I_n(t)+U_n(t) }{N_n} \bigg] - \nu I_n(t),
		\end{array}
		\right.
	\end{equation*}
	and 
	\begin{equation} \label{9.2}
		\left\lbrace
		\begin{array}{c}
			U_1'(t)=\nu^1_2 \,  I_1(t)-\eta U_1(t),\\
			\vdots\\
			U_n'(t)=\nu^n_{2} \,  I_n(t)-\eta U_n(t),
		\end{array}
		\right.
	\end{equation}
	with the initial values 
	$$
	I_i (t_0)=I^0_i , \text{ and } U_i (t_0)=U^0_i , \forall i=1, \ldots, n. 
	$$
	
	\subsection{Rate of contact}
	\label{Section9.4}
	The values in Figure \ref{Fig29} describe the contact rates between age groups. The values used are computed from the values obtained in  \cite{Prem17}.
	
	\begin{figure}
		\begin{center}
			%			\textbf{(a)} \hspace{6cm} \textbf{(b)} \\
			%	\vspace{0.2cm}
			\includegraphics[scale=0.8]{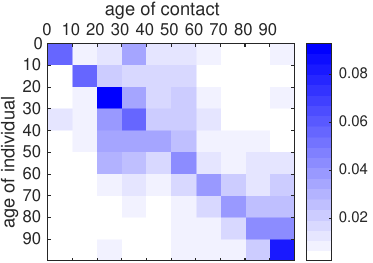}
			%\includegraphics[scale=1]{GMS-Biology-2020/Figure12B.pdf} \\
			%	\vspace{0.2cm}
			%	\textbf{(c)} \hspace{6cm} \textbf{(d)} \\
			%	\vspace{0.2cm}
			%	\includegraphics[scale=1]{GMS-Biology-2020/Figure12C.pdf}
			%	\includegraphics[scale=1]{GMS-Biology-2020/Figure12D.pdf}
			\caption{\textit{For each age class in the $y$-axis we plot the rate of contacts between one individual of this age class and another individual of the age class indicated on the $x$-axis.  The figure represents  the rate of contacts before the start of public measures (April 11).}}\label{Fig29}
		\end{center}
	\end{figure}

	We assume that 
	\begin{equation*} 
		\left\{ 
		\begin{array}{c}
			\CR_1(t)'= \nu^1_{1} I_1(t), \\
			\vdots\\
			\CR_n(t)'= \nu^n_{1} I_n(t), 
		\end{array}
		\right.
	\end{equation*}
	where
	$$
	\nu^i_1=\nu\, f_i, \text{ and } \nu^i_2 =\nu \, (1-f_i), \forall i=1,\ldots,n.
	$$
	Therefore, we obtain 
	\begin{equation*}
		I_j(t)=I^\star_j e^{\chi^j_2 t} ,
	\end{equation*}
	where 
	\begin{equation*}
		I^\star_j:=\dfrac{\chi^j_1 \, \chi^j_2 }{ \nu^j_{1}} .
	\end{equation*}
	If we assume that the  $U_j(t)$ have the following form  
	\begin{equation*}
		U_j(t)=U^\star_j e^{\chi^j_2  t} ,
	\end{equation*}
	then by substituting  in \eqref{9.2} we obtain 
	\begin{equation*}
		U^\star_j =\dfrac{ \nu^j_{2} I^\star_j  }{\eta  +\chi^j_2  }. 
	\end{equation*}
	
	The cumulative number of unreported cases $\CU_j(t)$ is computed as  
	\begin{equation*}
		\CU_j(t)' = \nu^j_2 I_j(t),
	\end{equation*}
	and we used the following initial condition: 
	\begin{equation*}
		\CU_j(0)=\CU^\star_j = \int_{-\infty}^{0}\nu^j_2 I^*_j e^{\chi^j_2 s}ds = \dfrac{\nu^j_2 I_j^\star}{\chi^j_2}.
	\end{equation*}
	We define the error between the data and the model as follows
	\begin{equation*}
		\left\{ 
		\begin{array}{c}
			I_1'(t)=\tau_{1} S_{1} \bigg[\phi_{11} \dfrac{ I_1(t)+U_1(t) }{N_1}+\ldots +\phi_{1n}\dfrac{ I_n(t)+U_n(t) }{N_n} \bigg]  - \nu I_1(t)+	\varepsilon_1(t) ,\\
			\vdots\\
			I_n'(t)=\tau_{n} S_{n} \bigg[ \phi_{n1} \dfrac{ I_1(t)+U_1(t) }{N_1} + \ldots +\phi_{nn}\dfrac{ I_n(t)+U_n(t) }{N_n} \bigg] -\nu I_n(t)+	\varepsilon_n(t),
		\end{array}
		\right.
	\end{equation*}
	or equivalently 
	%			\begin{equation*}
		%			\left\{ 
		%			\begin{array}{c}
			%				\varepsilon_1(t)=I_1'(t)-\tau_{1} S_{1} \bigg[\phi_{11} \dfrac{ I_1(t)+U_1(t) }{N_1}+\ldots +\phi_{1n}\dfrac{ I_n(t)+U_n(t) }{N_n} \bigg]  + \nu I_1(t),\\
			%				\vdots\\
			%				\varepsilon_n(t)=I_n'(t)-\tau_{n} S_{n} \bigg[ \phi_{n1} \dfrac{ I_1(t)+U_1(t) }{N_1} + \ldots +\phi_{nn}\dfrac{ I_n(t)+U_n(t) }{N_n} \bigg] +\nu I_n(t).
			%			\end{array}
		%			\right.
		%		\end{equation*}
	%		Then we obtain 
	\begin{equation*}
		%	\left\lbrace
		\left\{ 
		\begin{array}{c}
			\varepsilon_1(t)=\left( \chi^1_2 +  \nu\right) I^\star_1 e^{\chi^1_2 t} -\tau_{1} S_{1} \bigg[\phi_{11} \dfrac{ I^\star_1+U^\star_1}{N_1}e^{\chi^1_2 t} +\ldots +\phi_{1n} \dfrac{ I^\star_n+U^\star_n}{N_n} e^{\chi^n_2 t}  \bigg],  \\
			\vdots\\
			\varepsilon_n(t)=\left( \chi^n_2 +  \nu\right) I^\star_n e^{\chi^n_2 t} -\tau_{n} S_{n} \bigg[ \phi_{n1}  \dfrac{ I^\star_1+U^\star_1 }{N_1} e^{\chi^1_2 t} + \ldots 
			+\phi_{nn}\dfrac{ I^\star_n+U^\star_n}{N_n} e^{\chi^n_2 t} \bigg].
		\end{array}
		\right.
		%	\right.
	\end{equation*}
	
	\begin{lemma} Assume that the matrix $\phi$ be fixed.  If we consider the errors $\varepsilon^{\tau}_1(t), \ldots, \varepsilon^{\tau}_n(t) $ as a function of $\tau$, then we can a unique value $\tau^*=\left(\tau^\star_1, \ldots, \tau^\star_n \right)$ which minimizes that $L^2$ norm of the errors. That 		
		$$
		\sum_{j=1,\ldots,n}  \int_{d_1}^{d_2}   \varepsilon^{\tau^\star}_j(t) ^2 dt.= \min_{\tau \in \R^n}  \sum_{j=1,\ldots,n}  \int_{d_1}^{d_2}   \varepsilon^\tau_j(t) ^2 dt.
		$$
		Moreover, 
		\begin{equation*}
			\tau_j^\star=\dfrac{\int_{d_1}^{d_2} K_j(t) H_j(t) dt}{\int_{d_1}^{d_2} H_j(t)^2 dt},
		\end{equation*}
		with 
		$$
		K_j(t):=\left( \chi^j_2 +  \nu\right) I^\star_j e^{\chi^j_2 t}, \forall j=1,\ldots, n, 
		$$
		and 
		$$
		H_j(t):=S_{j} \bigg[ \phi_{j1}  \frac{ I^\star_1+U^\star_1 }{N_1} e^{\chi^1_2 t} + \ldots +\phi_{jn}\frac{ I^\star_n+U^\star_n}{N_n}e^{\chi^n_2 t} \bigg], \forall j=1,\ldots, n.
		$$
		
	\end{lemma}
	\begin{proof}
		We look for the vector $\tau=\left(\tau_1, \ldots, \tau_n \right)$ which minimizes of 
		$$
		\min_{\tau \in \R^n}  \sum_{j=1,\ldots,n}  \int_{d_1}^{d_2}   \varepsilon_j(t) ^2 dt.
		$$
		Define for each $j=1,\dots,n$ 
		$$
		K_j(t):=\left( \chi^j_2 +  \nu\right) I^\star_j e^{\chi^j_2 t} 
		$$
		and 
		$$
		H_j(t):=S_{j} \bigg[ \phi_{j1}  \frac{ I^\star_1+U^\star_1 }{N_1} e^{\chi^1_2 t} + \ldots 
		+\phi_{jn}\frac{ I^\star_n+U^\star_n}{N_n}e^{\chi^n_2 t} \bigg],
		$$
		so that 
		$$
		\varepsilon_j(t)=K_j(t)-\tau_j H_j(t).
		$$
		Hence for each $j=1,\dots,n$ 
		$$
		\int_{d_1}^{d_2} \varepsilon_j(t)^2 dt=\int_{d_1}^{d_2} K_j(t)^2 dt-2 \tau_j \int_{d_1}^{d_2} K_j(t) H_j(t) dt+\tau_j^2 \int_{d_1}^{d_2} H_j(t)^2 dt,
		$$
		and  the minimum of  $	\int_{d_1}^{d_2} \varepsilon_j(t)^2 dt$ is obtained for $ \tau_j$ satisfying 
		$$
		0=\dfrac{\partial }{\partial \tau_j}\int_{d_1}^{d_2} \varepsilon_j(t)^2 dt=-2 \int_{d_1}^{d_2} K_j(t) H_j(t) dt+2 \tau_j \int_{d_1}^{d_2} H_j(t)^2 dt
		$$
		whenever 
		\begin{equation*}
			\tau_j=\dfrac{\int_{d_1}^{d_2} K_j(t) H_j(t) dt}{\int_{d_1}^{d_2} H_j(t)^2 dt}.
		\end{equation*}
		Under this condition, we obtain 
		$$
		\int_{d_1}^{d_2} \varepsilon_j(t)^2 dt=\int_{d_1}^{d_2} K_j(t)^2 dt-\tau_j^2 \int_{d_1}^{d_2} H_j(t)^2 dt.
		$$
	\end{proof}

	\begin{remark} It does not seem possible to estimate the matrix of contact $\phi$ by using similar optimization method. Indeed, if we look for a matrix $\phi=\left(\phi_{ij} \right)$ which  minimizes 
		$$
		\min_{\phi \in M_n\left( \R \right)}  \sum_{j=1,\ldots,n}  \int_{d_1}^{d_2}   \varepsilon_j(t) ^2 dt,
		$$
		it turn out that 
		$$
		\sum_{j=1,\ldots,n}  \int_{d_1}^{d_2}   \varepsilon_j(t) ^2 dt=0
		$$ whenever $\phi$ is diagonal. Therefore the optimum is reached for any diagonal matrix. Moreover by using similar considerations, if several $\chi^2_j$ are equal, we can find a multiplicity of optima (possibly with $\phi$ not diagonal). This~means that trying to optimize by using the matrix $\phi$ does not yield significant and reliable information.   
	\end{remark}
	
	In the Figure below, we present an example of application of our method to fit the Japanese data. We use the period going from 20 March to 15 April.

	\begin{figure}[H]
		\centering
		\includegraphics[scale=0.8]{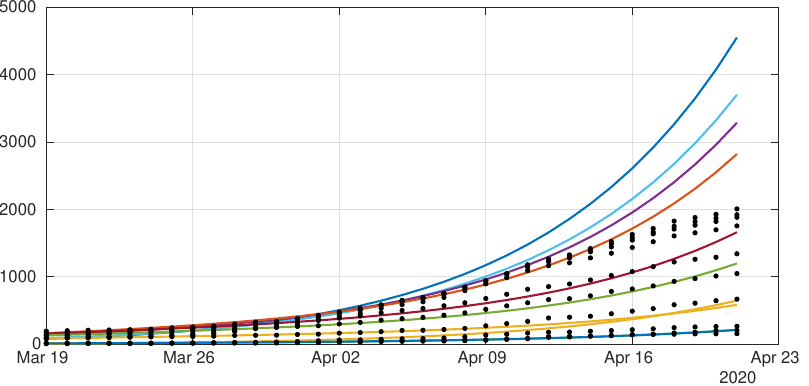}
		%	\includegraphics[scale=0.5]{GMS-Biology-2020/Figure8B.pdf}
		%\end{center}
		\caption{\textit{We plot a comparison between the model (without public intervention) and the age structured data from Japan (black dots). }}
		\label{Fig27}
	\end{figure}
	
	%\bigskip

	In the Figure below, we present an example of application of our method to fit the Japanese data. We use the period going from 20 March to 15 April.

	\begin{figure}[H]
		\centering
		\includegraphics[scale=0.8]{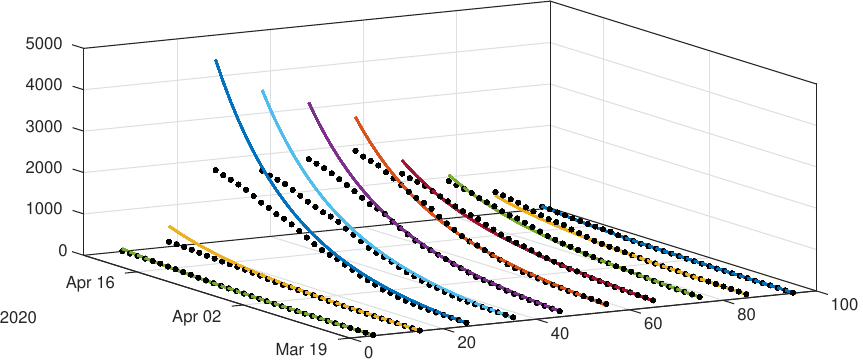}
		%\end{center}
		\caption{\textit{We plot a comparison between the model (without public intervention) and the age structured data from Japan  (black dots). }}
		\label{Fig28}
	\end{figure}

	\section{A survey for COVID-19 mathematical modeling}
	\label{Section10}

	During the COVID-19 pandemic, scientific workforces in different fields published COVID-19-related papers. The number of articles published increased considerably during this period. For example, on August 23, 2023, the WHO COVID-19 Research Database \cite{WHO1}  contains 724288 full texts of articles concerning the COVID-19 outbreak. Consequently, providing an extensive review on the subject is hopeless. Here, we make some arbitrary choices that can always be discussed. Our main goal is to give extra references on the topics mentioned earlier and highlight topics not considered in the previous sections. Several articles have attempted to do systematic reviews on COVID-19. We refer to \cite{Ioannidis21a, Rodriguez-Morales} for more results and a broader overview of the subject.

	\medskip 	
	The idea of this survey was mostly to collect references from the Infectious Disease Outbreak  webinar, which took place from 2020 to 2022 \cite{IDO}.

	\subsection{Medical survey}

	Mathematical models alone do not provide reliable information. In Figure \ref{Fig13}, we show the divergence of the mathematical model from the data. It is therefore fundamental to bring medical results into the models. 
	
	\medskip 
	It is therefore fundamental to integrate medical facts into mathematical models. We have tried throughout this text to explain how to make maximum use of the data either as input (test data) or as output (reported case data). But the dynamics of infection can be understood much better by examining concrete case studies in hospitals. For example, modeling the dynamics of infectious clusters is crucial in preventing the spread of disease. We refer to \cite{azevedo2023food, tille2023perspective, lu2023diseases, chen2022impact, Med1, Med2, Med3, Med4, Med5, Med6, Med7, Med8, Med9, Med10, Med11} for more results and references. 
	
	\medskip 
	The early development of an epidemic are very important, and an interesting retrospect of the first weeks of COVID-19  in China was presented by Zhao  in  \cite{zhao2022small}.

	\subsection{Incubation, Infectiousness, and Recovery Period}
	
	The infectious dynamic has three phases: 
	\begin{itemize}
		\item[\rm (i)]  The emission of the infectious agent, which depends on its concentration during its expulsion (remotely by air transportation or directly by secretion contact) from the contagious person; 
		\item[\rm (ii)] Transmission of the infectious agent (through an intermediate fluid or on a contact surface);

		\item[\rm (iii)] The reception of the infectious agent by a future host who becomes infected and whose symptomatology and secondary emission capacities will depend on the infectious agent's pathogenic nature and the host's immune defenses. 
		
	\end{itemize}
	These defenses are set up in two successive stages, corresponding to innate immunity, then to acquired immunity. It is, therefore, conceivable that the transmission capacity of an infectious person depends on the individual infection age. That is, the time since this person was infected.  We refer to \cite{jones2021estimating, he2020temporal, linton2020incubation, wu2022incubation, quesada2021incubation, rai2021incubation, nishiura2020serial, zuo2020airborne} for more results on the subject. In \cite{demongeot2023kermack}, we  proposed a method to understand the average individual dynamic of infection by clusters data. When considering epidemic exponential phase data, a time series approach is proposed in \cite{Demongeot22}. We refer to \cite{alvarez2021computing} for more results on the subject.

	\subsection{Data}
	An essential aspect of epidemics outbreaks is understanding the biases in the data. That is the different causes, such as unreported case data, tests, false positive PRC tests, and other factors that may bias our understanding of the data. Clusters of infected also provide another kind of data that may give another angle to examine the same problem. We should also mention the data provided by the wasted water that offers a helpful complement to the existing reported case data.  
	\subsubsection{Contact tracing}
	Contact tracing has been the main tool of public health authorities, for example, in South Korea when the COVID-19 pandemic started. In France, a dedicated digital tool called Stop-Covid has been developed. In \cite{rowe2020contact}, authors estimate that this digital approach was adopted not because digital solutions (to contact tracing) are superior to traditional ones but by default due to alienation and lack of interdisciplinary cooperation, which could be due to the fact that contact tracing is balancing personal privacy and public health, causing significant biases in classical inquiries with questionnaires \cite{kapa2020contact}. We refer to 	\cite{kretzschmar2020impact, browne2022differential, hart2021high,  giovanetti2021sars, zanella2021social, alvarez2021computing, rowe2020contact, bode2020contact, martinez2020digital, blasimme2021digital, jian2020contact, thayyil2020covid, cho2020contact, kapa2020contact} for more results on the subject.

	\subsubsection{Testing data}
	A mathematical model to understand the bias in PCR tests was proposed first by \cite{Peccoud-Jacob96, Peccoud-Jacob98}.  Diagnostic tests, particularly the PCR test, have been of considerable importance in most countries' follow-up of new cases. We refer \cite{Hellewell21, Kretzschmar22, Kucirka20, boger2021systematic, Salvatore23, arevalo2020false, pu2022screening}.  Mathematical models, including testing data as an input of the model, were proposed by   \cite{griette2021clarifying, bugalia2023assessing}.

	\subsubsection{Unreported and uncertainty in the number of reported case data}
	The origin of unreported cases of COVID-19 is multiple. It may be due to
	\begin{itemize}
		\item[\rm (i)]  a poor organization of the reporting system by the medical profession or recording by the administrative staff (especially at weekends); 
		\item[\rm (ii)] The presence of asymptomatic cases; 
		\item[\rm (iii)]  The non-consultation and/or the non-taking of medication in the symptomatic case, for reasons related to the patient or his entourage (presence of an intercurrent pathology or an existing chronic disease masking the symptoms, reasons financial, religious, philosophical, social, etc.).
	\end{itemize}
	We refer to \cite{aronna2022estimate, chow2020global, hortaccsu2021estimating, reis2020characterization, zhao2020estimating, Liu20a} for more results on the subject.

	\subsubsection{Clusters}
	The detection and monitoring of clusters are difficult to achieve and the discovery of patient zero, in a given geographical area, is always a delicate challenge. Nevertheless, there are a number of studies regarding this problem. We refer to \cite{oladipo2022performance, ito2020social, andrade2020covid, nazia2022methods, harris2022geospatial, gomes2020risk, adam2020clustering, chan2020familial} for more results on the subject. 
	
	%\begin{enumerate}
	%	
	%{\color{red}	\item Clusters of coronavirus disease in communities, Japan, January–April 2020
		%	Y Furuse, E Sando, N Tsuchiya, R Miyahara, I Yasuda, YK Ko, M Saito, ...
		%	Emerging infectious diseases 26 (9), 2176}
	%	
	%	\item     D.C. Adam, P. Wu, J. Y. Wong, E. H. Y. Lau, T. K. Tsang, S. Cauchemez, G. M. Leung \& B. J. Cowling, Clustering and superspreading potential of SARS-CoV-2 infections in Hong Kong, \textit{Nature Medicine}, \textbf{26(11)} (2020), 1714-1719.
	%	
	%	\item J. F-W. Chan, S. Yuan, K-H. Kok, K. K-W. To, H. Chu, J. Yang, F. Xing, J. Liu, C. C-Y. Yip, R. W-S. Poon, H-W. Tsoi, S. K-F Lo, K-H Chan, V. K-M. Poon, 
	%	W-M. Chan, J. D. Ip, J-P. Cai, V. C-C. Cheng, H. Chen, C. K-M. Hui, K-Y. Yuen, A familial cluster of pneumonia associated with the 2019 novel coronavirus indicating person-to-person transmission: a study of a family cluster, \textit{The Lancet}, \textbf{395(10223)} (2020), 514-523.
	%\end{enumerate} 
	
	\subsubsection{More phenomenological model to fit the data}
	
	Since Daniel Bernoulli's classic primordial model \cite{Bernoulli-Chapelle, Bernoulli1766, Blower04, Dietz-Heesterbeek}, a number of phenomenological models have emerged, such as that of Richards that Ma cited \cite{ma2020estimating} just before the beginning of Covid-19 outbreak. The COVID-19 pandemic was an opportunity to recall this princeps work and to propose new approaches along the same lines, namely minimal modeling integrating the basic mechanisms of infectious transmission. We refer to \cite{tat2020epidemic, ma2020estimating, miyama2022phenomenological, attanayake2020phenomenological, zuhairoh2022data, calatayud2023phenomenological, smith2021performance, griette2021robust, richards1959flexible} for more results and references on the subject.

	\subsubsection{Wasted water data}
	The French national Obepine project has shown the value of monitoring the COVID-19 pandemic in wastewater, where the concentration of viral RNA fragments can serve as an early indicator of the onset of new waves of cases. An Italian study (Gragnani et al.) has even suggested that SARS-Cov-2 RNA was present in wastewater from Milan, Turin (December 18, 2019) and Bologna (January 29, 2020) long before the first Italian case was described (February 20 2020). We refer to \cite{sari2021infectious, elsaid2021effects, ai2021wastewater, bertrand2021epidemiological, gragnani2021sars, wurtzer2022sars, wurtzer2021several, wurtzer2020evaluation} for more results on the subject. 
	
	\subsubsection{Discrete and random modeling}
	Some modeling approaches are discrete and play with daily data. The equations of the contagion dynamics can be of two types:
	\begin{itemize}
		\item[\rm (i)]  They can be difference equations modeled on the differential equations of the continuous SIR model; 
		\item[\rm (ii)] or they can be stochastic in nature, with generally additive Gaussian noise in the second member. 
	\end{itemize}
	They generally lend themselves well to the statistical estimation of their parameters from the data. We refer to \cite{zhou2022modeling, xue2020data, forien2020estimating, bacallado2020generation} for more results and references on the subject.

	\subsubsection{Time series and wavelet approaches}
	
	If we consider the data recorded on the size of the different sub-populations involved in the contagion process (susceptible, infected, cured, immune, etc.), a possible approach is that of the signal theory, 
	with its classical methods data processing (time series, Fourier transformation, wavelet transformation, etc.). 
	This approach is generally an excellent introduction to the implementation of prediction methods. We refer to \cite{tat2020epidemic, demongeot2021application, soubeyrand2020towards, oshinubi2022functional, benhamou2022phenotypic, Demongeot22} for more results and references on the subject.

	\subsubsection{Transmission estimation and spatial modeling}
	
	Estimating the transmission parameter and studying its spatio-temporal variations is fundamental because it conditions the epidemic waves' location, shape, and duration. The spatial heterogeneity of this parameter, often due to geo-climatic (such as temperature) and/or demographic (such as susceptible population density), are crucial factors in the existence of natural barriers to the spread of a pandemic.  We refer to \cite{zhao2020serial, gaudart2021factors, mizumoto2020transmission} for more results and references on the subject.

	\subsubsection{Forecasting methods}
	
	The prediction of epidemics is one of the major objectives of modeling. It can be carried out by the continuation, in time and space, of the solutions of the spatio-temporal equations of the chosen model or the extrapolation of a statistical description of the evolution of the observed variables. We refer to \cite{morel2023learning,   bakhta2020epidemiological,roosa2020real, Liu21} for more results and references on the subject.

	\subsection{SIR like models}
	Since 2020, many articles have appeared on using the SIR model in modeling the COVID-19 outbreak. These models progressively complexified to become SIAURDV models, incorporating explicitly as ODE variables the numbers of asymptomatic (A), non-reported (U), vaccinated (V), and deceased (D) patients.  We refer to \cite{parolini2021suihter, zhao2022prediction, saiprasad2022analysis, thomas2020primer, alamo2021data, lin2020conceptual} for more results and references on the subject.

	\subsubsection{Multigroups or multiscale models}
	
	The notion of multi-group and multi-scale appeared when the COVID-19 outbreak appeared, with specific dynamics in several geographical regions of different scales and, in one area, in several distinct groups (demographic, ethnic, economic, religious, social, etc.). We refer to \cite{zanella2021data,  prague2020multi, meng2021generalized, volpert2021epidemic, griffith2021interrogating, melo2023final, reingruber2023data} for more results and references on the subject.

	\subsubsection{Model with unreported  or asymptomatic compartment}
	
	Modeling the mechanisms of non-reporting of new cases or deaths due to an epidemic makes it possible to compensate for the bias coming from a partial observation of the infected, due to the existence of asymptomatic cases or a deficient administrative registration mechanism.   We refer to \cite{olumoyin2021data, zhang2022usage, anggriani2022mathematical,  anguelov2020big, batistela2021sirsi, chen2021ratio, aguiar2022role} for more results and references on the subject.

	\subsection{Connecting reported case data with SIR like model}
	
	Very few studies considered that problem in the literature, while again, it is interesting to understand the bias induced by such a mechanism. For example, it would make sense to consider a model including a delay in reporting the data
	$$
	\CR'(t)= f \nu \int_0^\tau \gamma(s )I(t-s)ds 
	$$
	where $s\mapsto \gamma(s)$ is a non negative map. The quantity $\gamma(s)$ is the probability of reporting $s$ units of time after the individual leaves the compartment $I$.  
	This corresponds to patients showing symptoms.  We deduce that we must have  
	$$
	\int_0^\tau \gamma(s )ds=1.  
	$$
	Unfortunately, people have not considered this issue in the literature. The consequence of such a model for reported case data seems particularly important.  We refer to \cite{  postnikov2020estimation, biswas2020covid, mehmood2021investigating} for more results and references on the subject.

	\subsection{Re-infections, natural and hybrid immunity}
	
	The risk of reinfection with the SARS-Cov2 virus comes from two factors: 
	\begin{itemize}
		\item[\rm (i)] One is due to the infectious agent and its mutagenic genius, modifying its contagiousness and pathogenicity; 
		\item[\rm (ii)]  The other is due to the host, whose natural, innate, and acquired defenses by the adaptive immune system or artificially by vaccination prevent or stop the infection. 
	\end{itemize}
	The modeling of these two facets of the reinfection process makes it possible to understand the mechanisms of eradication or, on the contrary, the continuation of a pandemic, thanks to or despite collective public health measures. We refer to \cite{ smith2022time, jenner2021covid, nordstrom2022risk, huang2022comparing, pilz2022sars, stein2023past} for more results and references on the subject.

	\subsection{Mortality} 
	Mortality may appear as more robust data to be connected with epidemic models. The bias for report cases data will also exist for the number of reported dead patients. Again, the model to connect the data and the epidemic model might be more complex than a fraction of the recovered.  Nevertheless, there is evidence of an increased risk of death in the event of co-infection. The mortality risk increases dramatically when a patient is infected with another severe disease. This question of co-infection with severe diseases with COVID-19 was studied in \cite{Ruan20}. 
	We refer to \cite{soubeyrand2020covid, kobayashi2020communicating, jung2020real, ioannidis2020global, ioannidis2021infection, ioannidis2021reconciling, ioannidis2023really, mehraeen2020predictors} for more results and references on the subject. 
	
	%\begin{enumerate}
	%	
	% 
	%	
	%	 
	%	
	%	 
	%		
	%%	\item Ioannidis, J. P. (2021). Infection fatality rate of COVID-19 inferred from seroprevalence data. Bulletin of the world health organization, 99(1), 19.
	%	
	%%		\item Ioannidis, J. P. (2021). Reconciling estimates of global spread and infection fatality rates of COVID‐19: an overview of systematic evaluations. European journal of clinical investigation, 51(5), e13554.
	%		
	%	\item 	Muka, T., Li, J. J., Farahani, S. J., \& Ioannidis, J. P. (2023). An umbrella review of systematic reviews on the impact of the COVID-19 pandemic on cancer prevention and management, and patient needs. Elife, 12.
	%	
	%	
	%\item 	Ioannidis, J. P., Zonta, F., \& Levitt, M. (2023). Estimates of COVID‐19 deaths in Mainland China after abandoning zero COVID policy. European Journal of Clinical Investigation, 53(4), e13956.
	%	
	%	
	%\item 	Ioannidis, J. P., Zonta, F., \& Levitt, M. (2023). What Really Happened During the Massive SARS-CoV-2 Omicron Wave in China?. JAMA Internal Medicine.
	%	
	%	
	%\item 	Pezzullo, A. M., Axfors, C., Contopoulos-Ioannidis, D. G., Apostolatos, A., \& Ioannidis, J. P. (2023). Age-stratified infection fatality rate of COVID-19 in the non-elderly population. Environmental Research, 216, 114655.
	%\end{enumerate}

	\subsection{Vaccination and mitigation measures}
	
	Vaccination and exclusion by temporary confinement or physical barriers (masks, anti-viral protection, or anti-transmission intermediates) are the public health measures intended to mitigate or stop an epidemic. The modeling of their gradual introduction and their effects on the spread of the epidemic makes it possible to understand their effectiveness or, on the contrary, their uselessness and, therefore, to adapt the coercive measures best, whether collective or individual \cite{ demongeot2022modeling, jarumaneeroj2022epidemiology, marinov2022adaptive, ioannidis2021covid, ioannidis2022estimating, he2022evaluation, hale2021government, acuna2020modeling, wu2020quantifying, rozhnova2021model, teslya2020impact,liu2022optimizing,coccia2022optimal, halim2021covid}.

	\subsection{Chronological age}
	
	The problem of age structure  is crucial in epidemic modeling for three reasons:  
	\begin{itemize}
		\item[{\rm (i)}] The immune system efficacy depends on age. Therefore, its adaptive component is less and less able to resist a new pathogenic agent or react to a vaccine;
		
		\item[{\rm (ii)}]  Age groups communicate differently with each other, with the most mobile (working age group) having the greatest chance of transmission and the most dependent (elderlies) on the care by younger caregivers having the greatest chance of being infected; 
		
		\item[{\rm (iii)}]  The prevalence of chronic diseases favoring infections is very unevenly distributed, the age groups at both ends of life being the most susceptible: the young due to the immaturity of the immune system and school promiscuity, and the elderly due to the existence of chronic comorbidities (diabetes, respiratory pathologies, cardiovascular diseases, and immune depression). 
	\end{itemize}
	
	These disparities make it necessary to take age into account (through at least three major classes, young people under 20, adults from 20 to 65, and seniors over 65), preventive measures (education, vaccination, isolation) being taken according to this age stratification, crossed with the risk factors linked to the occurrence of chronic pathologies.

	Few papers combined epidemic model with  age-structure and age structured data  \cite{Kretzschmar22, Age2, Age3, Age4, Age5, Age6, Age7, Age8, kuo2020covid, polidori2021covid, cao2022accelerated}. The problem of understanding the relationships between data and models is far from well understood. In Section \ref{Section9}, based on \cite{Griette20}, we proposed an approach to understanding how to connect the model and the data during the exponential phase. But such a problem needs further investigation.

	\subsection{Basic reproduction number}
	The basic reproduction number $R_0$ is an essential parameter for predicting the occurrence of an epidemic wave. It can vary over time and depends on two main factors: 
	\begin{itemize}
		\item[\rm (i)]   In the infectious subject, the successive establishment of natural defense mechanisms (innate and adaptive) explains the variations in daily $R_0$ during his period of contagiousness; 
		\item[\rm (ii)]   In subjects who are not yet infected, their susceptibility is also dependent on their immune status, but also on the collective public health measures taken at the population level. 
	\end{itemize}

	Methods for estimating daily $R_0$ are therefore fundamental to understanding the temporal and spatial evolution of a pandemic \cite{alimohamadi2020estimate, he2020estimation, billah2020reproductive, zhao2020preliminary, white2021statistical, demongeot2021estimation}.

	\subsection{Prediction of COVID-19 evolution}
	
	The difficulty of predicting the evolution of a pandemic is due to the adaptive capacities of the infectious agent and the infected and transmitting host. On the one hand, the genetic mutations of the infectious agent  and its contagious power and pathogenic dangerousness develop a highly infectious and low pathogenic variant, often signaling the natural end of a pandemic. On the other hand, the permanent adaptation strategy of individual and collective host defense measures makes it possible to anticipate the effects of changes in the agent's infectious strategy. In both cases, modeling the dynamics of mutation and prevention is essential to predict and act in near real-time on the evolution of a pandemic \cite{otto2021origins,  yang2022covid, roda2020difficult, miller2022forecasting, hussein2022short, hatami2022simulating, rashed2022covid, du2023incorporating, ioannidis2022forecasting}.

	\newpage 
	
	\vspace{1cm}
	%%%%%%%%%%%%%%%%%%%%%%%%%%%%%%%%%%%%%%%%%%
	%% Optional
 % Leave argument "no" if all appendix headings stay EMPTY (then no dot is printed after "Appendix A"). If the appendix sections contain a heading then change the argument to "yes".

	\appendix
	
	\begin{center}
		{\LARGE	\textbf{Appendix}}
	\end{center}
	\section{When the output is a single exponential function}
	\label{AppendixA}

	Let $X \in \R^n$. We recall that 
	\begin{itemize}
		\item $X\geq 0$ if for each $i \in \left\{1, \ldots,  n\right\}$ such that  $X_i  \geq 0$;
		\item $X> 0$ if $X\geq 0$ and there exists $i \in \left\{1, \ldots,  n\right\}$ such that  $X_i  > 0$;
		\item  $X\gg 0$ if $X_i >0$ for each $i \in \left\{1, \ldots,  n\right\}$. 
	\end{itemize} 
	Let $A=(a_{ij})\in M_n \left(\R \right)$ a $n \times n$ matrix with non-negative  off-diagonal elements, and assume that  $A+ \delta I$ is non-negative irreducible whenever $\delta >0$ is large enough. 
	The projector associated to the Perron-Frobenius dominant eigenvalue is defined by  
	\begin{equation} \label{A1} 
		\Pi \, x= \dfrac{\left\langle V_L(A), x\right\rangle V_R(A)}{\left\langle V_L(A),V_R(A)\right\rangle }, \forall x \in \R^n, 	 
	\end{equation}
	where $V_R(A)\gg 0$ (respectively $V_L(A)\gg 0$) is a right eigenvector (resp. left eigenvector ) of $A$ associated with the dominant eigenvalue 
	$$
	s(A)=\max \left\{ \Re \lambda : \lambda \in \sigma(A)\right\},
	$$ 
	where $\sigma(A)$ is the spectrum of $A$ (i.e. the set of all eigenvalues of $A$). Then we have 
	$$
	A \, \Pi= \Pi \, A=s(A) \, \Pi. 
	$$
	Recall that the euclidean inner product is defined by 
	$$
	\langle	X, Y\rangle= \sum_{i=1}^{n} X_i Y_i. 
	$$

	The network associated with a non-negative matrix $A$ corresponds to all the oriented paths from the node $i$ to the node $j$ whenever $a_{ij}>0$. 
	
	\medskip 
	A non-negative matrix $A$ is \textit{irreducible} if the network associated with $A$ is strongly connected. That is, if we can join any two nodes $i$ and $j$ by using a succession of oriented paths.  
	
	%		\medskip 
	%For example, in epidemiology, a matrix of contacts is irreducible if any individual in a given sub-group can contact any other subgroup. If needed, the contact is indirect and occurs through a succession of contacts between sub-groups. 
	
	\medskip 
	To understand irreducible matrices in epidemics, one may consider the contact matrix in epidemic models. Then,  the contact matrix is irreducible if any infected sub-group has a non-zero probability of infecting any other group (by transmitting the pathogen to intermediate sub-groups if needed). 
	
	\begin{theorem} \label{THA1} Let $A=(a_{ij})\in M_n \left(\R \right)$, and assume that the off-diagonal elements of $A $ are non-negative, and  $A+ \delta I$ is non-negative irreducible whenever $\delta >0$ is large enough.   We assume that  there exists a vector $X_0 >0$ such that 
		\begin{equation} \label{A2}
			X'(t)=AX(t), \forall t \in \left[0, \tau \right], \text{ with } X(0)=X_0,	
		\end{equation}
		and there exists a vector $Y >0$ satisfying 
		\begin{equation} \label{A3}
			\sum_{i=1}^{n} Y_{i} X_i(t)=\chi_1 e^{\chi_2 t },\forall t \in \left[0, \tau \right],
		\end{equation}
		with $\chi_1>0$, $\chi_2>0$, and $\tau>0$. 
		
		\medskip 
		\noindent Then we have 
		\begin{equation*} 
			\chi_2= s(A), \text{ and } \chi_1 =\langle	Y, \Pi X_0 \rangle. 
		\end{equation*} 
		That is,
		\begin{equation*} 
			\sum_{i=1}^{n} Y_{i} X_i(t)=	\langle	Y, \Pi X_0 \rangle  e^{s(A) t },\forall t \geq 0.
		\end{equation*}
		In other words, we can not distinguish the growth induced by $	\langle	Y,  X_0 \rangle $ and $	\langle	Y, \Pi X_0 \rangle $. Therefore we can replace $X_0$ with $\Pi X_0$, and the output $	\langle	Y,  X(t) \rangle $ will be the same. 
		
	\end{theorem}

	\begin{proof}
		The equation \eqref{A3} is equivalent to 
		\begin{equation*}	
			\langle	Y, e^{A t}X_0 \rangle =\chi_1 e^{ \chi_2 t },\forall t \in \left[0, \tau \right],
		\end{equation*}
		For each $\delta >0$ large enough such that $A+ \delta I$ is non-negative and primitive, we have 
		\begin{equation*}	
			\langle	Y, e^{\left(A+ \delta I \right)t}X_0 \rangle =\chi_1 e^{ \left(\chi_2 +\delta \right)t },\forall t \in \left[0, \tau \right],
		\end{equation*}
		so by computing the derivatives on both sides of  the above equation and taking $t=0$, we obtain 
		\begin{equation*} 
			\langle	Y, \left(A+ \delta I \right)^{m}X_0 \rangle =\chi_1 \left(\chi_2 +\delta \right)^{m}, \forall m \in \N. 
		\end{equation*}
		But we have $r\left(A+ \delta I \right)=s(A)+\delta$, and 
		\begin{equation*} 
			\langle	Y, \dfrac{\left(A+ \delta I \right)^{m}}{r\left(A+ \delta I \right)^{m}} X_0 \rangle =\chi_1  \dfrac{ \left(\chi_2 +\delta \right)^{m}}{ \left( s(A)+\delta\right)^{m}}, \forall m \in \N, 
		\end{equation*}
		and since the right-hand side of the above equality converges to $	\langle	Y, \Pi X_0 \rangle >0$ (where $\Pi \gg 0$ is the projector defined in \eqref{9.1}), we deduce that 
		$$
		\lim_{m \to \infty}  \dfrac{ \left(\chi_2 +\delta \right)^{m}}{ \left( s(A)+\delta\right)^{m}}= \dfrac{	\langle	Y, \Pi X_0 \rangle }{\chi_1}>0,
		$$
		and the result follows.  
	\end{proof}

	\bibliographystyle{siam}
%	\reftitle{References}
	
	\bibliography{biblio_review_COVID_Augsut_28_2023.bib}

	% If authors have biography, please use the format below
	%\section*{Short Biography of Authors}
	%\bio
	%{\raisebox{-0.35cm}{\includegraphics[width=3.5cm,height=5.3cm,clip,keepaspectratio]{Definitions/author1.pdf}}}
	%{\textbf{Firstname Lastname} Biography of first author}
	%
	%\bio
	%{\raisebox{-0.35cm}{\includegraphics[width=3.5cm,height=5.3cm,clip,keepaspectratio]{Definitions/author2.jpg}}}
	%{\textbf{Firstname Lastname} Biography of second author}
	
	% For the MDPI journals use author-date citation, please follow the formatting guidelines on http://www.mdpi.com/authors/references
	% To cite two works by the same author: \citeauthor{ref-journal-1a} (\citeyear{ref-journal-1a}, \citeyear{ref-journal-1b}). This produces: Whittaker (1967, 1975)
	% To cite two works by the same author with specific pages: \citeauthor{ref-journal-3a} (\citeyear{ref-journal-3a}, p. 328; \citeyear{ref-journal-3b}, p.475). This produces: Wong (1999, p. 328; 2000, p. 475)
	
	%%%%%%%%%%%%%%%%%%%%%%%%%%%%%%%%%%%%%%%%%%
	%% for journal Sci
	%\reviewreports{\\
		%Reviewer 1 comments and authors’ response\\
		%Reviewer 2 comments and authors’ response\\
		%Reviewer 3 comments and authors’ response
		%}
	%%%%%%%%%%%%%%%%%%%%%%%%%%%%%%%%%%%%%%%%%%
	%\PublishersNote{}
	%\end{adjustwidth}
\end{document}